\documentclass[aos]{imsart}

\usepackage{amsthm,amsmath}

\usepackage{amssymb}
\usepackage{mathrsfs, color}
\usepackage{subfigure}
\usepackage{graphicx}
\usepackage{graphics}
\usepackage{multirow}
\usepackage{hyperref}
\usepackage{url}
\usepackage{enumitem}
\usepackage{ulem}

\startlocaldefs
\newcommand{\va}[0]{\mathbf a}
\newcommand{\vb}[0]{\mathbf b}

\newcommand{\vg}[0]{\mathbf g}

\newcommand{\vt}[0]{\mathbf t}

\newcommand{\vw}[0]{\mathbf w}
\newcommand{\vx}[0]{\mathbf x}

\newcommand{\vX}[0]{\mathbf X}
\newcommand{\vY}[0]{\mathbf Y}
\newcommand{\vZ}[0]{\mathbf Z}

\newcommand{\vzero}[0]{\mathbf 0}
\newcommand{\vone}[0]{\mathbf 1}
\newcommand{\mbf}[1]{\mbox{\boldmath$#1$}}
\newcommand{\Cov}[0]{\text{Cov}}
\newcommand{\Var}[0]{\text{Var}}

\newcommand{\tr}[0]{\text{tr}}

\newcommand{\E}[0]{\mathbb{E}}
\newcommand{\Id}[0]{\text{Id}}

\newcommand{\hiota}[0]{h_{(\iota)}}
\newcommand{\fiota}[0]{f_{(\iota)}}
\newcommand{\vech}[0]{\text{vech}}
\newcommand{\MAX}[0]{\text{MAX}}

\newcommand{\Prob}[0]{\mathbb{P}}

\newcommand{\vertiii}[1]{{\left\vert\kern-0.25ex\left\vert\kern-0.25ex\left\vert #1 
    \right\vert\kern-0.25ex\right\vert\kern-0.25ex\right\vert}}

\theoremstyle{plain}
\newtheorem{thm}{Theorem}[section]
\newtheorem{lem}[thm]{Lemma}

\newtheorem{cor}[thm]{Corollary}

\theoremstyle{definition}
\newtheorem{defn}{Definition}[section]

\newtheorem{exmp}{Example}[section]

\theoremstyle{remark}
\newtheorem{rmk}{Remark}

\endlocaldefs

\begin{document}

\begin{frontmatter}

\title{Gaussian approximation for the sup-norm of high-dimensional matrix-variate U-statistics and its applications\thanksref{T1}}
\runtitle{Gaussian approximation for high-dimensional U-statistics}
\thankstext{T1}{First version: January 30, 2016. This version: \today.}


\begin{aug}
\author{\fnms{Xiaohui} \snm{Chen}\thanksref{m1}\ead[label=e1]{xhchen@illinois.edu}\thanksref{t2}}
\runauthor{Chen}
\thankstext{t2}{Research partially supported by NSF grant DMS-1404891 and UIUC Research Board Award RB15004.} 

\affiliation{University of Illinois at Urbana-Champaign\thanksmark{m1}}

\address{Xiaohui Chen\\
Department of Statistics \\
725 S. Wright Street \\
Champaign, IL 61820 \\
\printead{e1}}
\end{aug}
%


\begin{abstract}
This paper studies the Gaussian approximation of high-dimensional and non-degenerate U-statistics of order two under the supremum norm. We propose a two-step Gaussian approximation procedure that does not impose structural assumptions on the data distribution.  Specifically, subject to mild moment conditions on the kernel, we establish the explicit rate of convergence that decays polynomially in sample size for a high-dimensional scaling limit, where the dimension can be much larger than the sample size. We also supplement a practical Gaussian wild bootstrap method to approximate the quantiles of the maxima of centered U-statistics and prove its asymptotic validity. The wild bootstrap is demonstrated on statistical applications for high-dimensional non-Gaussian data including: (i) principled and data-dependent tuning parameter selection for regularized estimation of the covariance matrix and its related functionals; (ii) simultaneous inference for the covariance and rank correlation matrices. In particular, for the thresholded covariance matrix estimator with the bootstrap selected tuning parameter, we show that the Gaussian-like convergence rates can be achieved for heavy-tailed data, which are less conservative than those obtained by the Bonferroni technique that ignores the dependency in the underlying data distribution. In addition, we also show that even for subgaussian distributions, error bounds of the bootstrapped thresholded covariance matrix estimator can be much tighter than those of the minimax estimator with a universal threshold.
\end{abstract}

\end{frontmatter}

\section{Introduction}

Let $\vX_1,\cdots,\vX_n$ be a sample of independent and identically distributed (iid) random vectors in $\mathbb{R}^p$ with the distribution function $F$. Let $B$ be a separable Banach space equipped with the norm $\|\cdot\|$ and $h : \mathbb{R}^p \times \mathbb{R}^p \to B$ be a $B$-valued measurable and symmetric kernel function such that $h(\vx_1,\vx_2)=h(\vx_2,\vx_1)$ for all $\vx_1,\vx_2 \in \mathbb{R}^p$ and $\E\|h(\vX_1,\vX_2)\| < \infty$. Consider the U-statistics of order two
\begin{equation}
\label{eqn:ustat-order2}
U = {n \choose 2}^{-1} \sum_{1 \le i < j \le n} h(\vX_i, \vX_j).
\end{equation}
The main focus of this paper is to study the asymptotic behavior of the random variable $\|U-\E U\|$ in the high-dimensional setup when $p:=p(n) \to \infty$. Since the introduction of U-statistics by Hoeffding \cite{hoeffding1948}, their limit theorems have been extensively studied in the classical asymptotic setup where $n$ diverges and $p$ is fixed \cite{hoeffding1963,gregory1977,serfling1980,arconesgine1993,zhang1999,ginelatalazinn2000,houdrepatricia2003,hsingwu2004}. Recently, due to the explosive data enrichment, regularized estimation and dimension reduction of high-dimensional data have attracted a lot of research attentions such as covariance matrix estimation \cite{bickellevina2008a,bickellevina2008b,karoui2008,chenxuwu2013a}, graphical models \cite{dempster1972,yuanlin2007,buhlmannvandegeer2011}, discriminant analysis \cite{maizouyuan2012a}, factor models \cite{fanliaomincheva2011,MR2933663}, among many others. Those problems all involve the consistent estimation of an expectation of U-statistics of order two $\E[h(\vX,\vX')]$, where $\vX$ and $\vX'$ are two independent random vectors in $\mathbb{R}^p$ with the distribution $F$. Below are two examples.

\begin{exmp}
\label{exmp:sample-covmat}
The sample covariance matrix $\hat{S} = (n-1)^{-1} \sum_{i=1}^n (\vX_i-\bar{\vX}) (\vX_i-\bar{\vX})^\top$, where $\bar{\vX}=n^{-1}\sum_{i=1}^n \vX_i$ is the sample mean vector, is an unbiased estimator of the covariance matrix $\Sigma = \Cov(\vX)$. Then $\hat{S}$ is a matrix-valued U-statistic of form (\ref{eqn:ustat-order2}) with the unbounded kernel
\begin{equation}
\label{eqn:sample-covmat-kernel}
h(\vx_1,\vx_2) = {1\over2} (\vx_1-\vx_2) (\vx_1-\vx_2)^\top \qquad \text{ for } \vx_1,\vx_2 \in \mathbb{R}^p.
\end{equation}
\end{exmp}

\begin{exmp}
\label{exmp:kendalltau}
The covariance matrix $\Sigma$ quantifies the linear dependency in $\vX = (X_1,\cdots,X_p)^\top$. The rank correlation is another measure for the nonlinear dependency in a random vector. For $m,k=1,\cdots,p$, $(X_m, X_k)$ and $(X'_m, X'_k)$ are said to be {\it concordant} if $(X_m-X'_m) (X_k-X'_k) > 0$. Let \begin{equation}
\label{eqn:kendaltau-kernel}
h_{mk}(\vx_1,\vx_2) = 2 \cdot \vone\{ (\vx_{1m}-\vx_{2m}) (\vx_{1k}-\vx_{2k}) > 0\}
\end{equation}
and $h(\vx_1,\vx_2) = \{h_{mk}(\vx_1,\vx_2) \}_{m,k=1}^p$. Kendall's tau rank correlation coefficient matrix $T = \{\tau_{mk}\}_{m,k=1}^p$
can be written as a U-statistic with the bounded kernel $h_{mk}$ in (\ref{eqn:kendaltau-kernel})
$$
\tau_{mk} = {n \choose 2}^{-1} \sum_{1 \le i<j \le n} h_{mk}(\vX_i,\vX_j) - 1.
$$
Then, $(\tau_{mk}+1)/2$ is an unbiased estimator of $\Prob((X_m-X'_m)(X_k-X'_k)>0)$, i.e. the probability that $(X_m, X_k)$ and $(X'_m, X'_k)$ are concordant.
\end{exmp}

In this paper, we are interested in the following central questions: {\it how does the dimension impact the asymptotic behavior of U-statistics and how can we make statistical inference when $p \to \infty$?} Motivation of this paper comes from the estimation and inference problems for large covariance matrix and its related functionals \cite{meinshausenbuhlmann2006,yuanlin2007,rothmanbickellevinazhu2008a,pengwangzhouzhu2009a,yuan2010a,cailiuluo2011a,bickellevina2008b,chenxuwu2013a,chenxuwu2015+}. To establish rate of convergence for the regularized estimators or to study the $\ell^\infty$-norm Gaussian approximations in high-dimensions, a key issue is to characterize the supremum norm of $U-\E U$. Therefore, as the primary concern of the current paper, we shall consider $B=\mathbb{R}^{p\times p}$ and $\|h\| = \max_{1\le m,k \le p} |h_{mk}|$. 

Our first main contribution is to provide a Gaussian approximation scheme for the high-dimensional {\it non-degenerate} U-statistics under the sup-norm. Different from the central limit theorem (CLT) type results for the maxima of sums of iid random vectors \cite{cck2013}, which are directly approximated by the Gaussian counterparts with the matching first and second moments, approximating the sup-norm of U-statistics is more subtle because of its dependence and nonlinearity. Here, we propose a {\it two-step} Gaussian approximation method in Section \ref{sec:gaussian-approx}. In the first step, we approximate the U-statistics by the leading component of a linear form in the Hoeffding decomposition (a.k.a. the H\'ajek projection); in the second, the linear term is further approximated by the Gaussian random vectors. To approximate the distribution of the sup-norm of U-statistics by a linear form, a maximal moment inequality is developed to control the nonlinear and {\it canonical}, i.e. {\it completely degenerate}, form of the reminder term. Then the linear projection is handled by the recent development of Gaussian approximation in high-dimensions \cite{cck2013,zhangcheng2014,zhangwu2015a}. Explicit rate of convergence of the Gaussian approximation for high-dimensional U-statistics is established for unbounded kernels subject to sub-exponential and uniform polynomial moment conditions. Specifically, under either moment conditions, we show that the same convergence rate that decays polynomially in sample size as in the Gaussian approximation for the maxima of sums of iid random vectors is attained and the validity of the Gaussian approximation is proved for a high-dimensional scaling limit, where $p$ can be much larger than $n$.

The second contribution of this paper is to propose a Gaussian wild bootstrap procedure for approximating the quantiles of $\sqrt{n} \|U-\E U\|$. Since the (unobserved) linear projection terms of the centered U-statistics depend on the unknown underlying data distribution $F$ and there is a nonlinear remainder term, we use an additional estimation step beyond the Gaussian approximation. Here, we employ the idea of decoupling and estimate the linear projection on an independent dataset. Validity of the Gaussian wild bootstrap is established under the same set of assumptions in the Gaussian approximation results. One important feature of the Gaussian approximation and the bootstrap procedure is that no structural assumptions on the distribution $F$ are required and the strong dependence in $F$ is allowed, which in fact helps the Gaussian and bootstrap approximation. In Section \ref{sec:stat_apps}, we demonstrate the capability of the proposed bootstrap method applied to a number of important high-dimensional problems, including the data-dependent tuning parameter selection in the thresholded covariance matrix estimator and the simultaneous inference of the covariance and Kendall's tau rank correlation matrices. Two additional applications for the estimation problems of the sparse precision matrix and the sparse linear functionals are given in the Supplemental Materials (SM). In those problems, we show that the Gaussian like convergence rates can be achieved for non-Gaussian data with heavy-tails. For the sparse covariance matrix estimation problem, we also show that the thresholded estimator with the tuning parameter selected by the bootstrap procedure can gain potentially much tighter performance bounds over the minimax estimator with a universal threshold that ignores the dependency in $F$ \cite{bickellevina2008b,chenxuwu2013a,caizhou2011a}.

To establish the Gaussian approximation result and the validity of the bootstrap method, we have to bound the the expected sup-norm of the second-order canonical term in the Hoeffding decomposition of the U-statistics and establish its non-asymptotic maximal moment inequalities. An alternative simple data splitting approach by reducing the U-statistics to sums of iid random matrices can give the exact rate for bounding the moments in the non-degenerate case \cite{talagrand1996,massart2000,kleinrio2005,einmahlli2008}. Nonetheless, the reduction to the iid summands in terms of data splitting does not exploit the complete degeneracy structure of the canonical term and it does not lead to the convergence result in the Gaussian approximation for the non-degenerate U-statistics; see Section \ref{app:concentration_ineq_canonical-Ustat} for details. In addition, unlike the Hoeffding decomposition approach, the data splitting approximation is not asymptotically tight in distribution and therefore it is less useful in making inference of the high-dimensional U-statistics. \\

{\bf Notations and definitions.} For a vector $\vx$, we use $|\vx|_1 = \sum_j |x_j|$, $|\vx|:=|\vx|_2 = (\sum_j x_j^2)^{1/2}$, and $|\vx|_\infty = \max_j |x_j|$ to denote its entry-wise $\ell^1$, $\ell^2$, and $\ell^\infty$ norms, respectively. For a matrix $M$, we use $|M|_F=(\sum_{i,j} M_{ij}^2)^{1/2}$ and $\|M\|_2 =\max_{|\va|=1} |M\va|$ to denote its Frobenius and spectral norms, respectively. Denote $a \vee b = \max(a,b)$ and $a \wedge b = \min(a,b)$. We shall use $K, K_0, K_1,\cdots$ to denote positive finite absolute constants, and $C, C', C_0, C_1, \cdots$ and $c,c',c_0,c_1,\cdots$, to denote positive finite constants whose values do not depend on $n$ and $p$ and may vary at different places. We write $a \lesssim b$ if $a \le C b$ for some constant $C > 0$, and $a \asymp b$ if $a \lesssim b$ and $b \lesssim a$. For a random variable $X$, we write $\|X\|_q = (\E|X|^q)^{1/q}$ for $q>0$. We use $\|h\| = \max_{1\le m , k \le p} |h_{mk}|$ and $\|h\|_\infty = \sup_{\vx_1,\vx_2\in\mathbb{R}^p} \|h(\vx_1,\vx_2)\|$. Throughout the paper, we write $\vX_1^n = (\vX_1,\cdots,\vX_n)$ and let $\vX$ and $\vX'$ be two independent random vectors in $\mathbb{R}^p$ with the distribution $F$, which are independent of $\vX_1^n$. We write $\E h = \E [h(\vX,\vX')]$ and $\E g = \E g(\vX)$, where $g(\vx) = \E[h(\vx,\vX')] - \E h$ and $\vx \in \mathbb{R}^p$. For a matrix-valued kernel $h : \mathbb{R}^p \times \mathbb{R}^p \to \mathbb{R}^{p \times p}$, we say that: (i) $h$ is {\it non-degenerate} w.r.t. $F$ if $\Var(g_{mk}(\vX)) > 0$ for all $m,k=1,\cdots,p$; (ii) $h$ is {\it canonical} or {\it completely degenerate} w.r.t. $F$ if $\E[h_{mk}(\vx_1,\vX')] = \E[h_{mk}(\vX,\vx_2)] = \E[h_{mk}(\vX,\vX')] = 0$ for all $m,k=1,\cdots,p$ and for all $\vx_1,\vx_2 \in \mathbb{R}^p$. Without loss of generality, we shall assume throughout the paper that $p \ge 3$ and the matrix $h = \{h_{mk}\}_{m,k=1}^p$ is symmetric, i.e. $h_{mk} = h_{km}$.

\section{Gaussian approximation}
\label{sec:gaussian-approx}

In this section, we study the Gaussian approximation for $\max_{1\le m,k \le p} (U_{mk} - \E U_{mk})$ in (\ref{eqn:ustat-order2}), or equivalently the approximation for the sup-norm of the centered U-statistics by considering $U-\E U$ and $-U+\E U$. If $\vX_i$'s are non-Gaussian, a seemingly intuitive method would be generating Gaussian random vectors $\vY_i$'s by matching the first and second moments of $\vX_i$; i.e. to approximate $U$ by $U' = {n \choose 2}^{-1} \sum_{1 \le i < j \le n} h(\vY_i,\vY_j)$. However, empirical evidence suggests that this may not be a good approximation and theoretically it seems that the nonlinearity in $U$ and $U'$ accounts for a statistically invalid approximation. To illustrate this point, we simulate $n=200$ iid observations from the $p$-variate elliptic $t$-distribution in (\ref{eqn:elliptic_t_distn}) with mean zero and degree of freedom $\nu=8$ in the SM. We consider the covariance matrix kernel (\ref{eqn:sample-covmat-kernel}) as an example. For $p=40$, the P-P plot of the empirical cdfs of the sup-norm of the centered covariance matrices made from $\vX_i$ and $\vY_i$ is shown in Figure \ref{fig:gaussian-approx} (left) over 5000 simulations.

\begin{figure}[t!] 
   \centering
	\subfigure {\label{subfig:approx_gaussianchaos_n=200_p=40}\includegraphics[scale=0.35]{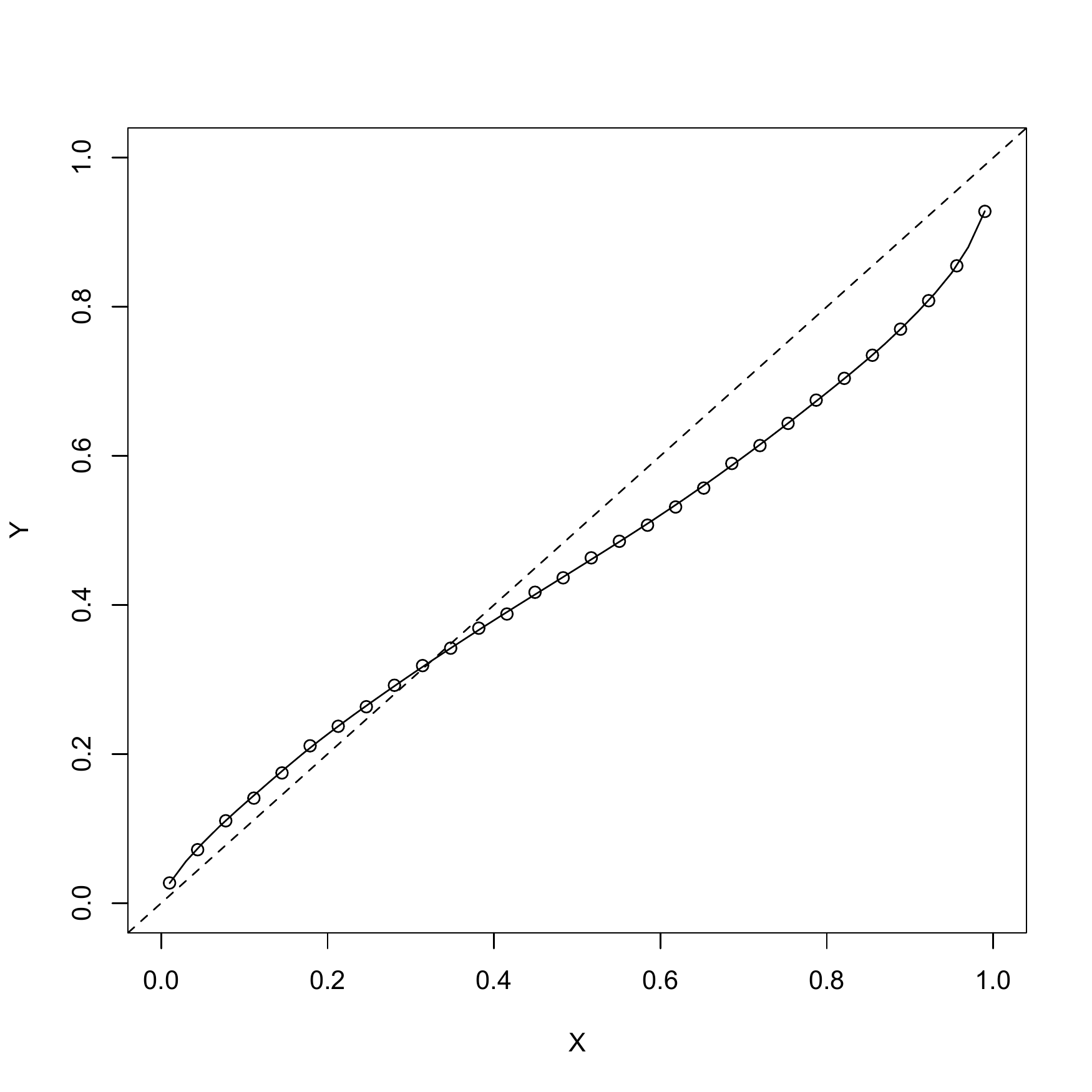}}
	\subfigure{\label{subfig:approx_hoeffding_n=200_p=40} \includegraphics[scale=0.35]{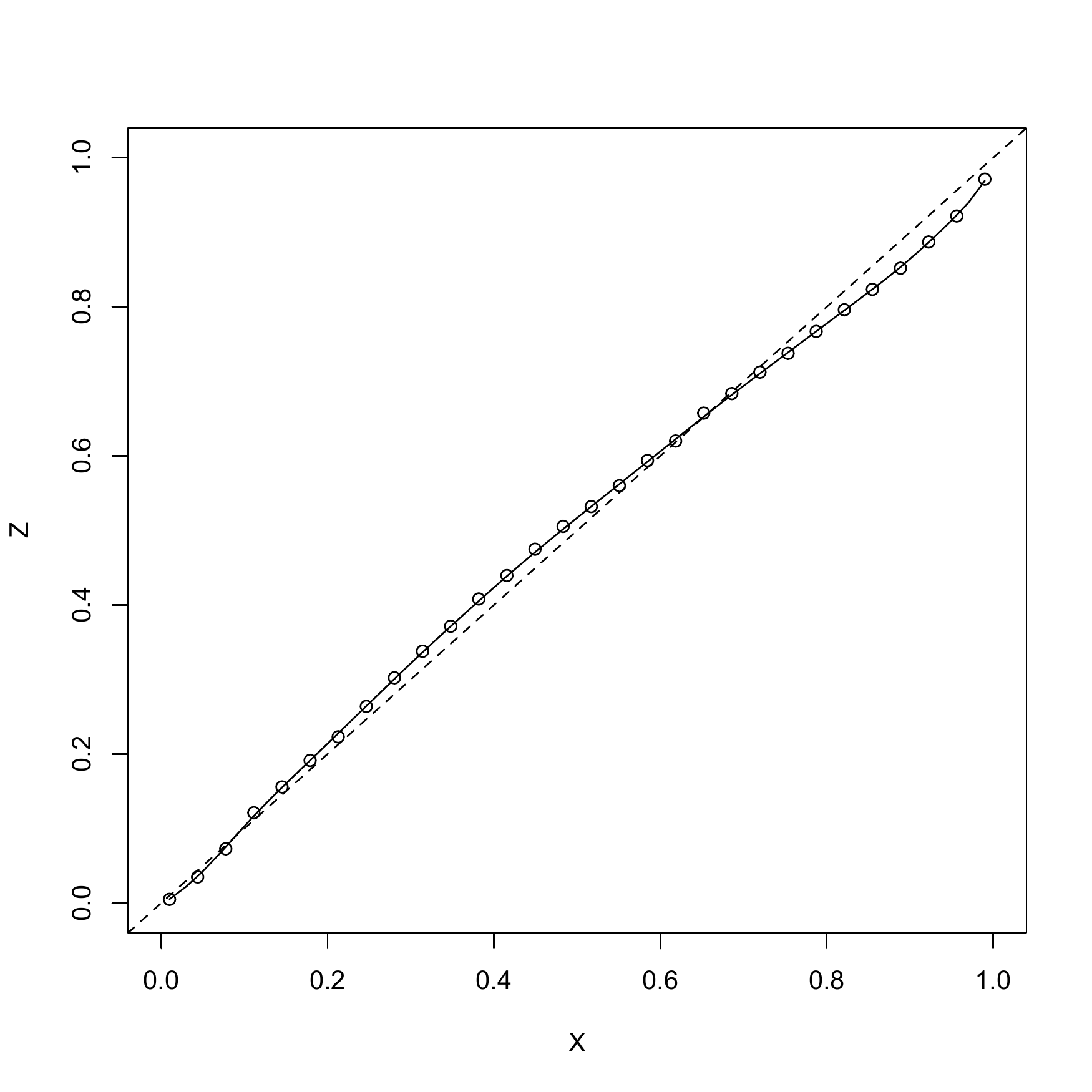}}
	
   \caption{P-P plots of the sup-norm approximation for the centered sample covariance matrix $U - \E U$ with the kernel (\ref{eqn:sample-covmat-kernel}) by $U'-\E U'$ (left) and by the leading term in the Hoeffding decomposition of $U - \E U$ (right).}
   \label{fig:gaussian-approx}
\end{figure}

To correct the bias, a closer inspection reveals that $U$ is an approximately linear statistic and its linear projection part in the Hoeffding decomposition is the leading term. This motivates us to propose a {\it two-step approximation} method. Let
\begin{eqnarray}
\label{eqn:hoeffding-decomp-order-one}
g(\vx_1) &=& \E h(\vx_1,\vX') - \E h, \\
\label{eqn:hoeffding-decomp-order-two}
f(\vx_1,\vx_2) &=& h(\vx_1,\vx_2) -\E h(\vx_1,\vX') - \E h(\vX,\vx_2) + \E h.
\end{eqnarray}
Clearly, $f$ is a $B$-valued symmetric and canonical U-statistic of order two w.r.t. the distribution $F$. Then the Hoeffding decomposition of the kernel $h$ is given by
\begin{equation}
\label{eqn:hoeffding-decomp}
h(\vx_1,\vx_2) = f(\vx_1,\vx_2) + g(\vx_1) + g(\vx_2) + \E h,
\end{equation}
from which we have
$$
U - \E U = {2 \over n(n-1)} \sum_{1 \le i<j \le n} f(\vX_i, \vX_j) + {2 \over n} \sum_{i=1}^n g(\vX_i).
$$
On the right-hand side of the last expression, the second term is expected to be the leading term (a.k.a. the H\'ajek projection) and the first term to be negligible under the sup-norm. Therefore, we can reasonably expect that
$$
{\sqrt{n}\over2}(U - \E U) \approx {1 \over \sqrt{n}} \sum_{i=1}^n g(\vX_i),
$$
where the latter can be further approximated by $n^{-1/2} \sum_{i=1}^n Z_i$ for iid Gaussian random vectors $Z_i \sim N(\vzero,\Gamma_g)$ and $\Gamma_g$ is the positive-definite covariance matrix of $g(\vX_i)$; c.f. \cite{cck2013}. Denote $p'=p(p+1)/2$. Here, we slightly abuse notations and write $\tilde{\vg}_i = \vech(g(\vX_i))$ as the half-vectorized lower triangular matrix of $g(\vX_i)$ by columns. Therefore $\Gamma_g=\Cov(\tilde{\vg}_i)$ is the $p' \times p'$ covariance matrix indexed by $((j,k),(m,l))$ such that $j \ge k$ and $m \ge l$. Similarly, we shall use $Z_i$ to denote either the $p \times p$ matrix or the $p' \times 1$ half-vectorized version. For the previous elliptic $t$-distribution example, we plot the empirical cdfs of $\max_{1\le m,k \le p}(\hat{S}_{mk}-\sigma_{mk})/2$ against $\max_{1\le m,k \le p} n^{-1}\sum_{i=1}^n Z_{i,mk}$. Figure \ref{fig:gaussian-approx} (right) shows a much better approximation using the leading term in the Hoeffding decomposition.

Let $T = \sqrt{n}(U-\E U)/2$, $L = n^{-1/2} \sum_{i=1}^n g(\vX_i)$, $W = n^{-1/2} (n-1)^{-1} \sum_{1 \le i<j \le n} f(\vX_i, \vX_j)$, and $Z = n^{-1/2} \sum_{i=1}^n Z_i$, where $Z_i$ are iid $N(\vzero, \Gamma_g)$. Denote $\bar{T}_0 = \max_{m,k} T_{mk}$, $\bar{L}_0 = \max_{m,k} L_{mk}$, and $\bar{Z}_0 = \max_{m,k} Z_{mk}$. Let 
$$\rho(\bar{T}_0,\bar{Z}_0) = \sup_{t \in \mathbb{R}} | \Prob(\bar{T}_0 \le t) - \Prob(\bar{Z}_0 \le t) |$$
be the Kolmogorov distance between $\bar{T}_0$ and $\bar{Z}_0$. Let $B_n \ge 1$ be a sequence of real numbers possibly tending to infinity. We consider two types of conditions on the kernel moments. First, we establish the explicit convergence rate for the kernels with sub-exponential moments; e.g. the $\varepsilon$-contaminated normal distribution (\ref{eqn:eps_contaminated_normal_distn}) in the SM.

\begin{thm}[Gaussian approximation for centered U-statistics: sub-exponential kernel]
\label{thm:gaussian-approximation-nondegenarate-U-exp-tail}
Let $U$ be a non-degenerate U-statistic of order two. Assume that there exist constants $C_1,C_2 \in (0,\infty)$ and $K \in (0,1)$ such that
\begin{enumerate}
\item[(GA.1)] {\bf Kernel moment:} $\E g_{mk}^2 \ge C_1$ and 
\begin{equation}
\label{eqn:subexponential-kernel-moment-condition}
\max_{\ell=0,1,2} \E(|h_{mk}|^{2+\ell} / B_n^\ell) \vee \E[\exp(|h_{mk}| / B_n)] \le 2
\end{equation}
for all $1\le m,k \le p$;

\item[(GA.2)] {\bf Scaling limit:}
\begin{equation}
\label{eqn:subexponential-kernel-scaling-limit}
{B_n^2 \log^7(pn) \over n} \le C_2 n^{-K}.
\end{equation}
\end{enumerate}
Then there exists a constant $C > 0$ depending only on $C_1,C_2$ such that
\begin{equation}
\label{eqn:gaussian-approximation-nondegenarate-U-exp-tail}
\rho(\bar{T}_0,\bar{Z}_0) \le C n^{-K/8}.
\end{equation}
\end{thm}

The assumptions in Theorem \ref{thm:gaussian-approximation-nondegenarate-U-exp-tail} have meaningful interpretations. (GA.1) ensures the non-degeneracy of the Gaussian approximation and that the truncation does not lose too much information due to the sub-exponential tails. (GA.2) describes the high-dimensional scaling limit of valid Gaussian approximation range. In the high-dimensional context, the dimension $p$ grows with the sample size $n$ and the distribution function $F$ also depends on $n$. Therefore, $B_n$ is allowed to increase with $n$. Theorem \ref{thm:gaussian-approximation-nondegenarate-U-exp-tail} shows that the approximation error in the Kolmogorov distance converges to zero even if $p$ can be much larger than $n$ and no structural assumptions on $F$ are required. In particular, Theorem \ref{thm:gaussian-approximation-nondegenarate-U-exp-tail} applies to kernels with the sub-exponential distribution such that $\|h_{mk}\|_q \le C q$ for all $q \ge 1$, in which case $B_n = O(1)$ and the dimension $p$ is allowed to have a subexponential growth rate in the sample size $n$, i.e. $p = O(\exp(n^{(1-K)/7}))$. Condition (GA.1) also covers bounded kernels $\|h\|_\infty \le B_n$, where $B_n$ may increase with $n$.

\begin{rmk}
Theorem \ref{thm:gaussian-approximation-nondegenarate-U-exp-tail} shows that the asymptotic validity of Gaussian approximation for centered non-degenerate U-statistics holds under the high-dimensional scaling limit (GA.2), which involves only a polynomial factor of $\log{p}$. However, the sup-norm convergence rate $n^{-K/8}$ obtained in (\ref{eqn:gaussian-approximation-nondegenarate-U-exp-tail}) is slower than $n^{-1/2}$. Similar observations have been made in the existing literature on the Berry-Esseen type bounds \cite{portnoy1986,bentkus2003} for the normalized sums of iid random vectors $\vX_i \in \mathbb{R}^p$ with mean zero and the identity covariance matrix. \cite{portnoy1986} showed that the sample mean $\sqrt{n} \bar{\vX}$ has the asymptotic normality if $p = o(\sqrt{n})$ and \cite{bentkus2003} showed that 
$$\sup_{A \in {\cal A}} |\Prob(\sqrt{n} \bar{\vX} \in A) - \Prob(\vZ \in A)| \le K p^{1/4} \E |\vX_1|^3 / n^{1/2},$$
where ${\cal A}$ is the class of all convex subsets in $\mathbb{R}^p$, $\vZ \sim N(\vzero, \Id_p)$, and $K > 0$ is an absolute constant. In either case, the dependence of the CLT rate on the dimension $p$ is polynomial ($p/n^{1/2}$ and $p^{7/4}/n^{1/2}$, resp). \cite{cck2013} considered the Gaussian approximation for $\max_{j \le p} \sqrt{n} \bar{X}_j$ and they obtained the rate $n^{-c}$ for some (unspecified) exponent $c>0$. Following the proofs of Theorem \ref{thm:gaussian-approximation-nondegenarate-U-exp-tail} in the current paper and Theorem 2.2 and Corollary 2.1 in \cite{cck2013}, we can show that $c$ is allowed to take the value $K/8$. Therefore, the effect of higher-order terms than the H\'ajek projection to a linear subspace in the Hoeffding decomposition vanishes in the Gaussian approximation. A similar observation is made for the uniform polynomial moment kernels; c.f. Theorem \ref{thm:gaussian-approximation-nondegenarate-U-uniform-poly-tail}. For multivariate symmetric statistics of order two, to the best of our knowledge, the Gaussian approximation result (\ref{eqn:gaussian-approximation-nondegenarate-U-exp-tail}) with the explicit convergence rate is new. When $p$ is fixed, the rate of convergence and the Edgeworth expansion of such statistics can be found in \cite{bickelgotzevanzwet1986,gotze1987,bentkusgotzevanzwet1997}. In those papers, assuming the Cram\'er condition on $g(\vX_1)$ and suitable moment conditions on $h(\vX_1,\vX_2)$, the Edgeworth expansion of U-statistics was established for the univariate case ($p=1$) with remainder $o(n^{-1/2})$ or $O(n^{-1})$ \cite{bickelgotzevanzwet1986,bentkusgotzevanzwet1997} and the multivariate case ($p>1$ fixed) with remainder $o(n^{-1/2})$ \cite{gotze1987}. In the latter work \cite{gotze1987}, it is unclear that how the constant in the error bound depends on the dimensionality parameter $p$. On the contrary, our Theorem \ref{thm:gaussian-approximation-nondegenarate-U-exp-tail} can allow $p$ to be larger than $n$ in order to obtain the CLT type results in much higher dimensions.
\qed
\end{rmk}

Next, we consider kernels with {\it uniform} polynomial moments (up to the fourth order); e.g. the elliptical $t$-distribution (\ref{eqn:elliptic_t_distn}) in the SM.

\begin{thm}[Gaussian approximation for centered U-statistics: kernel with uniform polynomial moment]
\label{thm:gaussian-approximation-nondegenarate-U-uniform-poly-tail}
Let $U$ be a non-degenerate U-statistic of order two. Assume that there exist constants $C_1,C_2 \in (0,\infty)$ and $K \in (0,1)$ such that
\begin{enumerate}
\item[(GA.1')] {\bf Kernel moment:} $\E g_{mk}^2 \ge C_1$ and 
\begin{equation}
\label{eqn:unifpoly-kernel-moment-condition}
\max_{\ell=0,1,2} \E(|h_{mk}|^{2+\ell} / B_n^\ell) \vee \E[(\|h\| / B_n)^4] \le 1
\end{equation}
for all $1\le m,k \le p$;

\item[(GA.2')] {\bf Scaling limit:}
\begin{equation}
\label{eqn:unifpoly-kernel-scaling-limit}
{B_n^4 \log^7(pn) \over n} \le C_2 n^{-K}.
\end{equation}
\end{enumerate}
Then there exists a constant $C > 0$ depending only on $C_1,C_2$ such that (\ref{eqn:gaussian-approximation-nondegenarate-U-exp-tail}) holds.
\end{thm}

Theorem \ref{thm:gaussian-approximation-nondegenarate-U-exp-tail} and \ref{thm:gaussian-approximation-nondegenarate-U-uniform-poly-tail} allow us to approximate the quantiles of $\bar{T}_0$ by those of $\bar{Z}_0$, with the knowledge of $\Gamma_g$. In practice, the covariance matrix $\Gamma_g$ and the H\'ajek projection terms $g(\vX_i), i=1,\cdots,n,$ depend on the data distribution $F$, which is unknown. Thus, quantiles of $\bar{Z}_0$ need to be estimated in real applications. However, we shall see in Section \ref{subsec:gassuain_wild_bootstrap} that Theorem \ref{thm:gaussian-approximation-nondegenarate-U-exp-tail} and \ref{thm:gaussian-approximation-nondegenarate-U-uniform-poly-tail} can still be used to derive a feasible resampling based method to approximate the quantiles of Gaussian maxima $\bar{Z}_0$ and therefore $ \bar{T}_0$.

\section{Wild bootstrap}
\label{subsec:gassuain_wild_bootstrap}

The main purpose of this section is to approximate the quantiles of $\bar{T}_0$. Let $\vX'_1,\cdots,\vX'_n$ be an independent copy of $\vX_1,\cdots,\vX_n$ that are observed; call this training data. Such data can always be obtained by a {\it half-sampling} or {\it data splitting} on the original data. Therefore, we assume that the sample size of total data $\{\vX_1,\cdots,\vX_n, \vX'_1,\cdots,\vX'_n\}$ is $2n$. Since $g(\vX_i), i=1,\cdots,n,$ are unknown, we construct an estimator for it. Let $\hat{g}_i := \hat{g}_i(\vX_1,\cdots,\vX_n, \vX'_1,\cdots,\vX'_n)$ be an estimator of $g(\vX_i)$ using the original and training data. Recall that $g(\vx) = \E h(\vx, \vX'_j) - \E h(\vX'_i,\vX'_j)$ for any fixed $\vx \in \mathbb{R}^p$, which can be viewed as the population version for the second variable $\vX'_j$. Therefore, we build an empirical version as our estimator of $g(\vX_i)$. Specifically, we consider
\begin{equation}
\label{eqn:hat_g_i}
\hat{g}_i = {1\over n} \sum_{j=1}^n h(\vX_i, \vX'_j) - {n \choose 2}^{-1} \sum_{1 \le j < l \le n} h(\vX'_j, \vX'_l).
\end{equation}
Conditional on $\vX_i$, $\hat{g}_i$ is an unbiased estimator of $g(\vX_i)$. It is interesting to view $\hat{g}_i$ as a {\it decoupled} estimator of $g(\vX_i)$. Let
\begin{equation*}
\hat{L}_0^* = \max_{1 \le m,k \le p} {1\over\sqrt{n}} \sum_{i=1}^n \hat{g}_{i,mk} e_i \quad \text{and} \quad  \bar{L}_0^* = \max_{1 \le m,k \le p} {1\over\sqrt{n}} \sum_{i=1}^n g_{mk}(\vX_i) e_i,
\end{equation*}
where $e_1,e_2,\cdots$ are iid standard Gaussian random variables that are also independent of $\vX_1,\cdots,\vX_n, \vX'_1,\cdots,\vX'_n$. Then $\hat{L}_0^*$ and $\bar{L}_0^*$ are bootstrapped versions of $\bar{L}_0$. Denote the conditional quantiles of $\hat{L}_0^*$ and $\bar{L}_0^*$ given the data $\vX_1,\cdots,\vX_n,\vX'_1,\cdots,\vX'_n$ as
\begin{eqnarray*}
a_{\hat{L}_0^*}(\alpha) &=& \inf\{ t \in \mathbb{R} : \Prob_e(\hat{L}_0^* \le t) \ge \alpha \}, \\
a_{\bar{L}_0^*}(\alpha) &=& \inf\{ t \in \mathbb{R} : \Prob_e(\bar{L}_0^* \le t) \ge \alpha \},
\end{eqnarray*}
where $\Prob_e$ is the probability taken w.r.t. $e_1,\cdots,e_n$. Now, we can compute the conditional quantile $a_{\hat{L}_0^*}(\alpha)$ by the Gaussian wild bootstrap method. Specifically, $a_{\hat{L}_0^*}(\alpha)$ can be numerically approximated by resampling on the multiplier Gaussian random variables $e_1,\cdots,e_n$ and we wish to use $a_{\hat{L}_0^*}(\alpha)$ to approximate the quantiles of $\bar{T}_0$.

\begin{thm}[Asymptotically validity of Gaussian wild bootstrap for centered U-statistics]
\label{thm:gaussian_wild_bootstrap_validity}
Let $U$ be a non-degenerate U-statistic of order two. \\
(i) {\bf (Subexponential kernel)} If (GA.1) and (GA.2) hold for some constants $C_1,C_2 \in (0,\infty)$ and $K \in (0,1)$, then there exist a constant $C > 0$ depending only on $C_1,C_2$ such that for all $\alpha \in (0,1)$
\begin{equation}
\label{eqn:gaussian_wild_bootstrap_validity}
| \Prob(\bar{T}_0 \le a_{\hat{L}_0^*}(\alpha)) - \alpha | \le C n^{-K/8}.
\end{equation}
(ii) {\bf (Uniform polynomial kernel)} If (GA.1') and (GA.2') hold for some constants $C_1,C_2 \in (0,\infty)$ and $K \in (0,1)$, then there exist a constant $C > 0$ depending only on $C_1,C_2$ such that for all $\alpha \in (0,1)$
\begin{equation}
\label{eqn:gaussian_wild_bootstrap_validity_unifpoly}
| \Prob(\bar{T}_0 \le a_{\hat{L}_0^*}(\alpha)) - \alpha | \le C n^{-K/12}.
\end{equation}
\end{thm}

\begin{rmk}
\label{rmk:wild-bootstrap}
From Theorem \ref{thm:gaussian_wild_bootstrap_validity}, the convergence rate of the wild bootstrap approach for subexponential kernels is the same as the Gaussian approximation results (Theorem \ref{thm:gaussian-approximation-nondegenarate-U-exp-tail}), while it is slower for kernels with uniform polynomial moment of the order four (Theorem \ref{thm:gaussian-approximation-nondegenarate-U-uniform-poly-tail}). The major error in the latter case is due to the estimation of $\bar{L}_0^*$ by $\hat{L}_0^*$ in the wild bootstrap. Under (GA.1') and (GA.2'), the approximation error of $\hat{L}_0^*$ for $\bar{L}_0^*$ is on the order $O(n^{-1/4} B_n (\log(np))^{1/2})$; see Lemma \ref{lem:gaussian_wild_bootstrap_Delta2_unifpoly_moment} in the SM. This is different from the previous work \cite{cck2013}, which does not need this extra estimation step for $g(\vX_i), i =1,\cdots,n,$ since only sums of iid random vectors $n^{-1/2} \sum_{i=1}^n \vX_i$ are involved. Therefore, for sums of iid random vectors, the wild bootstrap can attain the rate $n^{-K/8}$ for both subexponential and uniform polynomial moment (of the order four) observations. However, with better moment conditions on the U-statistic kernel, the rate $n^{-K/8}$ can be attained for polynomial moment kernels. Specifically, assuming that 
$\max_{\ell=0,1,2} \E(|h_{mk}|^{2+\ell} / B_n^\ell) \vee \E[(\|h\| / B_n)^q] \le 1$ for $q \ge 8$, one can show that the convergence rate (\ref{eqn:gaussian_wild_bootstrap_validity}) is attained. In addition, \cite{cck2013} does not deal with the higher-order nonlinear terms, here we have to explicitly handle the canonical part $W$ in the Hoeffding decomposition. The degeneracy structure plays a key role to establish the convergence of the bootstrap method (as well as the Gaussian approximation results in Section \ref{sec:gaussian-approx}) and new proof techniques, in particular the {\it decoupling}, are required. The established moment bounds in Section \ref{app:concentration_ineq_canonical-Ustat} are especially suitable for controlling the completely degenerate errors of quadratic forms.
\qed
\end{rmk}


To assess the quality of the Gaussian wild bootstrap, we show two examples for the covariance matrix kernel on the $\varepsilon$-contaminated normal distribution (\ref{eqn:eps_contaminated_normal_distn}) with the sub-exponential moment and on the elliptic $t$-distribution (\ref{eqn:elliptic_t_distn}) with the uniform polynomial moment. In each simulation, we generate 200 bootstrap samples for $\hat{L}_0^*$. Then, we estimate $\Prob(\bar{T}_0 \le a_{\hat{L}_0^*}(\alpha))$ for the whole range of probabilities $\alpha\in[0.01,0.99]$. Figure \ref{fig:wild_bootstrap} shows the empirical approximation result. Here, we choose $V = 0.9\times\vone_p\vone_p^\top + 0.1\times \Id_p$ in (\ref{eqn:eps_contaminated_normal_distn}) and (\ref{eqn:elliptic_t_distn}), where $\vone_p$ is the $p \times 1$ vector of all ones. From Figure \ref{fig:wild_bootstrap}, the bootstrap approximation seems to be better in the sub-exponential moment case than in the polynomial moment case; see (GA.1)+(GA.2) versus (GA.1')+(GA.2'). More simulation examples can be found in the SM.

\begin{figure}[t!] 
   \centering
      \subfigure{\label{subfig:approx_wild_bootstrap_eps_contaminated_n=200_p=40} \includegraphics[scale=0.35]{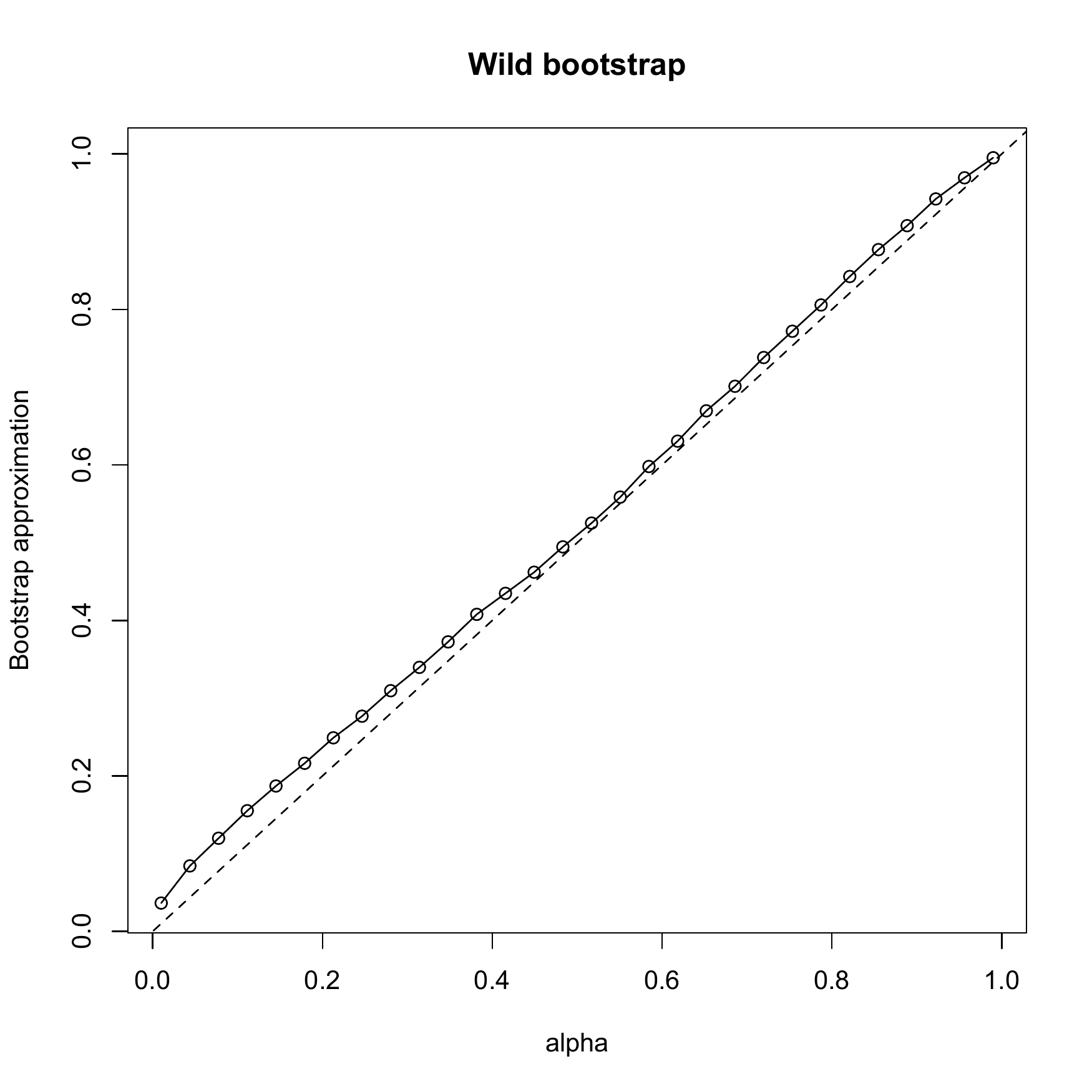}}
	\subfigure {\label{subfig:approx_wild_bootstrap_n=200_p=40}\includegraphics[scale=0.35]{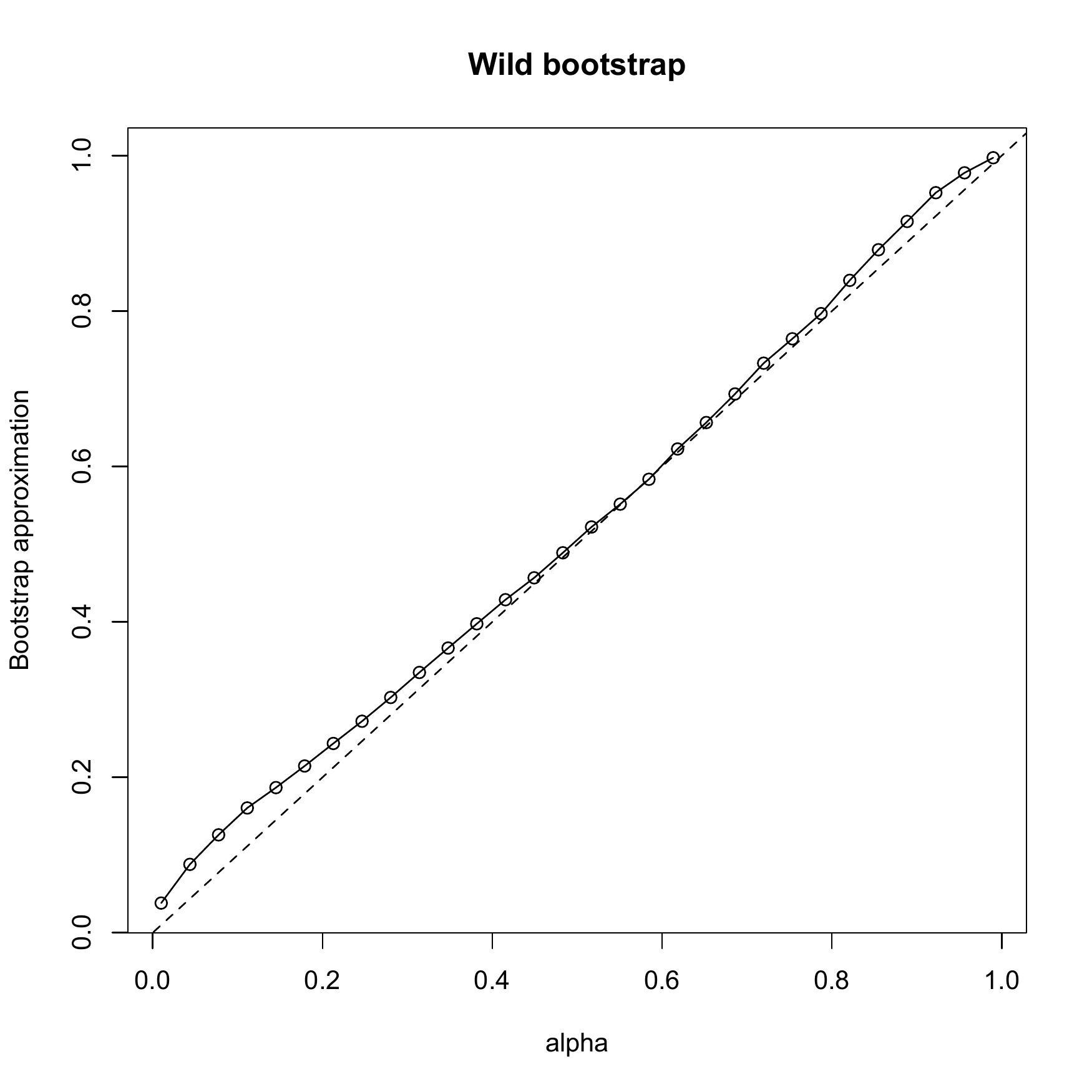}}
   \caption{Plots of the estimated probabilities $\Prob(\bar{T}_0 \le a_{\hat{L}_0^*}(\alpha))$ by the Gaussian wild bootstrap for the range of $\alpha \in [0.01,0.99]$. Left is the $\varepsilon$-contaminated normal distribution (model: (M1)+(D1) in the SM) and right is the elliptical $t$-distribution (model: (M2)+(D1)). Sample size $n=200$ and dimension $p=40$.}
   \label{fig:wild_bootstrap}
\end{figure}

\section{Statistical applications}
\label{sec:stat_apps}

In this section, we present two statistical applications for the theoretical results established in Section \ref{sec:gaussian-approx}--\ref{subsec:gassuain_wild_bootstrap}. Two additional examples can be found in the SM. Here, for notational convenience, we rescale $\hat{L}_0^*$ and let $\hat{L}_0^* = 2 n^{-1} \max_{1 \le m,k \le p} |\sum_{i=1}^n \hat{g}_{i,mk} e_i|$. Recall that $\Gamma_g = \Cov(\tilde\vg_i)$, where $\vg_i = \vech(g(\vX_i))$.

\subsection{Tuning parameter selection for the thresholded covariance matrix estimator}
\label{subsec:tuning_selection_thresholded_cov_mat}

Consider the problem of {\it sparse} covariance matrix estimation. Let $r \in [0,1)$ and
\begin{equation*}
{\cal G}(r, C_0, \zeta_p) = \Big\{ \Sigma \in \mathbb{R}^{p\times p} : \sigma_{mm} \le C_0 , \sum_{k=1}^p |\sigma_{mk}|^r \le \zeta_p \text{ for all } m = 1,\cdots,p \Big\}
\end{equation*}
be the class of sparse covariance matrices in terms of the strong $\ell^r$-ball. Here, $C_0 > 0$ is a constant and $\zeta_p > 0$ may grow with $p$. Let $\tau \ge 0$ and
$$
\hat\Sigma(\tau) = \{\hat{s}_{mk} \vone\{|\hat{s}_{mk}| > \tau \} \}_{m,k=1}^p
$$
be the thresholded sample covariance matrix estimator of $\Sigma$. The class ${\cal G}(r, C_0, \zeta_p)$ was introduced in \cite{bickellevina2008b} and the high-dimensional properties of $\hat\Sigma(\tau)$ were analyzed in \cite{bickellevina2008b} for iid sub-Gaussian data and in \cite{chenxuwu2013a,chenxuwu2015+} for heavy-tailed time series with algebraic tails. In both scenarios, the rates of convergence were obtained with the Bonferroni (i.e. union bound) technique and one-dimensional concentration inequalities. Those performance bounds of the thresholded estimator $\hat\Sigma(\tau)$ critically depend on the tuning parameter $\tau$. The ideal choice of the threshold for establishing the rate of convergence under the spectral and Frobenius norms is $\tau_\diamond = \|\hat{S}-\Sigma\|$, whose distribution depends on the unknown underlying data distribution $F$. In the problem of the high-dimensional sparse covariance matrix estimation, data-dependent tuning parameter selection is often empirically done with the cross-validation (CV) and its theoretical properties largely remain unknown. High probability bounds of $\tau_\diamond$ are given in \cite{bickellevina2008b,chenxuwu2013a}. Here, we provide a principled and data-dependent way to determine the threshold $\tau$.

\begin{defn}[Subgaussian random variable]
\label{defn:subgaussian_rv}
A random variable $X$ is said to be {\it subgaussian} with mean zero and {\it variance factor} $\nu^2$, if 
\begin{equation}
\label{eqn:subgaussian_rv}
\E [\exp(X^2 / \nu^2) ] \le \sqrt{2}.
\end{equation}
Denote $X \sim \text{subgaussian}(\nu^2)$. In particular, if $X \sim N(0,\sigma^2)$, then $X \sim \text{subgaussian}(4 \sigma^2)$.
\end{defn}

The upper bound $\sqrt{2}$ in (\ref{eqn:subgaussian_rv}) is not essential and it is chosen for conveniently comparing with $\|X\|_{\psi_2}$, which is the Orlicz norm of $X$ for $\psi_2(x) = \exp(x^2) - 1$ and $x \ge 0$. In general, the variance factor for a subgaussian random variable is {\it not} equivalent to the variance. For a sequence of random variables $X_n, n = 1,2,\cdots,$ if $X_n \sim \text{subgaussian}(\nu^2)$ and $\sigma^2 = \Var(X_n)$, then by Markov's inequality, we always have $\sigma^2 \le \sqrt{2} \nu^2$, while $\nu^2$ may depend on $n$ and it may diverge at faster rate than $\sigma^2$ such as $\sigma^2 \le C_0$ and $\nu^2 \to \infty$ as $n \to \infty$. As a simple example, let $a_n > 0$ be a sequence of real numbers such that $a_n \to \infty$ and consider random variables $X_1,\cdots,X_n$ such that $\Prob(X_n = \pm a_n) = (2 a_n^2)^{-1}$ and $\Prob(X_n = 0) = 1 - (2 a_n^2)^{-1}$. Obviously, $\E X_n = 0$ and $\Var(X_n) = 1$. Let $\nu^2 = C a_n^2$ for some constant $C > 0$. Then $\E[\exp(X_n^2 / \nu^2)] = [1-(2 a_n^2)^{-1}] + a_n^{-2} e^{C^{-1}} \le \sqrt{2}$ for all large enough $n$; i.e. $X_n \sim \text{subgaussian}(C a_n^2)$. In fact, if $X \sim \text{subgaussian}(\nu^2)$, then $\nu^2 \ge \|X\|_{\psi_2}$. Therefore, we are mainly interested in the general case when $\nu^2 := \nu^2_n \to \infty$ as $n \to \infty$ in the statistical applications.

\begin{thm}[Adaptive threshold selection by wild bootstrap: subgaussian observations]
\label{thm:thresholded_cov_mat_rate_adaptive}
Let $\nu \ge 1$ and $\vX_i$ be iid mean zero random vectors such that $X_{ik} \sim  \text{subgaussian}(\nu^2)$ for all $k=1,\cdots,p$ and $\Sigma \in {\cal G}(r, C_0, \zeta_p)$. Suppose that there exist constants $C_i > 0, i=1,\cdots,4,$ such that $\{\Gamma_g\}_{(j,k),(j,k)} \ge C_1$, $\|X_{1k}\|_4 \le C_2$, $\|X_{1k}\|_6 \le C_3 \nu^{1/3}$ and $\|X_{1k}\|_8 \le C_4 \nu^{1/2}$ for all $j,k=1,\cdots,p$. Let $\beta \in (0,1)$ and $\tau_* = \beta^{-1} a_{\hat{L}_0^*}(1-\alpha)$, where the bootstrap samples are generated with the covariance matrix kernel in (\ref{eqn:sample-covmat-kernel}). If $\nu^4 \log^7(np) \le C_5 n^{1-K}$ for some $K \in (0,1)$, then we have
\begin{eqnarray}
\label{eqn:thresholded_cov_mat_rate_spectral_adaptive}
\|\hat\Sigma(\tau_*)-\Sigma\|_2 &\le& \left[{3+2\beta \over \beta^{1-r}}+ \left({\beta \over 1-\beta}\right)^r \right] \zeta_p a_{\hat{L}_0^*}(1-\alpha)^{1-r}, \\
\label{eqn:thresholded_cov_mat_rate_F_adaptive}
p^{-1} | \hat\Sigma(\tau_*)-\Sigma|_F^2 &\le& 2\left[{4+3\beta^2 \over \beta^{2-r}}+ 2 \left({\beta \over 1-\beta}\right)^r \right] \zeta_p a_{\hat{L}_0^*}(1-\alpha)^{2-r},
\end{eqnarray}
with probability at least $1-\alpha-C n^{-K/8}$ for some constant $C > 0$ depending only on $C_1,\cdots,C_5$. In addition, we have $\E[a_{\hat{L}_0^*}(1-\alpha)] \le C' (\log(p)/n)^{1/2}$ and
\begin{equation}
\label{eqn:tau_*-bound_subgaussian}
\E[\tau_*] \le C' \beta^{-1} (\log(p)/n)^{1/2},
\end{equation}
where $C' > 0$ is a constant depending only on $\alpha$ and $C_1,\cdots,C_5$.
\end{thm}

\begin{rmk}[Comments on the conditions in Theorem \ref{thm:thresholded_cov_mat_rate_adaptive}]
\label{rmk:cov-kernel-comments-on-conditions}
The non-degeneracy condition $\{\Gamma_g\}_{(j,k),(j,k)} \ge C_1$ is quite mild in Theorem \ref{thm:thresholded_cov_mat_rate_adaptive}. Consider the multivariate cumulants of the joint distribution of the random vector $\vX = (X_1,\cdots,X_p)^\top$ following a distribution $F$ in $\mathbb{R}^p$. Let $\chi(\vt) = \E[\exp(\iota \vt^\top \vX)]$ be the characteristic function of $\vX$, where $\vt = (t_1,\cdots,t_p)^\top \in \mathbb{R}^p$ and $\iota = \sqrt{-1}$. Then, the {\it multivariate cumulants} $\kappa_{r_1 r_2 \cdots r_p}^{1 2 \cdots p}$ of the joint distribution of $\vX$ are the coefficients in the expansion
$$
\log \chi(\vt) = \sum_{r_1,r_2,\cdots,r_p = 0}^\infty  \kappa_{r_1 r_2 \cdots r_p}^{1 2 \cdots p} {(\iota t_1)^{r_1} (\iota t_2)^{r_2} \cdots (\iota t_p)^{r_p} \over r_1! r_2! \cdots r_p!}.
$$
For the covariance matrix kernel (\ref{eqn:sample-covmat-kernel}), we have
\begin{equation}
\label{eqn:write-Gamma_g-as-fourth-culmulant}
\{\Gamma_g\}_{(j,k), (m,l)} = (\kappa_{1111}^{jkml} + \sigma_{jm}\sigma_{kl} + \sigma_{jl}\sigma_{km}) / 4,
\end{equation}
where $\kappa_{1111}^{jkml}$ is the fourth-order cumulants of $F$. Therefore, if $\kappa_{1111}^{jkjk} \ge C$ for some (large) constant $C > 0$ depending only on $C_0$ and $C_1$, then $\{\Gamma_g\}_{(j,k), (j,k)} \ge C_1$ for $\Sigma \in {\cal G}(r, C_0, \zeta_p)$.

For data following distributions in the elliptic family \cite[Chapter 1]{muirhead1982}, the condition $\{\Gamma_g\}_{(j,k), (j,k)} \ge C_1$ is equivalent to $\min_{1 \le j \le p} \sigma_{jj} \ge C$ for some constant $C > 0$ depending only on $C_1$. To see this, for $F$ in the elliptic family \cite[Chapter 1]{muirhead1982}, it is known that
$$
\kappa_{1111}^{jkml} = \kappa (\sigma_{jk}\sigma_{ml}+\sigma_{jm}\sigma_{kl}+\sigma_{jl}\sigma_{km}),
$$
where the kurtosis parameter $\kappa = [1+\varepsilon(\nu^4-1)] / [1+\varepsilon(\nu^2-1)]^2-1$ for the $\varepsilon$-contaminated normal distribution in (\ref{eqn:eps_contaminated_normal_distn}) and $\kappa=2 / (\nu-4)$ for the elliptic $t$-distribution in (\ref{eqn:elliptic_t_distn}). For the elliptic $t$-distribution with $\nu=8$ (as considered in Figure \ref{fig:gaussian-approx}), we have $\{\Gamma_g\}_{(j,k), (j,k)} = (3\sigma_{jj}\sigma_{kk} + 4 \sigma_{jk}^2) / 8$. Therefore, $\{\Gamma_g\}_{(j,k), (j,k)} \ge C_1$ if and only if there exists a constant $C > 0$ such that $\sigma_{jj} \ge C$ for all $j=1,\cdots,p$. Similar comments apply to the $\varepsilon$-contaminated normal distribution.

The assumption $\sigma_{kk} \le C_0$ in Theorem \ref{thm:thresholded_cov_mat_rate_adaptive} is redundant and it is automatically fulfilled under a slightly stronger condition $\|X_{1k}\|_4 \le C_2$. Conditions on the growth rate on $\|X_{1k}\|_\ell, \ell=4,6,8$, are also not restrictive. Consider the special case for the multivariate Gaussian distribution $\vX_i \sim N(\vzero, \Sigma)$ such that $\sigma_{kk} \le C_0$. Then $X_{1k}$ are subgaussian$(4C_0)$ and $\max_{\ell=4,6,8} \|X_{1k}\|_\ell\le C$ for some constant $C > 0$ depending only on $C_0$. Therefore, if the data follow the Gaussian distribution, then the bootstrapped thresholded covariance matrix estimator $\hat\Sigma(\tau_*)$ attains (\ref{eqn:thresholded_cov_mat_rate_spectral_adaptive}) and (\ref{eqn:thresholded_cov_mat_rate_F_adaptive}) when $p = O(\exp(n^{(1-K)/7}))$. However, we shall emphasize that, for $\Sigma \in {\cal G}(r, C_0, \zeta_p)$, although the diagonal entries in $\Sigma$ are uniformly bounded by a constant $C_0$, we do allow $\nu^2$ to grow with $n$ in the subgaussian distribution, in which case the bootstrap approach can have advantages over the non-adaptive minimax thresholding procedure (see the paragraphs below for more detailed discussions).
\qed
\end{rmk}

There are a number of interesting features of Theorem \ref{thm:thresholded_cov_mat_rate_adaptive}. Consider $r=0$; i.e. $\Sigma$ is truly sparse such that $\max_{1 \le m \le p} \sum_{k=1}^p \vone\{\sigma_{mk} \neq 0\} \le \zeta_p$ for $\Sigma \in {\cal G}(0, C_0, \zeta_p)$. Then we can take $\beta=1$ and the convergence rates are 
$$
\|\hat\Sigma(\tau_*)-\Sigma\|_2 \le 6 \zeta_p a_{\hat{L}_0^*}(1-\alpha) \quad \text{ and } \quad p^{-1} | \hat\Sigma(\tau_*)-\Sigma|_F^2 \le 18 \zeta_p a_{\hat{L}_0^*}(1-\alpha)^2.
$$
Hence, the tuning parameter can be adaptively selected by bootstrap samples while the rate of convergence is {\it nearly optimal} in the following sense.
Since the distribution of $\tau_*$ mimics that of $\tau_\diamond$, $\hat\Sigma(\tau_*)$ achieves the same convergence rate as the thresholded estimator $\hat\Sigma(\tau_\diamond)$ for the oracle choice of the threshold $\tau_\diamond$ with probability at least $1-\alpha-C n^{-K/8}$. On the other hand, the bootstrap method is not fully equivalent to the oracle procedure in terms of the constants in the estimation error bounds. Suppose that we know the support $\Theta$ of $\Sigma$, i.e. locations of the nonzero entries in $\Sigma$. Then, the {\it oracle} estimator is simply $\breve\Sigma = \{\hat{s}_{mk} \vone\{(m,k) \in \Theta\} \}_{m,k=1}^p$ and we have
\begin{eqnarray*}
\|\breve\Sigma - \Sigma\|_2 &\le& \max_{m \le p} \sum_{k=1}^p |\hat{s}_{mk}-\sigma_{mk}| \vone\{(m,k) \in \Theta\} \\
&\le& \|\hat{S}-\Sigma\| \max_{m \le p} \sum_{k=1}^p \vone\{(m,k) \in \Theta\}  = \tau_\diamond \zeta_p.
\end{eqnarray*}
Therefore, the constant of the convergence rate for the bootstrap method does not attain the oracle estimator. However, we shall comment that $\beta$ is not a tuning parameter since it does not depend on $F$ and the effect of $\beta$ only appears in the constants in front of the convergence rates (\ref{eqn:thresholded_cov_mat_rate_spectral_adaptive}) and (\ref{eqn:thresholded_cov_mat_rate_F_adaptive}).

Assuming that the observations are subgaussian$(\nu^2)$ and the variance factor $\nu^2$ is a {\it fixed} constant, it is known that the threshold value $\tau_\Delta = C(\nu) \sqrt{\log(p)/n}$ achieves the minimax rate for estimating the sparse covariance matrix \cite{caizhou2011a}. Compared with the minimax optimal tuning parameter $\tau_\Delta$, our bootstrap threshold $\tau_*$ exhibits several advantages which we shall highlight (with stronger side conditions). First, $\tau_\Delta$ is non-adaptive since the constant $C(\nu) > 0$ depends on the underlying distribution $F$ through $\nu^2$ and it is more conservative than the bootstrap threshold $\tau_*$ in view of (\ref{eqn:tau_*-bound_subgaussian}). The reason is that the minimax lower bound is based on the worst case analysis and the matching upper bound is obtained by the union bound which ignores the dependence structures in $F$. On the contrary, $\tau_*$ takes into account the dependence information of $F$ by conditioning on the observations. Second, the bootstrap threshold $\tau_*$ does not need the knowledge of $\nu^2$ and it allows $\nu^2$ to increase with $n$ and $p$. In this case, the universal thresholding rule $\tau_\Delta = C' \nu \sqrt{\log(p)/n}$ even for $\Sigma \in {\cal G}(r, C_0, \zeta_p)$, in which the variances $\sigma_{kk}, k=1,\cdots,p,$ are uniformly bounded by a constant. In contrast, from (\ref{eqn:tau_*-bound_subgaussian}), the bootstrap threshold $\tau_* = O_\Prob( (\log(p)/n)^{1/2} )$, where the constant of $O_\Prob(\cdot)$ depends only on $\alpha,\beta,C_1,\cdots,C_5$. Therefore $\tau_* = o_\Prob(\tau_\Delta)$ as $\nu^2 \to \infty$ and $\tau_*$ can potentially gain much tighter performance bounds than $\tau_\Delta$. One exception for ruling out the increasing $\nu^2$ when $\max_{k \le p} \sigma_{kk} \le C_0$ is the Gaussian distribution $\vX_i \sim N(\vzero, \Sigma)$. However, the main focus of this paper is the statistical estimation and inference for high-dimensional non-Gaussian data and therefore the Gaussian example is not so interesting here. Third, as we shall demonstrate in Theorem \ref{thm:thresholded_cov_mat_rate_adaptive_polymom}, the Gaussian type convergence rate of the bootstrap method in Theorem \ref{thm:thresholded_cov_mat_rate_adaptive} remains valid even for heavy-tailed data with polynomial moments. Specifically, we have the following result.

\begin{thm}[Adaptive threshold selection by wild bootstrap: uniform polynomial moment observations]
\label{thm:thresholded_cov_mat_rate_adaptive_polymom}
Let $\vX_i$ be iid mean zero random vectors such that $\|\max_{1 \le k \le p} |X_{1k}| \|_8 \le \nu$ and $\Sigma \in {\cal G}(r, C_0, \zeta_p)$. Suppose that there exist constants $C_i > 0, i=1,\cdots,4,$ such that $\{\Gamma_g\}_{(j,k),(j,k)} \ge C_1$, $\|X_{1k}\|_4 \le C_2$, $\|X_{1k}\|_6 \le C_3 \nu^{1/3}$ and $\|X_{1k}\|_8 \le C_4 \nu^{1/2}$ for all $j,k=1,\cdots,p$. Let $\beta \in (0,1)$ and $\tau_* = \beta^{-1} a_{\hat{L}_0^*}(1-\alpha)$, where the bootstrap samples are generated with the covariance matrix kernel in (\ref{eqn:sample-covmat-kernel}). If $\nu^8 \log^7(np) \le C_5 n^{1-K}$ for some $K \in (0,1)$, then (\ref{eqn:thresholded_cov_mat_rate_spectral_adaptive}) and (\ref{eqn:thresholded_cov_mat_rate_F_adaptive}) hold with probability at least $1-\alpha-C n^{-K/12}$ for some constant $C > 0$ depending only on $C_1,\cdots,C_5$. In addition, (\ref{eqn:tau_*-bound_subgaussian}) holds for some constant $C' > 0$ depending only on $\alpha$ and $C_1,\cdots,C_5$.
\end{thm}

From Theorem \ref{thm:thresholded_cov_mat_rate_adaptive_polymom}, the subgaussian assumption on $F$ is not essential: for the non-Gaussian data with heavier tails than subgaussian, the thresholded covariance matrix estimator with the threshold selected by the wild bootstrap approach again attains the Gaussian type convergence rate at the asymptotic confidence level $100(1-\alpha)\%$. In particular, the dimension $p$ may still be allowed to increase subexponentially fast in the sample size $n$. The cost of the heavy-tailed distribution $F$ is only a sacrifice of the convergence rate from $n^{-K/8}$ to $n^{-K/12}$. However, as commented in Remark \ref{rmk:wild-bootstrap}, this gap becomes smaller and eventually vanishes for stronger moment conditions (here, we need $q \ge 16$).

Next, we compare Theorem \ref{thm:thresholded_cov_mat_rate_adaptive_polymom} with the threshold obtained by the union bound approach. Assume that $\E|X_{1k}|^q < \infty$ for $q \ge 8$. By the Nagaev inequality \cite{nagaev1979a} applied to the split sample in Remark \ref{rmk:data-splitting-reduction}, one can show that
$$\tau_\sharp = C_q \Big\{ {p^{4/q} \over n^{1-2/q}} \xi_q + \Big( {\log{p} \over n}\Big)^{1/2} \xi_4 \Big\}, \qquad \text{where } \xi_q = \max_{1 \le k \le p} \|X_{1k}\|_q,$$
is the right threshold that gives a large probability bound for $\tau_\diamond = \|\hat{S}-\Sigma\|$. For $q = 8$, we see that $\tau_* = o_\Prob(\tau_\sharp)$ when $n^{1/2} = o(p)$. Therefore, in high dimensional settings, the bootstrap method adapts to the dependence in $F$ and gives better convergence rate under the spectral and Frobenius norms. Moreover, for observations with polynomial moments, the minimax lower bound is currently not available to justify $\tau_\sharp$.

%


\subsection{Simultaneous inference for covariance and rank correlation matrices}

Another related important problem of estimating the sparse covariance matrix $\Sigma$ is the consistent recovery of its support, i.e. non-zero off-diagonal entries in $\Sigma$ \cite{lamfan2009a}. Towards this end, a lower bound of the minimum signal strength ($\Sigma$-min condition) is a necessary condition to separate the weak signals and true zeros. Yet, the $\Sigma$-min condition is never verifiable. To avoid this undesirable condition, we can alternatively formulate the recovery problem as a more general hypothesis testing problem
\begin{equation}
\label{eqn:cov_mat_simultaneous_test_formulation}
H_0: \Sigma = \Sigma_0  \quad \text{versus} \quad  H_1: \Sigma \neq \Sigma_0,
\end{equation}
where $\Sigma_0$ is a known $p \times p$ matrix. In particular, if $\Sigma_0=\Id_{p \times p}$, then the support recovery can be re-stated as the following simultaneously testing problem: for all $m, k \in \{ 1,\cdots,p \}$ and $m \neq k$,
\begin{equation}
\label{eqn:cov_mat_simultaneous_test_formulation_2}
H_{0,mk}: \sigma_{mk} = 0 \quad \text{versus} \quad H_{1,mk}: \sigma_{mk} \neq 0.
\end{equation}
The test statistic we construct is $\bar{T}_0 = \|\hat{S}-\Sigma_0\|_{\text{off}}$, which is an $\ell^\infty$ type statistic by taking the maximum magnitudes on the off-diagonal entries. Then $H_0$ is rejected if $\bar{T}_0 \ge a_{\hat{L}_0^*}(1-\alpha)$.

\begin{cor}[Asymptotic size of the simultaneous test: subgaussian observations]
\label{cor:cov_mat_simultaneous_test}
Let $\nu \ge 1$ and $\vX_i$ be iid mean zero random vectors such that $X_{ik} \sim  \text{subgaussian}(\nu^2)$ for all $k=1,\cdots,p$. Suppose that there exist constants $C_i > 0, i=1,\cdots,4,$ such that $\{\Gamma_g\}_{(j,k),(j,k)} \ge C_1$, $\|X_{1k}\|_4 \le C_2$, $\|X_{1k}\|_6 \le C_3 \nu^{1/3}$ and $\|X_{1k}\|_8 \le C_4 \nu^{1/2}$ for all $j,k=1,\cdots,p$. Let $\beta \in (0,1)$ and $\tau_* = \beta^{-1} a_{\hat{L}_0^*}(1-\alpha)$, where the bootstrap samples are generated with the covariance matrix kernel in (\ref{eqn:sample-covmat-kernel}). If $\nu^4 \log^7(np) \le C_5 n^{1-K}$ for some $K \in (0,1)$, then the above test based on $\bar{T}_0$ for (\ref{eqn:cov_mat_simultaneous_test_formulation}) has the size $\alpha + O(n^{-K/8})$; i.e. the family-wise error rate of the simultaneous test problem (\ref{eqn:cov_mat_simultaneous_test_formulation_2}) is asymptotically controlled at the level $\alpha$.
\end{cor}

From Corollary \ref{cor:cov_mat_simultaneous_test}, the test based on $\bar{T}_0$ is asymptotically exact of size $\alpha$ for subgaussian data. A similar result can be established for observations with polynomial moments. Due to the space limit, the details are omitted. \cite{changzhouzhouwang2016} proposed a similar test statistic for comparing the two-sample large covariance matrices. Their results (Theorem 1 in \cite{changzhouzhouwang2016}) are analogous to Corollary \ref{cor:cov_mat_simultaneous_test} in this paper in that no structural assumptions in $\Sigma$ are needed in order to obtain the asymptotic validity of both tests. However, we shall note that their assumptions (C.1), (C.2), and (C.3) on the non-degeneracy are stronger than our condition $\{\Gamma_g\}_{(j,k),(j,k)} \ge C_1$. For subgaussian observations $X_{ik} \sim \text{subgaussian}(\nu^2)$, (C.3) in \cite{changzhouzhouwang2016} assumed that $\min_{1 \le j \le k \le p} s_{jk} /\nu^4 \ge c$ for some constant $c > 0$, where $s_{jk}=\Var(X_{1j}X_{1k})$. If $\nu^2 \to \infty$, then \cite[Theorem 1]{changzhouzhouwang2016} requires that $s_{jk}$ for all $j,k=1,\cdots,p$ have to obey a uniform lower bound that diverges to infinity. For the covariance matrix kernel, since $g(\vx) = (\vx\vx^\top - \Sigma) / 2$, we only need that $\min_{j,k} s_{jk} \ge c$ for some fixed lower bound.

Next, we comment that a distinguishing feature of our bootstrap test from the $\ell^2$ test statistic \cite{chenzhangzhong2010} is that no structural assumptions are made on $F$ and we allow for the strong dependence in $\Sigma$. For example, consider again the elliptic distributions with the positive-definite $V=\varrho \vone_p \vone_p^\top + (1-\varrho) \Id_{p\times p}$ such that the covariance matrix $\Sigma$ is proportion to $V$. Then, we have 
\begin{eqnarray*}
\tr(V^4) &=& p[1+(p-1)\varrho^2]^2 + p(p-1)[2\varrho+(p-2)\varrho^2]^2, \\
\tr(V^2) &=& \varrho^2 p^2 + (1-\varrho^2) p.
\end{eqnarray*}
For any $\varrho\in(0,1)$, $\tr(V^4) / \tr^2(V^2) \to 1$ as $p \to \infty$. Therefore, the limiting distribution of the $\ell^2$ test statistic in \cite{chenzhangzhong2010} is no longer normal and its asymptotic distribution remains unclear.

Finally, the covariance matrix testing problem (\ref{eqn:cov_mat_simultaneous_test_formulation}) can be generalized further to nonparametric forms which can gain more robustness to outliers and the nonlinearity in the dependency structures. Let $U_\diamond = \E[h(\vX_1,\vX_2)]$ be the expectation of the random matrix associated with $h$ and $U_0$ be a known $p \times p$ matrix. Consider the testing problem
$$
H_0: U_\diamond = U_0  \quad \text{versus} \quad  H_1: U_\diamond \neq U_0.
$$
Then, the test statistic can be constructed as $\bar{T}_0 = \|U-U_0\|$ (or $\bar{T}'_0 = \|U-U_0\|_{\text{off}}$) and $H_0$ is rejected if $\bar{T}_0 \ge a_{\hat{L}_0^*}(1-\alpha)$ (or $\bar{T}'_0 \ge a_{\hat{L}_0^*}(1-\alpha)$), where the bootstrap samples are generated w.r.t. the kernel $h$. The above test covers Kendall's tau rank correlation matrix as a special case where $h$ is the bounded kernel defined in (\ref{eqn:kendaltau-kernel}).

\begin{cor}[Asymptotic size of the simultaneous test for Kendall's tau correlation matrix]
\label{cor:kendall_tau_mat_simultaneous_test}
Let $\vX_i$ be iid random vectors following the distribution $F$ in $\mathbb{R}^p$. Suppose that there exists a constant $C_1 > 0$ such that $\{\Gamma_g\}_{(j,k),(j,k)} \ge C_1$ for all $j,k=1,\cdots,p$. Let $\beta \in (0,1)$ and $\tau_* = \beta^{-1} a_{\hat{L}_0^*}(1-\alpha)$, where the bootstrap samples are generated with Kendall's tau rank correlation matrix kernel in (\ref{eqn:kendaltau-kernel}). If $\log^7(np) \le C_2 n^{1-K}$ for some $K \in (0,1)$, then the test based on $\bar{T}'_0$ has the size $\alpha + O(n^{-K/8})$.
\end{cor}

Therefore, the asymptotic validity of the bootstrap test for large Kendall's tau rank correlation matrix is obtained when $\log{p} = o(n^{1/7})$ without imposing structural and moment assumptions on $F$.

\section{Proofs of the main results}
\label{sec:proof}

The rest of the paper is organized as follows. In Section \ref{app:concentration_ineq_canonical-Ustat}, we first present a useful inequality for bounding the expectation of the sup-norm of the {\it canonical} U-statistics and then compare with an alternative simple data splitting bound by reducing to the moment bounding exercise for the sup-norm of sums of iid random matrices. We shall discuss several advantages of using the U-statistics approach by exploring the degeneracy structure. Section \ref{subsec:proof-of-GA} contains the proofs of the main results on Gaussian approximation and Section \ref{subsec:proof-of-bootstrap} proves the convergence rate of the Gaussian wild bootstrap. Proofs of the statistical applications are given in Section \ref{subsec:proof-of-applications}. Additional proofs and technical lemmas are given in the SM.

\subsection{A maximal inequality for canonical U-statistics}
\label{app:concentration_ineq_canonical-Ustat}

Before proving our main results, we first establish a maximal inequality of the canonical U-statistics of order two. The derived expectation bound is useful in controlling the size of the nonlinear and completely degenerate error term in the Gaussian approximation.

\begin{thm}[Expectation bound for canonical U-statistics]
\label{thm:expectation-bound}
Let $f: \mathbb{R}^p \times \mathbb{R}^p \to \mathbb{R}^{p\times p}$ be a symmetric and canonical U-statistic kernel of order two and $V={n\choose2}^{-1}\sum_{1\le i < j \le n} f(\vX_i,\vX_j)$ such that $\E\|f(\vX_1,\vX_2)\|<\infty$. Let  ${\vX'}_1^n$ be an independent copy of $\vX_1^n$, $M = \max_{1\le i < j \le n} (\|f(\vX_i,\vX_j)\| \vee \| f(\vX_i, \vX'_j)\|)$, $D_q = \max_{1 \le m,k \le p} (\E|f_{mk}|^q)^{1/q}$ for $q > 0$. If $2 \le p \le  \exp(b n)$ for some absolute constant $b>0$, then there exists an absolute constant $K > 0$ such that 
\begin{equation}
\label{eqn:expectation-bound-canonical}
\E \|V\| \le K (1 \vee b^{1/2}) \Big\{ \Big( {\log{p} \over n} \Big)^{3/2} \|M\|_4 + {\log{p} \over n} D_2 + \Big( {\log{p} \over n} \Big)^{5/4}  D_4 \Big\}.
\end{equation}
\end{thm}

Note that Theorem \ref{thm:expectation-bound} is non-asymptotic. As immediate consequences of Theorem \ref{thm:expectation-bound}, we can derive the rate of convergence of $\E \|V\|$ with kernels under the subexponential and uniform polynomial moment conditions.

\begin{cor}[Kernel with subexponential moment]
\label{cor:expectation-subexponential-kernel}
Let $B_n$ be a sequence of positive reals and $f$ be a symmetric and canonical kernel. Suppose that 
\begin{equation}
\label{eqn:subexponential-moment}
\max_{1 \le m,k \le p}  \E \left[ \exp(|f_{mk}| / B_n) \right] \le 2
\end{equation}
and $2 \le p \le  \exp(b n)$ for some absolute constant $b>0$. Then, there exists a constant $C(b)>0$ such that
\begin{equation}
\label{eqn:expectation-subexponential-kernel}
\E \|V\| \le C(b) B_n  \{  n^{-3/2} (\log{p})^{3/2} \log(np)   + n^{-1} \log{p} \}.
\end{equation}
\end{cor}

\begin{cor}[Kernel with uniform polynomial moment]
\label{cor:expectation-polynomial-kernel}
Let $q \ge 4$ and $B_n, B'_n$ be two sequences of positive reals. Let $f$ be a symmetric and canonical kernel. Suppose that
\begin{equation}
\label{eqn:uniform-polynomial-moment}
\E (\|f\| / B_n)^q + \max_{1 \le m,k \le p} \E (|f_{mk}| / B'_n)^4 \le 1
\end{equation}
and $2 \le p \le  \exp(b n)$ for some absolute constant $b>0$. Then, there exists a constant $C(b)>0$ such that
\begin{equation}
\label{eqn:expectation-polynomial-kernel}
\E \|V\| \le C(b) \{ B_n n^{-1} (\log{p})^{3/2} + B'_n n^{-1} \log{p}  \}.
\end{equation}
\end{cor}

\begin{rmk}[Comparison of Theorem \ref{thm:expectation-bound} with sums of iid random matrices]
\label{rmk:data-splitting-reduction}
For the $U$-statistic taking values in a Banach space $B$ (here, we consider $B=\mathbb{R}^{p \times p}$), its expected norm can also be bounded by the expected norm of sums of iid random matrices. Assume that $\E \|f(\vX_1,\vX_2)\| < \infty$ and let $m = [n/2]$ be the largest integer no greater than $n/2$. As noted in \cite{hoeffding1963}, we can write
\begin{equation}
\label{eqn:hoeffding_average}
m(V-\E V) = {1 \over n!} \sum_{\text{all } \pi_n} \sum_{i=1}^{m} [f(\vX_{\pi_n(2i-1)},\vX_{\pi_n(2i)}) - \E f],
\end{equation}
where the summation $\sum_{\text{all } \pi_n}$ is taken over all possible permutations $\pi_n : \{1,\cdots,n\} \to \{1,\cdots,n\}$. By Jensen's inequality and the iid assumption of $\vX_i$, we have
\begin{equation}
\label{eqn:reducing-Ustat-to-iid-sum}
\E \|V - \E V\| \le {1 \over m} \E \Big\|\sum_{i=1}^{m} [f(\vX_i,\vX_{i+m}) - \E f ] \Big\|,
\end{equation}
which can be viewed as a data splitting method into two halves. Under condition (\ref{eqn:subexponential-moment}),  it follows from Bernstein's inequality \cite[Proposition 5.16]{vershynin2010a} that
\begin{equation}
\label{eqn:reducing-Ustat-to-iid-sum-bernstein-bound}
\E \|V - \E V\| \le K_1 B_n ( \sqrt{\log(p)/n} + \log(p)/n ).
\end{equation}
So if $\log{p} \le b n^{1-\varepsilon}$ for some $\varepsilon \in (0,1)$, then $\E \| V - \E V \| \le C(b) B_n (\log(p)/n)^{1/2}$. There are three advantages of using the U-statistics approach in Theorem \ref{thm:expectation-bound} over the data splitting method into iid summands (\ref{eqn:reducing-Ustat-to-iid-sum}) and (\ref{eqn:reducing-Ustat-to-iid-sum-bernstein-bound}). For the canonical kernel, $\E V = 0$.

First, we can obtain from (\ref{eqn:expectation-subexponential-kernel}) that
$$
\E\|V\| \le C(b) B_n \{ n^{1-5\varepsilon/2} + n^{-3/2\varepsilon} (\log{n}) + n^{-\varepsilon} \}.
$$
Therefore, sharper rate is obtained by (\ref{eqn:expectation-subexponential-kernel}) when $\varepsilon \in (1/2,1)$ which covers the regime of valid Gaussian approximation and bootstrap. Under the scaling limit for the Gaussian approximation validity, i.e. $B_n^a \log^7(np) / n \le C n^{-K_2}$ for some $K_2 \in (0,1)$, where $a=2$ for the subexponential moment kernel and $a=4$ for the uniform polynomial moment kernel, it is easy to see that $\log{p} \le \log(np) \le C n^{(1-K_2)/7}$ so we can take $\varepsilon = (6+K_2)/7$.

Second and more importantly, the rate of convergence obtained by the Bernstein bound (\ref{eqn:reducing-Ustat-to-iid-sum-bernstein-bound}) does not lead to a convergence rate for the Gaussian approximation and the bootstrap method. Following the argument of proving Theorem \ref{thm:GA_master_thm}, one can only show with (\ref{eqn:reducing-Ustat-to-iid-sum-bernstein-bound}) that for any $\gamma\in(0,1)$ and $\ell_n=\log(np/\gamma) \ge 1$
$$
\rho(\bar{T}_0, \bar{Z}_0) \lesssim n^{-1/8} \ell_n^{7/8} (\tilde{D}_3^3+\tilde{D}_4^2)^{1/4} + \ell_n^{1/2} B_n^{1/2} + n^{-1/2} \ell_n^{3/2} u(\gamma) + \gamma,
$$
where the second term on the right-hand side does not converge to zero and $u(\gamma)$ is defined in the proof of Theorem \ref{thm:GA_master_thm}. The reason is that, although (\ref{eqn:reducing-Ustat-to-iid-sum-bernstein-bound}) is rate-exact for {\it non-degenerate} U-statistics, where the dependence of the rate in (\ref{eqn:reducing-Ustat-to-iid-sum-bernstein-bound}) on the sample size is $O(B_n n^{-1/2})$, it is not strong enough to control the size of the nonlinear correction term $\E\|\sqrt{n} V\|$ when $p \to \infty$ (recall that $W=\sqrt{n} V$). On the contrary, our bound in Theorem \ref{thm:expectation-bound} exploits the degeneracy structure of $V$ and the dependence of the rate in (\ref{eqn:expectation-bound-canonical}) on the sample size is $O(B_n n^{-1} + \|M\| n^{-3/2})$. Therefore, Theorem \ref{thm:expectation-bound} is more mathematically appealing in the degenerate case.

Third, the reduction to sums of iid random matrices in (\ref{eqn:reducing-Ustat-to-iid-sum}) does not give tight asymptotic distributions in order to make inference on the non-degenerate U-statistics. To illustrate this point, we consider the case $p=1$ and let $X_i$ be iid mean zero random variables with variance $\sigma^2$. Let $\zeta_1^2 = \Var(g(X_1))$ and $\zeta_2^2 = \Var(h(X_1, X_2))$. Assume that $\zeta_1^2 > 0$. So $\zeta_1^2$ is the variance of the leading projection term used in the Gaussian approximation and by Jensen's inequality $\zeta_1^2 \le \zeta_2^2$. Note that $\sqrt{n}(U-\E U) \stackrel{D}{\to} N(0, 4 \zeta_1^2)$ \cite[Theorem A, page 192]{serfling1980} and by the CLT $\sqrt{2/m} \sum_{i=1}^m [f(\vX_i,\vX_{i+m}) - \E f] \stackrel{D}{\to} N(0, 2 \zeta_2^2)$. Since in general $\zeta_2^2 \neq 2 \zeta_1^2$, the limiting distribution of the U-statistic is not the same as that in the data splitting method in view of (\ref{eqn:hoeffding_average}). For example, consider the non-degenerate covariance kernel $h(x_1,x_2)=(x_1-x_2)^2/2$ w.r.t. $F$ and $U = {n \choose 2}^{-1} \sum_{1 \le i<j \le n} h(X_i, X_j)$. Denote $\mu_4 = \E X_1^4$ and $g(x_1)=(x_1^2-\sigma^2)/2$. Then, $\zeta_2^2 = (\mu_4 + \sigma^4)/2$ and $\zeta_1^2 = (\mu_4-\sigma^4) / 4$ so that $\zeta_2^2 > 2 \zeta_1^2$ when $\sigma^2 > 0$. In particular, if $X_i$ are iid $N(0,\sigma^2)$, then $\mu_4=3\sigma^4$, $4 \zeta_1^2 = 2\sigma^4$, and $2 \zeta_2^2= 4\sigma^4$. Therefore, even though (\ref{eqn:reducing-Ustat-to-iid-sum-bernstein-bound}) gives better rate in the non-degenerate case, the reduction by splitting the data into the iid summands is not optimal for the Gaussian approximation purpose, which is the main motivation of this paper. In fact, $\zeta_2^2$ serves no purpose in the limiting distribution of $\sqrt{n}(U - \E U)$.
\qed
\end{rmk}

\subsection{Proof of results in Section \ref{sec:gaussian-approx}}
\label{subsec:proof-of-GA}

Let $\beta > 0$ and for $Z \in \mathbb{R}^{p \times p}$
$$F_\beta(Z) = \beta^{-1} \log(\sum_{m,k=1}^p \exp(\beta Z_{mk}))$$
be the smooth-max function for approximating $\bar{Z}_0=\max_{1 \le m,k \le p} Z_{mk}$. Denote $\pi_{mk}(Z) = \partial_{mk} F_\beta(Z)$ as the first-order partial derivative w.r.t. $Z_{mk}$ and ${\cal C}_b^3(\mathbb{R})$ the space of all bounded functions that are three times continuously differentiable on $\mathbb{R}$ with $\sup_{t \in \mathbb{R}} |\partial^q f(t)| < +\infty$ for $q = 0,1,2,3$. Let $\lambda_0 : \mathbb{R} \to [0,1]$ be such that: (i) $\lambda_0 \in {\cal C}_b^3(\mathbb{R})$ and $\sup_{s\in\mathbb{R}} |\partial^q \lambda_0(s)| \le K_0$ for some absolute constant $K_0>0$ and for $q = 0,1,2,3$; (ii) $\lambda_0(s)=1$ if $s\le0$ and $\lambda_0(s) = 0$ if $s \ge 1$. Let $e_\beta = 2 \beta^{-1} \log{p}$ and for any $t \in \mathbb{R}$, let $\lambda(s) = \lambda_0(\psi (s-t-e_\beta))$ for $\psi > 0$. Then $\sup_{s\in\mathbb{R}} |\partial^q \lambda(s)| \le K_0 \psi^q$ for $q = 0,1,2,3$. Here, $\beta$ and $\psi$ are smoothing parameters for approximating the max and indicator functions, respectively. In particular, we have
\begin{eqnarray}
\label{eqn:smoothing_max}
\bar{Z}_0 &\le& F_\beta(Z) \le \bar{Z}_0 + e_\beta \qquad \text{for all } Z \in \mathbb{R}^{p \times p}, \\
\label{eqn:smoothing_indicator}
\vone(s \le t + e_\beta) &\le& \lambda(s) \le \vone(s \le t+e_\beta+\psi^{-1}) \qquad \text{for all } s,t \in \mathbb{R}.
\end{eqnarray}
Define
\begin{equation*}
\tilde{D}_q = \max_{1 \le m,k \le p} (\E|g_{mk}|^q)^{1/q}, \quad D_q = \max_{1 \le m,k \le p} (\E|h_{mk}|^q)^{1/q}.
\end{equation*}
Clearly $\tilde{D}_q \le D_q$ for $q\ge1$. Let
$$u_x(\gamma) = \inf\{u\ge0 : \Prob(g_{mk}^2(\vX_i) \le u^2 \E g_{mk}^2 \text{ for all } m,k,i) \ge 1- \gamma\}$$
and $u_z(\gamma)$ be similarly defined with $g_{mk}(\vX_i)$ replaced by $Z_{i,mk}$, where $Z_i$ follow iid $N(\vzero, \Gamma_g)$. Put $u(\gamma)=u_x(\gamma) \vee u_z(\gamma)$. To prove Theorem \ref{thm:gaussian-approximation-nondegenarate-U-exp-tail} and \ref{thm:gaussian-approximation-nondegenarate-U-uniform-poly-tail}, we first need a smoothing lemma.

\begin{lem}
\label{lem:gaussian-approximation-nondegenarate-U-smooth}
Let $u>0, \gamma\in(0,1)$. Assume that $c_0 \le \tilde{D}_2 \le C_0$, $\sqrt{8} u \tilde{D}_2 \beta / \sqrt{n} \le 1$ and $u \ge u(\gamma)$. Then we have
\begin{equation}
\label{eqn:gaussian-approximation-nondegenarate-U-smooth}
\rho(\bar{T}_0,\bar{Z}_0) \le C_1 [ (\psi \E\|W\| + \Delta_v) + (e_\beta + \psi^{-1}) \sqrt{1 \vee \log(p \psi)} ],
\end{equation}
where
\begin{equation}
\label{eqn:bound_Delta_v}
\Delta_v = n^{-1/2} (\psi^3 + \psi^2 \beta + \psi \beta^2) \tilde{D}_3^3 +  (\psi^2 + \psi \beta) \bar{\varphi}(u) + \psi \bar{\varphi}(u) \sqrt{\log(p/\gamma)} + \gamma
\end{equation}
and $\bar{\varphi}(u) = C_2 u^{-1} \tilde{D}_4^2$. Here, $C_1, C_2>0$ are constants only depending on $c_0$ and $C_0$.
\end{lem}

\begin{thm}[Rate of convergence for Gaussian approximation]
\label{thm:GA_master_thm}
Let $\gamma \in (0,1)$ and $\ell_n = \log(pn/\gamma) \ge 1$. Assume $c_0 \le \tilde{D}_2 \le C_0$ and $2 \le p \le \exp(b n)$ for some absolute constants $c_0,C_0,b > 0$. Then, there exists a constant $C > 0$ depending only on $c_0,C_0,$ and $b$ such that
\begin{eqnarray}
\nonumber
\rho(\bar{T}_0, \bar{Z}_0) &\le& C \Big\{ n^{-1/8} \ell_n^{7/8} (\tilde{D}_3^3 + \tilde{D}_4^2)^{1/4} + n^{-1/2} \ell_n \|M\|_4^{1/2} + n^{-1/4} \ell_n^{3/4} D_2^{1/2} \\
\label{eqn:GA_master_thm}
&& \qquad + n^{-3/8} \ell_n^{7/8} D_4^{1/2} + n^{-1/2} \ell_n^{3/2} u(\gamma) + \gamma \Big\}.
\end{eqnarray}
\end{thm}

\begin{proof}[Proof sketch of Theorem \ref{thm:GA_master_thm}]
The proof is based on a delicate combination of Lemma \ref{lem:gaussian-approximation-nondegenarate-U-smooth} and Theorem \ref{thm:expectation-bound}. Since the proof of Theorem \ref{thm:GA_master_thm} is quite involved, here we only explain the main idea and give a sketch of the proof. All proof details can be found in the SM.

{\bf Main idea.} Lemma \ref{lem:gaussian-approximation-nondegenarate-U-smooth} gives a general rate of convergence for the Gaussian approximation with some unspecified smoothing parameters $\beta$ and $\psi$. The error bound in Lemma \ref{lem:gaussian-approximation-nondegenarate-U-smooth} involves three parts: (i) one from approximating the linear projection $\Delta_v$, (ii) one from the second-order canonical remainder $W$, and (iii) one from the smoothing errors of the max and indicator functions. Recall that $\beta$ is the degree of smoothing for the max function (\ref{eqn:smoothing_max}) and $\psi$ controls the approximation of the indicator function (\ref{eqn:smoothing_indicator}). For larger $\beta>0$ (or $\psi>0$), $F_\beta(Z)$ (or $\lambda(s)$) is closer to the non-smooth function $\max_{m,k} Z_{mk}$ (or $\vone(s \le t + e_\beta)$). Specifically, for larger $\beta$ and $\psi$, (iii) contributes less, while (i) and (ii) contribute more, to the error bound for the Gaussian approximation (\ref{eqn:gaussian-approximation-nondegenarate-U-smooth}). Hence, the key step is to optimize the Gaussian approximation error bound (\ref{eqn:gaussian-approximation-nondegenarate-U-smooth}) on the smoothing parameters $\beta$ and $\psi$ to find the best trade-off of the three parts. 

{\bf Step 1. Choose the smoothing parameters.} Note that both $\beta$ and $\psi$ depend on the truncation parameter $u$ in Lemma \ref{lem:gaussian-approximation-nondegenarate-U-smooth}. Therefore, the optimization problem eventually boils down to choose a proper threshold $u$. Once $u$ is chosen, then we shall first have a natural choice of $\beta$ to make the constraint of Lemma \ref{lem:gaussian-approximation-nondegenarate-U-smooth} on $\sqrt{8} u \tilde{D}_2 \beta / \sqrt{n} \le 1$ active in order to apply (\ref{eqn:gaussian-approximation-nondegenarate-U-smooth}); i.e $\beta = C n^{-1/2} u^{-1}$. So $\beta$ is a strictly decreasing function in $u$ and we need to choose a non-decreasing $\psi(u)$ to counter-balance the smoothing errors. Motivated from the proof of \cite[Theorem 2.2]{cck2013} in which only the linear part of $\bar{T}_0$ was dealt with, here we need to choose a larger $u$ because of the extra nonlinear term $W$. Since in the linear case of \cite[Theorem 2.2]{cck2013} $u = \max\{ u_0, u_1, u(\gamma) \}$ where $u_0=n^{3/8} \ell_n^{-5/8} \tilde{D}_4^{1/2}$ and $u_1=n^{3/8} \ell_n^{-5/8} \tilde{D}_3^{3/4}$, it is intuitive to choose $u=\max\{ u_0, u_1, u_2, u_3, u_4, u(\gamma) \}$ in our case such that $u_i$ takes the form $n^{\varepsilon_{i1}} \ell_n^{-\varepsilon_{i2}} T_i^{\varepsilon_{i3}}$, where $T_i =\|M\|_4, D_2, D_4$ for $i=2,3,4$ and the corresponding indicator smoothing parameters as
\begin{equation*}
\psi_i(u) = \min \{ n^{1/2-\varepsilon_{i1}} \ell_n^{\varepsilon_{i2}-1} T_i^{-\varepsilon_{i3}}, \quad \ell_n^{-1/6} (\bar\varphi(u))^{-1/3} \},
\end{equation*}
where $\bar\varphi(u)$ is defined in Lemma \ref{lem:gaussian-approximation-nondegenarate-U-smooth}. For $i=2,3,4$, we require different $\{\varepsilon_{ij}\}_{j =1,2,3}$, to balance the error bound (\ref{eqn:gaussian-approximation-nondegenarate-U-smooth}). Let $u_i^*$ be the solution for balancing the two components in $\psi_i(u)$. Then, $\psi_i(u)$ is strictly increasing when $u \in (0,u_i^*)$ and it is truncated to a constant level for $u \ge u_i^*$. 

{\bf Step 2. Calculate the error bound for the chosen parameters.} Now, we invoke Theorem \ref{thm:expectation-bound} to quantify the contributions of $\|M\|_4, D_2$, and $D_4$ to $\E\|W\|$. Combining (\ref{eqn:expectation-bound-canonical}) and (\ref{eqn:gaussian-approximation-nondegenarate-U-smooth}), it will be shown (after some algebraic manipulations on the two cases $0 < u < u^*$ and $u \ge u^*$, where $u^*=\max_{1 \le i \le 4} \{u_i^*\}$) that the optimal choice of $\{\varepsilon_{ij}\}_{j =1,2,3}$ in order to achieve the overall error bound
\begin{equation}
\label{eqn:desired-bound-general-GAR}
\rho(\bar{T}_0, \bar{Z}_0) \lesssim n^{-1/2} \ell_n^{3/2} u + \gamma
\end{equation}
is given by $u_2=\ell_n^{-1/2} \|M\|_4^{1/2}$, $u_3=n^{1/4} \ell_n^{-3/4} D_2^{1/2}$, and $u_4=n^{1/8} \ell_n^{-5/8} D_4^{1/2}$. Then, the explicit rate of convergence (\ref{eqn:GA_master_thm}) is immediate by substituting the choice of $u$ into (\ref{eqn:desired-bound-general-GAR}). 
\end{proof}

In the following proofs of Theorem \ref{thm:gaussian-approximation-nondegenarate-U-exp-tail}, \ref{thm:gaussian-approximation-nondegenarate-U-uniform-poly-tail} and \ref{thm:gaussian_wild_bootstrap_validity}, the constants of $\lesssim$ depend only on $C_1$ and $C_2$ in (GA.1) and (GA.2) in the sub-exponential kernel case and (GA.1') and (GA.2') in the uniform polynomial kernel case.

\begin{proof}[Proof of Theorem \ref{thm:gaussian-approximation-nondegenarate-U-exp-tail}]
Choose $\gamma = n^{-K_1}$ for some $K_1 \ge 1/8$. Let $\ell_n=\log(p n^{1+K_1})$. By Theorem \ref{thm:GA_master_thm} and Lemma \ref{lem:moment-bounds-subexponential-kernel}, we have
\begin{eqnarray*}
\rho(\bar{T}_0, \bar{Z}_0) &\lesssim& n^{-1/8} \ell_n^{7/8} B_n^{1/4} + n^{-1/2} \ell_n B_n^{1/2} \ell_n^{1/2} + n^{-1/4} \ell_n^{3/4} \\
&& \qquad + n^{-3/8} \ell_n^{7/8} B_n^{1/4} + n^{-1/2} \ell_n^{3/2} B_n \ell_n^2 + \gamma \\
&\lesssim& n^{-1/8} \ell_n^{7/8} B_n^{1/4} + n^{-1/2} \ell_n^{7/2} B_n + n^{-K_1}.
\end{eqnarray*}
By (\ref{eqn:subexponential-kernel-scaling-limit}), we have $\rho(\bar{T}_0, \bar{Z}_0) \lesssim n^{-\min(K_1, K/8)}$. Since $K \in (0,1)$, (\ref{eqn:gaussian-approximation-nondegenarate-U-exp-tail}) follows.
\end{proof}

\begin{proof}[Proof of Theorem \ref{thm:gaussian-approximation-nondegenarate-U-uniform-poly-tail}]
Choose $\gamma = n^{-K_1}$ for some $K_1 \in(0, K)$. Let $\ell_n=\log(p n^{1+K_1})$. By Theorem \ref{thm:GA_master_thm} and Lemma \ref{lem:moment-bounds-unifpoly-kernel} with $q=4$, we have
\begin{eqnarray*}
\rho(\bar{T}_0, \bar{Z}_0) &\lesssim& n^{-1/8} \ell_n^{7/8} B_n^{1/4} + n^{-1/2} \ell_n B_n^{1/2} n^{1/4} + n^{-1/4} \ell_n^{3/4} \\
&& \qquad + n^{-3/8} \ell_n^{7/8} B_n^{1/4} + n^{-1/2} \ell_n^{3/2} B_n n^{1+K_1 \over 4} + \gamma \\
&\lesssim& n^{-1/8} \ell_n^{7/8} B_n^{1/4} + \ell_n^{3/2} B_n n^{-{1-K_1 \over 4}} + n^{-K_1} \\
&\lesssim& n^{-\min({K \over 8}, {K-K_1 \over 4}, K_1)}.
\end{eqnarray*}
Choose $K_1 \in [K/8, K/2]$. Then, it follows that $\rho(\bar{T}_0, \bar{Z}_0) \lesssim n^{-K/8}$.
\end{proof}

\subsection{Proof of results in Section \ref{subsec:gassuain_wild_bootstrap}}
\label{subsec:proof-of-bootstrap}

\begin{proof}[Proof of Theorem \ref{thm:gaussian_wild_bootstrap_validity}] 

Let $\rho_\ominus(\alpha) = \Prob(\{\bar{T}_0 \le a_{\hat{L}_0^*}(\alpha) \} \ominus \{ \bar{L}_0 \le a_{\bar{Z}_0}(\alpha) \} )$, where $a_{\bar{Z}_0}(\alpha) = \inf\{t\in\mathbb{R} : \Prob(\bar{Z}_0 \le t) \ge \alpha\}$ is the $\alpha$-th quantile of $\bar{Z}_0$ and $A \ominus B = (A \setminus B) \cup (B \setminus A)$ is the symmetric difference of two subsets $A$ and $B$. Let $\ell_n = \log(np) \ge 1$. We first deal with the sub-exponential kernel moment condition. Assume (GA.1) and (GA.2). The proof contains two steps. \\

\underline{\bf \it Step 1.} Relate the bootstrap approximation to the Gaussian approximation. \\

Note that for all $\alpha \in (0,1)$
\begin{eqnarray*}
| \Prob(\bar{T}_0 \le a_{\hat{L}_0^*}(\alpha)) - \alpha | &\le& \rho(\bar{L}_0,\bar{Z}_0) + | \Prob(\bar{T}_0 \le a_{\hat{L}_0^*}(\alpha)) - \Prob(\bar{L}_0 \le a_{\bar{Z}_0}(\alpha)) | \\
&\le& \rho(\bar{L}_0,\bar{Z}_0) + \rho_\ominus(\alpha).
\end{eqnarray*}
Under conditions (GA.1) and (GA.2), following the proof of Theorem \ref{thm:GA_master_thm} and Theorem \ref{thm:gaussian-approximation-nondegenarate-U-exp-tail}, we have $\rho(\bar{L}_0,\bar{Z}_0) \lesssim n^{-K/8}$. Let $\Delta_1$ be defined in (\ref{eqn:gaussian_wild_bootstrap_Delta1}). By Lemma \ref{lem:bound-on-rho_ominus}, we can bound $\rho_\ominus(\alpha)$ as
\begin{equation}
\label{eqn:rho_ominus_bound}
\rho_\ominus(\alpha) \le 2 \left[ \rho(\bar{L}_0, \bar{Z}_0) + C v^{1/3} (\log{p})^{2/3}+ \Prob(\Delta_1 > v) \right] + C' \zeta_1 (\log{p})^{1/2} + 5 \zeta_2,
\end{equation}
provided that
\begin{eqnarray}
\label{eqn:quantile_cond_1}
\Prob(|\bar{T}_0 - \bar{L}_0| > \zeta_1) < \zeta_2, \hspace{0.6in} (\text{effect of } W), \\
\label{eqn:quantile_cond_2}
\Prob(\Prob_e(|\hat{L}_0^* - \bar{L}_0^*| > \zeta_1) > \zeta_2) < \zeta_2, \hspace{0.6in} (\text{effect of } \Delta_2),
\end{eqnarray}
where (\ref{eqn:quantile_cond_1}) is due to the nonlinear remainder of the U-statistics decomposition and (\ref{eqn:quantile_cond_2}) is due to the estimation error of $g(\vX_i)$. Choose $\zeta_1 = C n^{-K_1} \ell_n^{-1/2}$ and $\zeta_2 = C n^{-K_2}$ for some $K_1,K_2>0$ whose values are to be determined in Step 2. Then, $\zeta_1 \sqrt{\log{p}} \le C n^{-K_1}$. Choose $v=\ell_n^{-1/2} (\E \Delta_1)^{3/4}$. By Lemma \ref{lem:gaussian_wild_bootstrap_Delta1_exp_moment}, $v \lesssim \ell_n^{-1/2} \varpi_1(n,p)^{3/4}$, where $\varpi_1(n,p)$ is defined in (\ref{eqn:varpi_1}). By Markov's inequality, 
$$\Prob(\Delta_1 > v) \le \E(\Delta_1) / v = \ell_n^{1/2} (\E \Delta_1)^{1/4} \lesssim \ell_n^{1/2} \varpi_1(n,p)^{1/4},$$
where under (GA.2) $\varpi_1(n,p) \le n^{-1/2} \ell_n^{1/2} B_n + n^{-1} \ell_n^3 B_n^2 \lesssim n^{-1/2} \ell_n^{1/2} B_n$. Therefore, we have
$$
v \lesssim \ell_n^{-1/2} (n^{-1/2} \ell_n^{1/2} B_n)^{3/4} \lesssim \ell_n^{-2} n^{-3K/8}.
$$
So it follows that $(v^{1/2} \log{p})^{2/3} \lesssim n^{-K/8}$
and
$$
\Prob(\Delta_1 > v) \lesssim \ell_n^{1/2} (n^{-1/2} \ell_n^{1/2} B_n)^{1/4} \lesssim n^{-K/8}.
$$
Substituting those bounds into (\ref{eqn:rho_ominus_bound}), we get
\begin{equation}
\label{eqn:rho_ominus_bound_simplified}
\rho_\ominus(\alpha) \lesssim n^{-K/8} + n^{-K_1} + n^{-K_2} \lesssim n^{-\min\{K/8, \; K_1, \; K_2\}}.
\end{equation}

\vspace{0.1in}

\underline{\bf \it Step 2.} Verify the constraints (\ref{eqn:quantile_cond_1}) and (\ref{eqn:quantile_cond_2}).  \\

Since max is a 1-Lipschitz function and by Markov's inequality and Lemma \ref{lem:gaussian_wild_bootstrap_W_exp_moment}, we have $\Prob(|\bar{T}_0 - \bar{L}_0| > \zeta_1) \le \E \|W\| / \zeta_1 \le C \varpi_3(n,p) \zeta_1^{-1}$, where $\varpi_3(n,p)$ is defined in (\ref{eqn:bound_W_exp_moment}). Then, (\ref{eqn:quantile_cond_1}) is fulfilled whenever $C \zeta_1 \zeta_2 \ge \varpi_3(n,p)$ for some constant $C>0$. Next, we deal with (\ref{eqn:quantile_cond_2}). Note that
$$
|\hat{L}_0^* - \bar{L}_0^*| \le {1\over\sqrt{n}} \max_{1 \le m,k \le p} \Big| \sum_{i=1}^n [\hat{g}_{i,mk} - g_{mk}(\vX_i)] e_i \Big|.
$$
By the argument leading to (\ref{eqn:V_1}), conditional on $\vX_1^n$ and ${\vX'}_1^n$, we have
\begin{eqnarray*}
\E_e |\hat{L}_0^* - \bar{L}_0^*| \le {K_3 \over\sqrt{n}} (\log{p})^{1/2} \max_{1 \le m,k \le p} \Big\{ \sum_{i=1}^n [\hat{g}_{i,mk} - g_{mk}(\vX_i)]^2 \Big\}^{1/2},
\end{eqnarray*}
from which it follows that
$$
\E |\hat{L}_0^* - \bar{L}_0^*| \le K_3 (\log{p})^{1/2} \E (\Delta_2^{1/2})
$$
and $\Delta_2$ is defined in (\ref{eqn:gaussian_wild_bootstrap_Delta2}). By Markov's inequality (also conditional on $\vX_1^n$ and ${\vX'}_1^n$), $\Prob_e(|\hat{L}_0^* - \bar{L}_0^*| > \zeta_1) \le \zeta_1^{-1} \E_e |\hat{L}_0^* - \bar{L}_0^*|$ so that
\begin{eqnarray*}
\Prob(\Prob_e(|\hat{L}_0^* - \bar{L}_0^*| > \zeta_1) > \zeta_2) &\le& \Prob(\E_e |\hat{L}_0^* - \bar{L}_0^*| > \zeta_1 \zeta_2) \\
&\le& {\E |\hat{L}_0^* - \bar{L}_0^*| \over \zeta_1 \zeta_2} \le {K_3 (\log{p})^{1/2} \E (\Delta_2^{1/2}) \over \zeta_1 \zeta_2}.
\end{eqnarray*}
By Lemma \ref{lem:gaussian_wild_bootstrap_Delta2_exp_moment}, (\ref{eqn:quantile_cond_2}) is fulfilled when $\varpi_2(n,p) \le C \zeta_1 \zeta_2^2 / \sqrt{\log{p}}$. Recall that we need to check
\begin{eqnarray}
\label{eqn:quantile_cond_1_explict}
\varpi_3(n,p) &\le& C \zeta_1 \zeta_2 = C {n^{-(K_1+K_2)} \over \sqrt{\log(np)}}, \\
\label{eqn:quantile_cond_2_explicit}
\varpi_2(n,p) &\le& C {\zeta_1 \zeta_2^2 \over \sqrt{\log{p}}} = C {n^{-(K_1+2K_2)} \over \sqrt{\log{p}} \sqrt{\log(np)}}.
\end{eqnarray}
By Lemma \ref{lem:gaussian_wild_bootstrap_Delta2_exp_moment} and \ref{lem:gaussian_wild_bootstrap_W_exp_moment}
\begin{eqnarray*}
\varpi_3(n,p) &\le& n^{-1} (\log(np))^{5/2} B_n + n^{-1/2} (\log{p}) + n^{-3/4} (\log(np))^{5/4} B_n^{1/2} , \\
\varpi_2(n,p) &\le& n^{-1/2} (\log(np))^{3/2} B_n + n^{-1} (\log(np))^2 B_n.
\end{eqnarray*}
So sufficient conditions for (\ref{eqn:quantile_cond_1_explict}) and (\ref{eqn:quantile_cond_2_explicit}) are given by
\begin{eqnarray}
\label{eqn:quantile_cond_1_explict_suffcond}
n^{-1} \ell_n^3 B_n + n^{-1/2} \ell_n^{3/2} + n^{-3/4} \ell_n^{7/4} B_n^{1/2} &\lesssim& n^{-(K_1+K_2)}, \\
\label{eqn:quantile_cond_2_explicit_suffcond}
n^{-1/2} \ell_n^{5/2} B_n + n^{-1} \ell_n^3 B_n &\lesssim& n^{-(K_1+2K_2)}.
\end{eqnarray}
Under (GA.2), since the LHS of (\ref{eqn:quantile_cond_1_explict_suffcond}) is bounded by $n^{-K} + n^{-K/2} + n^{-3K/4} \le 3 n^{-K/2}$ and the LHS of (\ref{eqn:quantile_cond_2_explicit_suffcond}) is bounded by $2 n^{-K/2}$, we deduce that $K_1$ and $K_2$ must satisfy the constraint $K_1 + 2 K_2 \le K/2$. Take $K_1=K_2=K/6$. Then we obtain from (\ref{eqn:rho_ominus_bound_simplified}) that
$$
| \Prob(\bar{T}_0 \le a_{\hat{L}_0^*}(\alpha)) - \alpha |  \lesssim n^{-\min\{K/8, \; K/6\}} = n^{-K/8}.
$$

Similar argument applies to the kernel with uniform polynomial moment, so we only sketch the proof in this case. Assume (GA.1') and (GA.2'). \\

\underline{\bf \it Step 1'.} We shall use the same $v = \ell_n^{-1/2} (\E \Delta_1)^{3/4}$ as in the sub-exponential kernel case. Under (GA.1') and (GA.2'), by Lemma \ref{lem:gaussian_wild_bootstrap_Delta1_unifpoly_moment} with $q = 4$, we have $\varpi_1(n,p) \lesssim n^{-1/2} \ell_n B_n^2$ and $v \lesssim \ell_n^{-1/2} (n^{-1/2} \ell_n B_n^2)^{3/4} = (n^{-1} \ell_n^{2/3} B_n^4)^{3/8}$. Therefore, $(v^{1/2} \log{p})^{2/3} \lesssim (n^{-1} \ell_n^6 B_n^4)^{1/8} \lesssim n^{-K/8}$ and
$$
\Prob(\Delta_1 > v) \lesssim \ell_n^{1/2} (n^{-1/2} \ell_n B_n^2)^{1/4} \lesssim n^{-K/8}.
$$
Then, by (\ref{eqn:rho_ominus_bound}), we get $\rho_\ominus(\alpha) \lesssim n^{-\min\{K/8, \; K_1, \; K_2\}}.$

\vspace{0.1in}

\underline{\bf \it Step 2'.} By Lemma \ref{lem:gaussian_wild_bootstrap_Delta2_unifpoly_moment} and \ref{lem:gaussian_wild_bootstrap_W_unifpoly_moment} with $q=4$, we must verify the constraints
\begin{eqnarray*}
n^{-1/2} \ell_n^2 B_n + n^{-1/2} \ell_n^{3/2} + n^{-3/4} \ell_n^{7/4} B_n^{1/2} &\lesssim& n^{-(K_1+K_2)}, \\
n^{-1/4} \ell_n^{3/2} B_n + n^{-1/2} \ell_n^2 B_n &\lesssim& n^{-(K_1+2K_2)}.
\end{eqnarray*}
Under (GA.2'), a sufficient condition for the last two inequalities to hold is $K_1 + 2 K_2 \le K/4$. Then, we can take $K_1=K_2=K/12$ to get $\rho_\ominus(\alpha) \lesssim n^{-K/12}$ and thus (\ref{eqn:gaussian_wild_bootstrap_validity_unifpoly}).
\end{proof}

\subsection{Proof of results in Section \ref{sec:stat_apps}}
\label{subsec:proof-of-applications}

\begin{proof}[Proof of Theorem \ref{thm:thresholded_cov_mat_rate_adaptive}]
Let $\beta \in (0,1)$ and $\tau_\diamond = \beta^{-1} \|\hat{S}-\Sigma\|$. By the subgaussian assumption and Lemma \ref{lem:moment-bounds-gaussian-obs}, it is easy to verify that there is a large enough constant $C>0$ depending only on $C_2,C_3,C_4$ such that
\begin{equation}
\label{eqn:check_GA1}
\max_{\ell=0,1,2} \E[|h_{mk}|^{2+\ell} / (C \nu^{2\ell})] \vee \E[\exp(|h_{mk}| / \nu^2)] \le 2,
\end{equation}
where $h$ is the covariance matrix kernel in (\ref{eqn:sample-covmat-kernel}). Since $\{\Gamma_g\}_{(j,k),(j,k)} \ge C_1$ for all $j,k=1,\cdots,p$ and $\nu^4 \log^7(np) \le C_5 n^{1-K}$, by Theorem \ref{thm:gaussian_wild_bootstrap_validity}, we have $ \|\hat{S}-\Sigma\| \le a_{\hat{L}_0^*}(1-\alpha)$ with probability at least $1-\alpha-Cn^{-K/8}$, where $C > 0$ is a constant depending only on $C_i,i=1,\cdots,5$. Therefore, $\Prob(\tau_\diamond \le \tau_*) \ge 1- \alpha -Cn^{-K/8}$ and the rest of the proof is restricted to the event $\{\tau_\diamond \le \tau_*\}$. By the decomposition
\begin{eqnarray*}
\rho(\hat\Sigma(\tau_*)-\Sigma) &\le& \rho(\hat\Sigma(\tau_*)-T_{\tau_*}(\Sigma)) + \rho(T_{\tau_*}(\Sigma)-\Sigma) \\
&\le& I + II + III + \tau_*^{1-r} \zeta_p,
\end{eqnarray*}
where $T_\tau(\Sigma) = \{\sigma_{mk} \vone\{|\sigma_{mk}| > \tau\}\}_{m,k=1}^p$ is the resulting matrix of the thresholding operator on $\Sigma$ and
\begin{eqnarray*}
I &=& \max_m \sum_k |\hat{s}_{mk}| \vone\{|\hat{s}_{mk}| > \tau_*, |\sigma_{mk}| \le \tau_*\}, \\
II &=& \max_m \sum_k |\sigma_{mk}| \vone\{|\hat{s}_{mk}| \le \tau_*, |\sigma_{mk}| > \tau_*\}, \\
III &=& \max_m \sum_k |\hat{s}_{mk}-\sigma_{mk}| \vone\{|\hat{s}_{mk}| > \tau_*, |\sigma_{mk}| > \tau_*\}.
\end{eqnarray*}
Note that on the event $\{\tau_\diamond \le \tau_*\}$, $\max_{m,k} |\hat{s}_{mk}-\sigma_{mk}| \le \beta \tau_*$. Since $\Sigma \in {\cal G}(r, C_0, \zeta_p)$, we can bound
$$
III \le (\beta \tau_*) (\tau_*^{-r} \zeta_p) = \beta \tau_*^{1-r} \zeta_p.
$$
By triangle inequality,
\begin{eqnarray*}
II &\le& \max_m \sum_k |\hat{s}_{mk}-\sigma_{mk}| \vone\{|\sigma_{mk}| > \tau_*\} + \max_m \sum_k |\hat{s}_{mk}| \vone\{|\hat{s}_{mk}| \le \tau_*, |\sigma_{mk}| > \tau_*\} \\
&\le& (\beta \tau_*) (\tau_*^{-r} \zeta_p) + \tau_* (\tau_*^{-r} \zeta_p) = (1+\beta) \tau_*^{1-r} \zeta_p.
\end{eqnarray*}
Let $\eta \in (0,1)$. We have $I \le IV + V + VI$, where
\begin{eqnarray*}
IV &=& \max_m \sum_k |\sigma_{mk}| \vone\{|\hat{s}_{mk}| > \tau_*, |\sigma_{mk}| \le \tau_*\}, \\
V &=&  \max_m \sum_k |\hat{s}_{mk}-\sigma_{mk}| \vone\{|\hat{s}_{mk}| > \tau_*,  |\sigma_{mk}| \le \eta \tau_*\},  \\
VI &=& \max_m \sum_k |\hat{s}_{mk}-\sigma_{mk}| \vone\{|\hat{s}_{mk}| > \tau_*,  \eta \tau_* < |\sigma_{mk}| \le \tau_*\} .
\end{eqnarray*}
Clearly, $IV \le \tau_*^{1-r} \zeta_p$. On the indicator event of $V$, we observe that 
$$
\beta \tau_* \ge |\hat{s}_{mk}-\sigma_{mk}| \ge |\hat{s}_{mk}|-|\sigma_{mk}| > (1-\eta) \tau_*.
$$
Therefore, $V=0$ if $\eta + \beta \le 1$. For $VI$, we have
$$
VI \le (\beta\tau_*) (\eta \tau_*)^{-r} \zeta_p.
$$
Collecting all terms, we conclude that
$$
\rho(\hat\Sigma(\tau_*)-\Sigma) \le (3+2\beta+\eta^{-r} \beta) \zeta_p \tau_*^{1-r} + V.  
$$
Then (\ref{eqn:thresholded_cov_mat_rate_spectral_adaptive}) follows from the choice $\eta=1-\beta$. The Frobenius norm rate (\ref{eqn:thresholded_cov_mat_rate_F_adaptive}) can be established similarly. Details are omitted.

Next, we prove (\ref{eqn:tau_*-bound_subgaussian}). Let $\Phi(\cdot)$ denote the cdf of the standard Gaussian random variable. By the union bound, we have for all $t > 0$
$$
\Prob_e \Big({2 \over \sqrt{n}} \Big\| \sum_{i=1}^n \hat{g}_i e_i \Big\| \ge t \Big) \le 2 p^2 \Big[1-\Phi \Big( {t \over \bar\xi} \Big) \Big],
$$
where $\bar\xi = \max_{1 \le m,k \le p} \xi_{mk}$ and $\xi_{mk}^2 = 4 n^{-1} \sum_{i=1}^n \hat{g}_{i,mk}^2$. Let $\tilde\tau = n^{-1/2} \beta^{-1} \bar\xi \Phi^{-1}(1-\alpha/(2p^2))$; then $\tau_* \le \tilde\tau$. Since $\Phi^{-1}(1-\alpha/(2p^2)) \asymp (\log{p})^{1/2}$, we have $\E[\tau_*] \le C' \beta^{-1} \E[\bar\xi] (\log(p)/n)^{1/2}$, where $C' > 0$ is a constant only depending on $\alpha$. Now, we bound $\E[\bar\xi]$. Let $\Delta = n^{-1} \max_{m,k} \sum_{i=1}^n [\hat{g}_{i,mk}-g_{mk}(\vX_i)]^2$. Then,
$$
\bar\xi^2 \le {8\over n} \max_{m,k} \sum_{i=1}^n g_{mk}^2(\vX_i) + 8 \Delta.
$$
By Lemma \ref{lem:gaussian_wild_bootstrap_Delta2_exp_moment} and recall that $\nu^4 \log^7(np) \le C_5 n^{1-K}$,
$$
\E[\Delta^{1/2}] \le C n^{-1/2}  (\log(pn))^{3/2}  \nu^2,
$$
where $C > 0$ is constant depending only on $C_i,i=1,\cdots,5$. By \cite[Lemma 9]{cck2014b} and Pisier's inequality \cite[Lemma 2.2.2]{vandervaartwellner1996}, we have
\begin{eqnarray*}
\E[\max_{m,k} \sum_{i=1}^n g_{mk}^2(\vX_i)] &\le& K_1 \Big\{ \max_{m,k} \E[\sum_{i=1}^n g_{mk}^2(\vX_i)] + (\log{p}) \E[\max_{m,k}\max_{i \le n} g_{mk}^2(\vX_i)] \Big\} \\
&\le& C \Big\{ n + (\log(np))^3 \nu^4 \Big\}.
\end{eqnarray*}
By Jensen's inequality, we get
$$
\E[\bar\xi] \le C \Big\{ 1 + (\nu^4 \log^3(np) / n)^{1/2} \Big\} \le C.
$$
Then, we conclude that $\E[\tau_*] \le C(\alpha,C_1,\cdots,C_5) \beta^{-1} (\log(p)/n)^{1/2}$.
\end{proof}

\begin{proof}[Proof of Theorem \ref{thm:thresholded_cov_mat_rate_adaptive_polymom}]
The proof is similar to that of Theorem \ref{thm:thresholded_cov_mat_rate_adaptive} and we only sketch the differences. By the assumptions and Lemma \ref{lem:moment-bounds-unifpoly-obs}, we have
$$
\max_{\ell=0,1,2} \E[|h_{mk}|^{2+\ell} / (C \nu^{2\ell})] \vee \E [\|h\| / (2\nu^2)]^4 \le 1.
$$
By Theorem \ref{thm:gaussian_wild_bootstrap_validity}, we have $ \|\hat{S}-\Sigma\| \le a_{\hat{L}_0^*}(1-\alpha)$ with probability at least $1-\alpha-Cn^{-K/12}$, where $C > 0$ is constant depending only on $C_i, i =1,\cdots,5$. So (\ref{eqn:thresholded_cov_mat_rate_spectral_adaptive}) and (\ref{eqn:thresholded_cov_mat_rate_F_adaptive}) follow. Note that
\begin{eqnarray*}
\E[\max_{m,k} \sum_{i=1}^n g_{mk}^2(\vX_i)] &\le& K_1 \Big\{ \max_{m,k} \E[\sum_{i=1}^n g_{mk}^2(\vX_i)] + (\log{p}) \E[\max_{m,k}\max_{i \le n} g_{mk}^2(\vX_i)] \Big\} \\
&\le& C \Big\{ n + (\log{p}) n^{1/2} \nu^4 \Big\}.
\end{eqnarray*}
Then, under the assumption that $\nu^8 \log^7(np) \le C_5 n^{1-K}$, it follows from Lemma \ref{lem:gaussian_wild_bootstrap_Delta2_unifpoly_moment} that
$$
\E[\bar\xi] \le C \Big\{ 1 + {\nu^2 \log^{1/2}(p) \over n^{1/4}} + {\nu^2 \log^{1/2}(np) \over n^{1/4}} + {\nu^2 \log(np) \over n^{1/2}} \Big\} \le C.
$$
Therefore, we get $\E[\tau_*] \le C(\alpha, C_1,\cdots,C_5) \beta^{-1} (\log(p)/n)^{1/2}$.
\end{proof}

\section*{Acknowledgments}
The author would like to thank two anonymous referees, an Associate Editor, and the Co-Editor Tailen Hsing for their many constructive comments that lead to the significant improvements of this paper. The author is also grateful to Stephen Portnoy (UIUC), Xiaofeng Shao (UIUC), and Wei Biao Wu (University of Chicago) for their helpful discussions.

\begin{supplement}[id=suppA]
  \sname{Supplemental Materials to}
  \stitle{``Gaussian approximation for the sup-norm of high-dimensional matrix-variate U-statistics and its applications"}
  \slink[doi]{}
  \sdescription{This supplemental file contains the additional proofs, technical lemmas, and simulation results.}
\end{supplement}

\bibliographystyle{imsart-number}
\bibliography{max-norm-cov}

\newpage

\begin{center}
{\Large \bf Supplemental Materials to ``Gaussian approximation for the sup-norm of high-dimensional matrix-variate U-statistics and its applications"} \\

\vspace{0.1in}

Xiaohui Chen \\

\vspace{0.1in}

University of Illinois at Urbana-Champaign
\end{center}

\vspace{0.2in}

\appendix

Let $q > 0$ and $g_i := g(\vX_i) = \E[h(\vX_i,\vX') | \vX_i] - \E h$ is the H\'ajek projection in (\ref{eqn:hoeffding-decomp-order-one}). Write $g_{mk} = g_{mk}(\vX)$ for $m,k=1,\cdots,p$. Recall the definitions
\begin{eqnarray*}
&& D_q = \max_{m,k} (\E|h_{mk}|^q)^{1/q}, \qquad \tilde{D}_q = \max_{m,k} (\E|g_{mk}|^q)^{1/q}, \\
&& M = \max_{1 \le i < j \le n} (\|h(\vX_i,\vX_j)\| \vee \|h(\vX_i,\vX'_j)\|).
\end{eqnarray*}
We shall use $K, K_0, K_1,\cdots$ to denote positive absolute constants, and $C, C', C_0, C_1, \cdots$ and $c,c',c_0,c_1,\cdots$ to denote positive finite constants whose values are independent of $n$ and $p$ and may vary at different places. We write $a \lesssim b$ if $a \le C b$ for some constant $C$, and $a \asymp b$ if $a \lesssim b$ and $b \lesssim a$.

\section{Auxiliary lemmas in the proof of Section 2}
\label{appendix:technical_lemmas}

Appendix \ref{appendix:technical_lemmas} contains additional technical lemmas that are used to prove the main results of this paper.

Recall that $u(\gamma)=u_x(\gamma) \vee u_z(\gamma)$, where
$$u_x(\gamma) = \inf\{u\ge0 : \Prob(g_{mk}^2(\vX_i) \le u^2 \E g_{mk}^2 \text{ for all } m,k,i) \ge 1- \gamma\}$$
and $u_z(\gamma)$ is similarly defined with $g_{mk}(\vX_i)$ replaced by $Z_{i,mk}$. Here, $Z_i$ follow iid $N(\vzero, \Gamma_g)$.

\begin{lem}[Moment bounds for sub-exponential kernel]
\label{lem:moment-bounds-subexponential-kernel}
Let $K_1>0$ and $\gamma = n^{-K_1}$. If (GA.1) holds for all $m,k=1,\cdots,p$, then we have
\begin{eqnarray*}
\tilde{D}_2 &\le& D_2 \le 2^{1/2}, \\
\tilde{D}_3 &\le& D_3 \le 2^{1/3} B_n^{1/3}, \\
\tilde{D}_4 &\le& D_4 \le 2^{1/4} B_n^{1/2}, \\
\|M\|_q &\le& K q!B_n \log(np) \qquad \forall q \ge 1, \\
u(\gamma) &\le& C B_n \log^2(np),
\end{eqnarray*}
where $C > 0$ is a constant depending only on $C_1$ in (GA.1).
\end{lem}

\begin{proof}[Proof of Lemma \ref{lem:moment-bounds-subexponential-kernel}]
The bounds on $\tilde{D_\ell}$ and $D_\ell$ for $\ell = 2,3,4,$ are obvious under (\ref{eqn:subexponential-kernel-moment-condition}). Since $\|h_{mk}(\vX,\vX') / B_n\|_{\psi_1} \le 1,$
where $\|\cdot\|_{\psi_q}$ is the Orlicz norm for $\psi_q(x) = \exp(x^q) - 1$ for $q \ge 1$ and $x \ge 0$, we have by Pisier's inequality \cite[Lemma 2.2.2]{vandervaartwellner1996} that
$$
\left\| {M \over B_n} \right\|_{\psi_1} \le K \log(np) \max_{m,k} \max_{i<j}  \left( \left\| {h_{mk}(\vX_i,\vX_j) \over B_n} \right\|_{\psi_1} + \left\| {h_{mk}(\vX_i,\vX'_j) \over B_n} \right\|_{\psi_1} \right).
$$
Since $\|M\|_q \le q! \left\| M \right\|_{\psi_1}$ for any $q \ge 1$, we get $\|M\|_q \le K q! B_n \log(np).$ By \cite[Lemma 2.2]{cck2013}, we have
$$
u(\gamma) \le C \max\{M' \log(1+n/\gamma), \; \tilde{D}_2 \sqrt{\log(np/\gamma)}\},
$$
where $C > 0$ is a constant only depending on $C_1$ in (GA.1) and $M'$ satisfies $\E [\exp(\max_{1\le m,k \le p} |g_{mk}| / M')] \le 2$. By Pisier's and Jensen's inequalities, and using (\ref{eqn:subexponential-kernel-moment-condition}), we have
$$
\left\| \max_{m,k} |g_{mk}| / B_n \right\|_{\psi_1} \le K (\log{p}) \max_{m,k} \left\| g_{mk} / B_n \right\|_{\psi_1} \le K \log{p}.
$$
Therefore, we can take $M' = K B_n \log{p}$ for a large enough absolute constant $K > 0$. Then, we have
$u(\gamma) \le C B_n \log^2(np)$ for $\gamma = n^{-K_1}$.
\end{proof}

\begin{lem}[A moment bound for subgaussian observations]
\label{lem:moment-bounds-gaussian-obs}
Suppose that $\vX_i$ are iid mean zero random vectors such that $X_{im} \sim \text{subgaussian}(\nu^2)$. If $h$ is the covariance matrix kernel in (\ref{eqn:sample-covmat-kernel}), then we have for all $m,k=1,\cdots,p,$
\begin{eqnarray*}
\E[\exp(|h_{mk}| / \nu^2)] \le 2,
\end{eqnarray*}
i.e. $\|h_{mk}\|_{\psi_1} \le \nu^2$.
\end{lem}

\begin{proof}[Proof of Lemma \ref{lem:moment-bounds-gaussian-obs}]
The lemma follows from
\begin{eqnarray*}
\E \left[\exp \left({|h_{mk}(\vX_1,\vX_2)| \over \nu^2} \right) \right] &=& \E\left[\exp \left({1\over2} {|X_{1m}-X_{2m}| \over  \nu} {|X_{1k}-X_{2k} | \over \nu} \right) \right] \\
&\le& \E \left[ \exp \left(  {(X_{1m}-X_{2m})^2 \over 4 \nu^2 } +  {(X_{1k}-X_{2k})^2 \over 4 \nu^2 }  \right) \right] \\
&\le& \max_{1\le m \le p} \E \left[ \exp \left( {(X_{1m}-X_{2m})^2 \over 2 \nu^2 } \right)\right] \\
&\le& \max_{1\le m \le p} \E \left[ \exp \left( {X_{1m}^2 + X_{2m}^2 \over \nu^2 } \right)\right] \\
&\le& \max_{1\le m \le p} \left\{ \E \left[ \exp \left( {X_{1m}^2 \over \nu^2 } \right)\right] \right\}^2 \le 2,
\end{eqnarray*}
where we used the elementary inequality $|ab| \le (a^2+b^2)/2$ in the second step, the Cauchy-Schwarz inequality in the third step, $(a-b)^2 \le 2(a^2+b^2)$ in the fourth step, the iid assumption in the fifth step, and the assumption that $X_{im} \sim \text{subgaussian}(\nu^2)$ in the last step. 
\end{proof}

\begin{lem}[Moment bounds for uniform polynomial kernel]
\label{lem:moment-bounds-unifpoly-kernel}
Let $q \ge 4$, $K_1>0$ and $\gamma = n^{-K_1}$. If $\E g_{mk}^2 \ge C_1$ and 
\begin{equation}
\label{eqn:unifpoly-kernel-moment-condition}
\max_{\ell=0,1,2} \E(|h_{mk}|^{2+\ell} / B_n^\ell) \vee \E[(\|h\| / B_n)^q] \le 1
\end{equation}
for all $m,k=1,\cdots,p$, then we have
\begin{eqnarray*}
\tilde{D}_2 &\le& D_2 \le 1, \\
\tilde{D}_3 &\le& D_3 \le B_n^{1/3}, \\
\tilde{D}_4 &\le& D_4 \le B_n^{1/2}, \\
\|M\|_q &\le& 2 B_n n^{2/q}, \\
u(\gamma) &\le& C \max\{ B_n n^{1+K_1 \over q}, \log^{1/2}(np) \},
\end{eqnarray*}
where $C > 0$ is a constant depending only on $C_1$.
\end{lem}

\begin{proof}[Proof of Lemma \ref{lem:moment-bounds-unifpoly-kernel}]
The bounds on $\tilde{D_\ell}$ and $D_\ell$ for $\ell = 2,3,4,$ are similar to those in Lemma \ref{lem:moment-bounds-subexponential-kernel}. By \cite[Lemma 2.2.2]{vandervaartwellner1996} and (\ref{eqn:unifpoly-kernel-moment-condition}), we have
$$
\left\| \max_{i<j}  {\|h(\vX_i,\vX_j)\| \over B_n} \right\|_q \le n^{2/q} \max_{i<j} \left\| {\|h\| \over B_n}\right\|_q \le n^{2/q}.
$$
Same bound holds for $\max_{i<j} \|h(\vX_i,\vX'_j)\|$. The bound on $u(\gamma)$ follows from \cite[Lemma 2.2]{cck2013} and the choice $\gamma = n^{-K_1}$.
\end{proof}

\begin{lem}[A moment bound for observations with uniform polynomial moments]
\label{lem:moment-bounds-unifpoly-obs}
Let $q\ge8$. Suppose that $\vX_i$ are iid mean zero random vectors such that $\|\max_{1 \le k \le p} |X_{1k}| \|_q \le \nu$. If $h$ is the covariance matrix kernel in (\ref{eqn:sample-covmat-kernel}), then we have
\begin{eqnarray*}
\E [\|h\| / (2 \nu^2)]^4 \le 1.
\end{eqnarray*}
\end{lem}

\begin{proof}[Proof of Lemma \ref{lem:moment-bounds-unifpoly-obs}]
The lemma follows from
\begin{eqnarray*}
\E [\|h\| / \nu^2]^4 &\le& \left[ \E(\|h\| / \nu^2)^{q/2} \right]^{8/q} \\
&=& 2^{-4} \left[ \E \left(\max_m |X_{1m}-X_{2m}|^q / \nu^q \right) \right]^{8/q} \\
&\le& 2^{-4} \left[ 2^q \E (\max_m |X_{1m}|^q / \nu^q) \right]^{8/q} \le 2^4.
\end{eqnarray*}
\end{proof}

\section{Proof details in Section 5}

\begin{proof}[Proof of Theorem \ref{thm:expectation-bound}]
In the proof, we shall use $K_1,K_2,\cdots,$ to denote absolute constants whose values may differ from place to place, and the indices $i<j$ and $(m,k)$ implicitly run over $1 \le i < j \le n$ and $1 \le m,k \le p$. Let $\varepsilon_i, i = 1,\cdots,n$, be a sequence of iid Rademacher random variables such that $\Prob(\varepsilon_i = \pm1) = 1/2$ and $\varepsilon_i, i =1,\cdots,n,$ are also independent of $\vX_1^n$ and ${\vX'}_1^n$. By the randomization inequality \cite[Theorem 3.5.3]{delaPenaGine1999}, we have
$$
\E\|\sum_{i<j} f(\vX_i,\vX_j)\| \le K_1 \E\|\sum_{i<j} \varepsilon_i \varepsilon_j f(\vX_i,\vX_j)\|.
$$
Fix an $m,k=1,\cdots,p$ and let $\Lambda^{m,k}$ be the $n \times n$ upper triangular matrix with diagonal of zeros and $\Lambda^{m,k}_{ij} = f_{mk}(\vX_i,\vX_j)$ for $i<j$. Since $\tr(\Lambda^{m,k})=0$ and $\varepsilon_i$'s are iid sub-Gaussian, by the Hanson-Wright inequality \cite[Theorem 1]{rudelsonvershynin2013a}, conditional on $\vX_1^n$, we have for all $t>0$
\begin{equation*}
\Prob(|\mbf\varepsilon^\top \Lambda^{m,k} \mbf\varepsilon| \ge t \mid \vX_1^n) \le 2 \exp\left\{ -K_2 \min\left[ {t^2 \over |\Lambda^{m,k}|_F^2}, {t \over |\Lambda^{m,k}|_2} \right] \right\},
\end{equation*}
where $\mbf\varepsilon = (\varepsilon_1,\cdots,\varepsilon_n)^\top$. Denote $V_1 = \max_{m,k} |\Lambda^{m,k}|_F$ and $V_2 = \max_{m,k} |\Lambda^{m,k}|_2$. Let 
$$t^* = \max \left\{ V_1 \sqrt{\log(p^2) / K_2}, \quad V_2 {\log(p^2) / K_2} \right\}.$$
By the union bound, we have
\begin{eqnarray*}
&& \E[\max_{m,k} |\mbf\varepsilon^\top \Lambda^{m,k} \mbf\varepsilon|  \mid \vX_1^n] = \int_0^\infty \Prob(\max_{m,k} |\mbf\varepsilon^\top \Lambda^{m,k} \mbf\varepsilon| \ge t \mid \vX_1^n) \; dt \\
&& \qquad \le t^* + 2 p^2 \int_{t^*}^\infty \max\left\{ \exp\left( -{K_2 t^2 \over V_1^2} \right), \; \exp\left( -{K_2 t \over V_2} \right) \right\} \; dt.
\end{eqnarray*}
Changing variables, we see that
$$
\int_{t^*}^\infty \exp\left( -{K_2 t^2 \over V_1^2} \right) \; dt \le {V_1 \over \sqrt{2 K_2} } \int_{\sqrt{4\log p}}^\infty \exp\left(-{s^2 \over2} \right) \;ds.
$$
By the tail bound $1-\Phi(x) \le \phi(x)/x$ for all $x>0$, where $\Phi(\cdot)$ and $\phi(\cdot)$ are the cdf and pdf of the standard Gaussian random variable, respectively, it follows that
\begin{equation}
\label{eqn:V_1}
2 p^2 \int_{t^*}^\infty \exp\left( -{K_2 t^2 \over V_1^2} \right) \; dt  \le {V_1 \over \sqrt{2 K_2 \log{p}}} \le K_2 V_1.
\end{equation}
Here, we used $p \ge 2$. Similarly, we have
$$
2 p^2 \int_{t^*}^\infty \exp\left( -{K_2 t / V_2} \right) \; dt \le 2 V_2 / K_2.
$$
Note that $V_2 \le V_1$. Therefore, we have
\begin{eqnarray}
\nonumber
\E\|\sum_{i<j} \varepsilon_i \varepsilon_j f(\vX_i,\vX_j)\| \le K_3 t^* &\le&  K_3 (\log{p}) \E\max_{m,k} |\Lambda^{m,k}|_F \\
\label{eqn:Lambda2_term}
&\le& K_3 (\log{p}) (\E\max_{m,k} |\Lambda^{m,k}|_F^2)^{1/2},
\end{eqnarray}
where the last step follows from Jensen's inequality. 

Next, we bound the term $\E[\max_{m,k} |\Lambda^{m,k}|_F^2]$. Consider the Hoeffding decomposition of $f_{mk}^2$. Let
$$
\tilde{f}^{m,k}_1(\vx_1) = \E f^2_{mk}(\vx_1,\vX') - \E f^2_{mk}
$$
and
$$
\tilde{f}^{m,k}(\vx_1,\vx_2) = f^2_{mk}(\vx_1,\vx_2) - \E f^2_{mk}(\vx_1,\vX') - \E f^2_{mk}(\vX,\vx_2) + \E f^2_{mk}.
$$
Clearly, $\E [\tilde{f}^{m,k}_1(\vX)] = 0$ and $\E [\tilde{f}^{m,k}(\vX,\vX')] = \E [\tilde{f}^{m,k}(\vx_1,\vX')] = \E [\tilde{f}^{m,k}(\vX,\vx_2)] = 0$ for all $\vx_1,\vx_2 \in \mathbb{R}^p$; i.e. $\tilde{f}^{m,k}_1$ is centered and $\tilde{f}^{m,k}$ is a canonical kernel of U-statistic of order two w.r.t. $F$. Since
$$
f^2_{mk}(\vx_1,\vx_2) - \E f^2_{mk} = \tilde{f}^{m,k}(\vx_1,\vx_2) + \tilde{f}^{m,k}_1(\vx_1) + \tilde{f}^{m,k}_1(\vx_2),
$$
we have by the triangle inequality
\begin{eqnarray}
\nonumber
\E \max_{m,k} \sum_{i<j} f_{mk}^2(\vX_i,\vX_j) &\le& \max_{m,k} \sum_{i<j} \E f_{mk}^2 + \E \max_{m,k} \Big| \sum_{i<j} ( f_{mk}^2(\vX_i,\vX_j) - \E f_{mk}^2 ) \Big| \\
\label{eqn:f^2_quad}
&\le& n^2 \max_{m,k} \E f_{mk}^2 + \E \|\sum_{i<j} \tilde{f}(\vX_i,\vX_j)\| + (n-1) \E\|\sum_{i=1}^n \tilde{f}_1(\vX_i)\|,
\end{eqnarray}
where $\tilde{f}=\{\tilde{f}^{m,k}\}_{m,k=1}^p$ and $\tilde{f}_1=\{\tilde{f}^{m,k}_1\}_{m,k=1}^p$ are $p \times p$ random matrices. By the Hoeffding inequality, conditional on $\vX_1^n$, we have for all $t>0$
\begin{equation*}
\Prob(|\sum_{i=1}^n \varepsilon_i \tilde{f}_1^{m,k}(\vX_i)| \ge t \mid \vX_1^n) \le 2 \exp\left( - {K_4 t^2 \over \sum_{i=1}^n (\tilde{f}_1^{m,k}(\vX_i))^2} \right).
\end{equation*}
By the symmetrization inequality \cite[Lemma 2.3.1]{vandervaartwellner1996} and the argument for bounding (\ref{eqn:V_1}), we get
\begin{equation}
\label{eqn:f_tilde_1}
\E \|\sum_{i=1}^n \tilde{f}_1(\vX_i)\| \le 2 \E \|\sum_{i=1}^n \varepsilon_i \tilde{f}_1(\vX_i)\| \le K_5 (\log{p})^{1/2} \E \Big[ \max_{m,k} \sum_{i=1}^n \tilde{f}^{m,k}_1(\vX_i)^2 \Big]^{1/2}.
\end{equation}
By the randomization inequality as in the previous argument before Jensen's inequality (\ref{eqn:Lambda2_term}), we get
\begin{equation}
\label{eqn:f_tilde_2}
\E \|\sum_{i<j}\tilde{f}(\vX_i,\vX_j)\| \le K_1 \E \|\sum_{i<j} \varepsilon_i \varepsilon_j \tilde{f}(\vX_i,\vX_j)\| \le K_3 (\log{p}) \E \Big[ \max_{m,k} \sum_{i<j} \tilde{f}^{m,k}(\vX_i,\vX_j)^2 \Big]^{1/2}.
\end{equation}
By the triangle and Jensen's inequalities, we have
\begin{eqnarray*}
&& \E \Big[ \max_{m,k} \sum_{i<j} \tilde{f}^{m,k}(\vX_i,\vX_j)^2 \Big]^{1/2}  \le 2 \E \Big[\max_{m,k} \sum_{i<j} f_{mk}^4(\vX_i,\vX_j) \Big]^{1/2} \\
&& \qquad \qquad  + 8^{1/2} \E\Big\{ \max_{m,k} \sum_{i<j} [\E( f_{mk}^2(\vX_i, \vX') | \vX_1^n)]^2 \Big\}^{1/2} + 2 \Big[ \max_{m,k} \sum_{i<j} (\E f_{mk}^2)^2 \Big]^{1/2}.
\end{eqnarray*}
Let $I = \E[\max_{m,k} \sum_{i<j} f_{mk}^2(\vX_i,\vX_j)]$. Then by the Cauchy-Schwarz inequality, we have
$$
\E \Big[\max_{m,k} \sum_{i<j} f_{mk}^4(\vX_i,\vX_j) \Big]^{1/2} \le \sqrt{I} \sqrt{\E M^2},
$$
and
\begin{eqnarray*}
&& \E\Big\{ \max_{m,k} \sum_{i<j} [\E( f_{mk}^2(\vX_i, \vX') | \vX_1^n)]^2 \Big\}^{1/2}\\
&\le& \E \Big\{ \Big[\max_{m,k} \max_{i<j} \E( f_{mk}^2(\vX_i, \vX'_j) | \vX_1^n) \Big]^{1/2} \Big[\max_{m,k} \sum_{i<j} \E( f_{mk}^2(\vX_i, \vX'_j) | \vX_1^n \Big]^{1/2} \Big\} \\
&\le& \Big[\E \max_{m,k} \max_{i<j} \E( f_{mk}^2(\vX_i, \vX'_j) | \vX_1^n) \Big]^{1/2} \Big[ \E \max_{m,k} \sum_{i<j} \E( f_{mk}^2(\vX_i, \vX'_j) | \vX_1^n) \Big]^{1/2} \\
&\le& \sqrt{\E M^2} \sqrt{4 I},
\end{eqnarray*}
where in the last step we used Jensen's inequality and the decoupling inequality \cite[Theorem 3.1.1]{delaPenaGine1999}. In addition, $\max_{m,k} \sum_{i<j} (\E f_{mk}^2)^2 \le I \E M^2$. Therefore, we obtain from (\ref{eqn:f_tilde_2}) that
$$
\E \|\sum_{i<j}\tilde{f}(\vX_i,\vX_j)\| \le K_6 (\log{p}) \sqrt{I} \sqrt{\E M^2}.
$$
By (\ref{eqn:f^2_quad}), (\ref{eqn:f_tilde_1}), and (\ref{eqn:f_tilde_2}), we obtain that
$$
I \le K_7 \Big\{ n^2 \max_{m,k} \E f_{mk}^2 + (\log{p}) \sqrt{I} \sqrt{\E M^2} + n (\log{p})^{1/2} \E \Big[ \max_{m,k} \sum_i \tilde{f}^{m,k}_1(\vX_i)^2 \Big]^{1/2} \Big\}.
$$
The solution of this quadratic inequality for $I$ is given by
\begin{equation}
\label{eqn:bound_I}
I \le K_8 \Big\{ (\log{p})^2 (\E M^2) + n^2 \max_{m,k} \E f_{mk}^2 + n (\log{p})^{1/2} \E \Big[ \max_{m,k} \sum_i \tilde{f}^{m,k}_1(\vX_i)^2 \Big]^{1/2} \Big\}.
\end{equation}
By \cite[Lemma 9]{cck2014b} and Jensen's inequality,
\begin{eqnarray}
\nonumber
\E \Big[ \max_{m,k} \sum_i \tilde{f}^{m,k}_1(\vX_i)^2 \Big] &\le& K_9 \Big\{ \max_{m,k} \E \sum_i \tilde{f}^{m,k}_1(\vX_i)^2 + (\log{p}) \E  \Big[ \max_{m,k} \max_i \tilde{f}^{m,k}_1(\vX_i)^2 \Big] \Big\} \\
\nonumber
&\le& K_9 \Big\{ n \max_{m,k} \E f_{mk}^4(\vX_1,\vX_2) + (\log{p}) \E  \Big[ \max_{i<j} \max_{m,k} f_{mk}^4(\vX_i,\vX'_j) \Big] \Big\}  \\
\label{eqn:bound_I_simplify}
&\le& K_9 \Big\{ n \max_{m,k} \E f_{mk}^4(\vX_1,\vX_2) + (\log{p}) \E M^4 \Big\}.
\end{eqnarray}

Now, combining (\ref{eqn:Lambda2_term}), (\ref{eqn:bound_I}), and (\ref{eqn:bound_I_simplify}), we conclude that
\begin{eqnarray*}
&&\qquad  \E \| \sum_{i<j} f(\vX_i,\vX_j) \| \\
&\le& K_{10} (\log{p}) \Big\{ n^2 \max_{m,k} \E f_{mk}^2 + (\log{p})^2 \E M^2 + n (\log{p})^{1/2} [n \max_{m,k} \E f_{mk}^4 + (\log{p}) \E M^4]^{1/2}  \Big\}^{1/2} \\
&\le& K_{10} (\log{p}) \Big\{ n D_2 + (\log{p}) \|M\|_2 + n^{3/4} (\log{p})^{1/4} D_4 + (n \log{p})^{1/2} \|M\|_4 \Big\}.
\end{eqnarray*}
Since $\|M\|_2 \le \|M\|_4$ and $p \le \exp(b n)$, we have $(\log{p}) \|M\|_2 \le (b n \log{p})^{1/2} \|M\|_4$, from which (\ref{eqn:expectation-bound-canonical}) follows.
\end{proof}

\begin{proof}[Proof of Corollary \ref{cor:expectation-subexponential-kernel}]
The Corollary follows from the bounds $D_2 \le D_4 \le B_n$ and $\|M\|_4 \le K B_n \log(np)$, where the last inequality comes from Lemma \ref{lem:moment-bounds-subexponential-kernel}.
\end{proof}

\begin{proof}[Proof of Corollary \ref{cor:expectation-polynomial-kernel}]
The Corollary follows from the bounds $D_2 \le D_4 \le B'_n$ and $\|M\|_4 \le B_n n^{1/2}$, where the last inequality comes from Lemma \ref{lem:moment-bounds-unifpoly-kernel}.
\end{proof}

\begin{proof}[Proof of Lemma \ref{lem:gaussian-approximation-nondegenarate-U-smooth}]
Let $v = \lambda \circ F_\beta$. By the mean value theorem and noticing that $\pi_{mk}(Z) \ge 0$ and $\sum_{m,k=1}^p \pi_{mk}(Z) = 1$ for all $Z \in \mathbb{R}^{p \times p}$, we have
\begin{equation*}
\left| \E[v(T) - v(L)] \right| = | \E [ \sum_{m,k=1}^p W_{mk} \partial \lambda(F_\beta(\xi))\pi_{mk}(\xi) ] | \le  K_0 \psi \E \|W\|,
\end{equation*}
where $\xi$ is a random matrix on the line segment between $T$ and $L$. Let $\Delta_v = \left| \E[v(L) - v(Z)] \right|$. Then,
\begin{equation*}
\left| \E[v(T) - v(Z)] \right| \le \Delta_v + K_0 \psi \E \|W\|.
\end{equation*}
By the smoothing properties of $F_\beta(\cdot)$ in (\ref{eqn:smoothing_max}) and $\lambda(\cdot)$ in (\ref{eqn:smoothing_indicator}), we have
\begin{eqnarray*}
\Prob(\bar{T}_0 \le t) &\le& \Prob(F_\beta(T) \le t + e_\beta) \le \E[\lambda(F_\beta(T))]  = \E[v(T)] \\
&\le& \E[v(Z)] + (K_0 \psi \E \|W\| + \Delta_v) \\
&=& \E[\lambda(F_\beta(Z))] + (K_0 \psi \E \|W\| + \Delta_v) \\
&\le& \Prob(F_\beta(Z) \le t + e_\beta + \psi^{-1}) +  (K_0 \psi \E \|W\| + \Delta_v) \\
&\le& \Prob(\bar{Z}_0 \le t + e_\beta + \psi^{-1}) +  (K_0 \psi \E \|W\| + \Delta_v).
\end{eqnarray*}
By the anti-concentration inequality \cite[Lemma 2.1]{cck2013},
\begin{equation*}
\Prob(\bar{Z}_0 \le t + e_\beta + \psi^{-1}) - \Prob(\bar{Z}_0 \le t) \le C (e_\beta + \psi^{-1}) \sqrt{1 \vee \log(p \psi)},
\end{equation*}
where $C>0$ is a constant depending only on $c_0$ and $C_0$. Therefore, we get
\begin{equation*}
\Prob(\bar{T}_0 \le t) - \Prob(\bar{Z}_0 \le t) \le C[ (\psi \E \|W\| + \Delta_v) +  (e_\beta + \psi^{-1}) \sqrt{1 \vee \log(p \psi)}].
\end{equation*}
Since $L$ involves sums of iid random matrices and $Z$ is Gaussian of the matching first and second moments to $L$, the bound for $\Delta_v$ in (\ref{eqn:bound_Delta_v}) follows from Theorem 2.1 and the Step 1 in the proof of Theorem 2.2 in \cite{cck2013}. Similarly, we can prove the other half inequality.
\end{proof}

\begin{proof}[Proof of Theorem \ref{thm:GA_master_thm}]
In this proof, we shall use $a \lesssim b$ to denote $a \le C b$ for some constant $C > 0$ depending only on $c_0,C_0,$ and $b$. Since $\tilde{D}_2$ is lower and upper bounded, we may assume that $\tilde{D}_2=1$. For $u>0$, let $\bar\varphi(u) = C_1 \tilde{D}_4^2 / u$ and $\beta(u) = n^{1/2} / (2\sqrt{2} u)$. Define
\begin{eqnarray*}
\psi_1(u) &=& \min\{ n^{1/8} \ell_n^{-3/8} \tilde{D}_3^{-3/4}, \; \ell_n^{-1/6} (\bar\varphi(u))^{-1/3} \}, \\
\psi_2(u) &=& \min\{ n^{1/2} \ell_n^{-1/2} \|M\|_4^{-1/2}, \; \ell_n^{-1/6} (\bar\varphi(u))^{-1/3} \}, \\
\psi_3(u) &=& \min\{ n^{1/4} \ell_n^{-1/4} D_2^{-1/2}, \; \ell_n^{-1/6} (\bar\varphi(u))^{-1/3} \}, \\
\psi_4(u) &=& \min\{ n^{3/8} \ell_n^{-3/8} D_4^{-1/2}, \; \ell_n^{-1/6} (\bar\varphi(u))^{-1/3} \},
\end{eqnarray*}
and let $\psi(u) = \min_{1 \le i \le 4} \psi_i(u)$. Clearly, $\psi(u) \le n^{1/8}$ for $u \ge 0$ and $\psi_i(u)$ are strictly increasing functions for $u \in (0, u_i^*)$, where $u_i^*$ balances the corresponding two components in $\psi_i(u)$ for $i=1,2,3,4$. In addition, for $u \ge u_i^*$ we have $\psi_i(u) = \psi_i(u_i^*) = \max_{v \ge 0} \psi_i(v)$, i.e. all the $u_i^*$'s attain the corresponding maximum values of $\psi_i(\cdot)$ at the truncation levels. Let $u^* = \max_{1 \le i \le 4} u_i^*$.

Our goal is to show that: for any $\gamma \in (0,1)$ and for some properly chosen $u$ in {\it Step 1} below, we have
\begin{equation}
\label{eqn:GA_master_thm_revised_goal}
\boxed{\rho(\bar{T}_0, \bar{Z}_0) \lesssim n^{-1/2} \ell_n^{3/2} u + \gamma.}
\end{equation}

\vspace{0.1in}

\uline{{\it Step 1.} Choose a proper value of $u$.} 

To prove (\ref{eqn:GA_master_thm_revised_goal}), we may further assume that $n^{-1/2} \ell_n^{3/2} u \le 1$ because otherwise it trivially holds. Choose $u = u_0 \vee u_1 \vee u_2 \vee u_3 \vee u_4 \vee u(\gamma)$, where
\begin{eqnarray*}
u_0 &=& n^{3/8} \ell_n^{-5/8} \tilde{D}_4^{1/2}, \\
u_1 &=& n^{3/8} \ell_n^{-5/8} \tilde{D}_3^{3/4}, \\
u_2 &=& \ell_n^{-1/2} \|M\|_4^{1/2}, \\
u_3 &=& n^{1/4} \ell_n^{-3/4} D_2^{1/2}, \\
u_4 &=& n^{1/8} \ell_n^{-5/8} D_4^{1/2}.
\end{eqnarray*}
For this chosen $u$, we then determine the smoothing parameters $\beta := \beta(u)$ and $\psi = \min_{1 \le i \le 4} \psi_i$ where $\psi_i := \psi_i(u)$. Let $\bar\beta_i = n^{1/2} u_i^{-1}$ for $i=1,2,3,4,$ and $\bar{\beta} = \min_{1 \le i \le 4} \bar{\beta}_i$. Then, 
$$(\beta \vee \psi) \le \bar\beta$$
because $\beta = 8^{-1/2} n^{1/2} u^{-1} \le n^{1/2} \min_{1 \le i \le 4} u_i^{-1} = \bar\beta$ and $\psi_i \le \bar\beta_i$ for $i=1,2,3,4$. \\

\uline{{\it Step 2.} Show that (\ref{eqn:GA_master_thm_revised_goal}) holds for our choice of $u$ (and therefore $\beta$ and $\psi$).} Then, (\ref{eqn:GA_master_thm}) follows immediately from the substitution of $u$ into (\ref{eqn:GA_master_thm_revised_goal}) 
$$
\rho(\bar{T}_0, \bar{Z}_0) \lesssim n^{-1/2} \ell_n^{3/2} [u_0 \vee u_1 \vee u_2 \vee u_3 \vee u_4 \vee u(\gamma)] + \gamma.
$$
The rest of the proof is to show (\ref{eqn:GA_master_thm_revised_goal}). First, note that since $\psi \le n^{1/8}$, 
$$
e_\beta \sqrt{1 \vee \log(p\psi)} \le K_1 \beta^{-1} (\log{p}) \sqrt{\log(pn)} \le K_1 \beta^{-1} \ell_n^{3/2} \le K_1 n^{-1/2} \ell_n^{3/2} u.
$$
For the rest of the terms, we divide into two cases. \\

\underline{\bf Case I: $u \ge u^*$.} By Lemma \ref{lem:gaussian-approximation-nondegenarate-U-smooth} and Theorem \ref{thm:expectation-bound}, we have
\begin{eqnarray*}
\rho(\bar{T}_0, \bar{Z}_0) &\lesssim& \underbrace{n^{-1/2} \bar\beta^2 \psi \tilde{D}_3^3}_{\text{(i)}} + \underbrace{\bar\beta \psi \bar\varphi(u)}_{\text{(ii)}} + \underbrace{\psi \bar\varphi(u) \ell_n^{1/2}}_{\text{(iii)}} + \gamma + n^{-1/2} \ell_n^{3/2} u \\
&& + \underbrace{\psi^{-1} \ell_n^{1/2}}_{\text{(iv)}} + \underbrace{\psi \ell_n^{3/2} n^{-1} \|M\|_4}_{\text{(v)}} + \underbrace{\psi \ell_n n^{-1/2} D_2}_{\text{(vi)}} + \underbrace{\psi \ell_n^{5/4} n^{-3/4} D_4}_{\text{(vii)}},
\end{eqnarray*}
where we used $(\beta \vee \psi) \le \bar\beta$ in terms (i)--(ii) and $\psi \le n^{1/8}$ in term (iv). Next, we bound the terms (i)--(vii).

\underline{Term (i).} We have
\begin{eqnarray*}
n^{-1/2} \bar\beta^2 \psi \tilde{D}_3^3 &\le_{(1)}& n^{-1/2} \bar\beta_1^2 \psi_1 \tilde{D}_3^3  =_{(2)} n^{-1/2} \bar\beta_1^2 \psi_1(u_1^*) \tilde{D}_3^3 \\
&=_{(3)}& n^{-1/2} (n^{1/8} \ell_n^{5/8} \tilde{D}_3^{-3/4})^2 (n^{1/8} \ell_n^{-3/8} \tilde{D}_3^{-3/4}) \tilde{D}_3^3 \\
&=_{(4)}& n^{-1/8} \ell_n^{7/8} \tilde{D}_3^{3/4} =_{(5)} n^{-1/2} \ell_n^{3/2} u_1 \le_{(6)} n^{-1/2} \ell_n^{3/2} u,
\end{eqnarray*}
where $(1)$ follows from the definitions of $\psi$ and $\bar\beta$, $(2)$ from $u \ge u^* \ge u_1^*$, $(3)$ from the definition of $\bar\beta_1$ and the truncation property of $\psi_1(u_1^*) = \max_{v \ge 0} \psi_1(v)$, $(4)$ from the direct calculations, $(5)$ from the definition of $u_1$, and $(6)$ from $u \ge u_1$.

\underline{Term (ii).} We have
\begin{eqnarray*}
\bar\beta \psi \bar\varphi(u) &\le_{(1)}& \bar\beta_1 \psi_1(u_1^*) \bar\varphi(u_1^*)  \\
&=_{(2)}& (n^{1/8} \ell_n^{5/8} \tilde{D}_3^{-3/4}) (n^{1/8} \ell_n^{-3/8} \tilde{D}_3^{-3/4}) ( n^{-3/8} \ell_n^{5/8} \tilde{D}_3^{9/4}) \\
&=_{(3)}& n^{-1/8} \ell_n^{7/8} \tilde{D}_3^{3/4} \le_{(4)} n^{-1/2} \ell_n^{3/2} u,
\end{eqnarray*}
where $(1)$ follows from the definitions of $\psi,\bar\beta$, $\bar\varphi(u)$ is a decreasing function in $u$ and $u \ge u_1^*$, $(2)$ from the definition of $\bar\beta_1$ and the truncation property of $\psi_1(\cdot)$, $(3)$ from the direct calculations, and $(4)$ from the steps $(5)$ and $(6)$ in bounding term (i).

\underline{Term (iii).} Note that $\psi \bar\varphi(u) \ell_n^{1/2} = \bar\beta \psi \bar\varphi(u) \ell_n^{1/2} / \bar\beta$. We claim that $\ell_n^{1/2} / \bar\beta \le 1$. Then, by the calculations in term (ii), we have $\psi \bar\varphi(u) \ell_n^{1/2} \le \bar\beta \psi \bar\varphi(u) \le n^{-1/2} \ell_n^{3/2} u$. Now, we verify the claim. Since $\psi_i = \psi_i(u_i^*)$ for $u \ge u^*$, it is easy to check that $\psi_i = \ell_n^{-1} n^{1/2} u_i^{-1} = \ell_n^{-1} \bar\beta_i$ and therefore $\ell_n^{1/2} \bar\beta^{-1} = \ell_n^{-1/2} \psi^{-1}$. By Lemma \ref{lem:key_bounds_GAR}, we have $\ell_n^{1/2} \bar\beta^{-1} = \ell_n^{-1/2} \psi^{-1} \le \ell_n^{1/2} \psi^{-1} \le n^{-1/2} \ell_n^{3/2} u \le 1$.

\underline{Term (iv).} The bound $\ell_n^{1/2} \psi^{-1} \le n^{-1/2} \ell_n^{3/2} u$ follows from Lemma \ref{lem:key_bounds_GAR}.

\underline{Term (v), (vi), and (vii).} Since $\psi \le \psi_i$ for $i=2,3,4$, we can show by direct calculations that
\begin{eqnarray*}
\psi \ell_n^{3/2} n^{-1} \|M\|_4 &\le& \psi_2 \ell_n^{3/2} n^{-1} \|M\|_4  = \ell_n^{1/2} \psi_2^{-1} \le \ell_n^{1/2} \psi^{-1} \le n^{-1/2} \ell_n^{3/2} u, \\
\psi \ell_n n^{-1/2} D_2 &\le& \psi_3 \ell_n n^{-1/2} D_2 = \ell_n^{1/2} \psi_3^{-1} \le \ell_n^{1/2} \psi^{-1} \le n^{-1/2} \ell_n^{3/2} u, \\
\psi \ell_n^{5/4} n^{-3/4} D_4 &\le& \psi_4 \ell_n^{5/4} n^{-3/4} D_4 = \ell_n^{1/2} \psi_4^{-1} \le \ell_n^{1/2} \psi^{-1} \le n^{-1/2} \ell_n^{3/2} u
\end{eqnarray*}
where the last steps of each line follow from Lemma \ref{lem:key_bounds_GAR}.

Therefore, we conclude that  (\ref{eqn:GA_master_thm_revised_goal}) holds when $u \ge u^*$. \\

\underline{\bf Case II: $0< u < u^*$.} By Lemma \ref{lem:gaussian-approximation-nondegenarate-U-smooth} and Theorem \ref{thm:expectation-bound}, we have
\begin{eqnarray*}
\rho(\bar{T}_0, \bar{Z}_0) &\lesssim& \underbrace{n^{-1/2} \bar\beta^2 \psi \tilde{D}_3^3}_{\text{(i)}} + \underbrace{(\psi^2 + \psi \beta) \bar\varphi(u)}_{\text{(ii)}} + \underbrace{\psi \bar\varphi(u) \ell_n^{1/2}}_{\text{(iii)}} + \gamma + n^{-1/2} \ell_n^{3/2} u \\
&& + \underbrace{\psi^{-1} \ell_n^{1/2}}_{\text{(iv)}} + \underbrace{\psi \ell_n^{3/2} n^{-1} \|M\|_4}_{\text{(v)}} + \underbrace{\psi \ell_n n^{-1/2} D_2}_{\text{(vi)}} + \underbrace{\psi \ell_n^{5/4} n^{-3/4} D_4}_{\text{(vii)}},
\end{eqnarray*}
where we used $(\beta \vee \psi) \le \bar\beta$ in terms (i) and $\psi \le n^{1/8}$ in term (iv). Again, our task is to bound the terms (i)--(vii) in order to achieve (\ref{eqn:GA_master_thm_revised_goal}).

\underline{Term (i).} We have
\begin{equation*}
n^{-1/2} \bar\beta^2 \psi \tilde{D}_3^3 \le_{(1)} n^{-1/2} \bar\beta_1^2  \psi_1  \tilde{D}_3^3 \le_{(2)} n^{-1/8} \ell_n^{7/8} \tilde{D}_3^{3/4} \le_{(3)} n^{-1/2} \ell_n^{3/2} u,
\end{equation*}
where $(1)$ follows from the definitions of $\psi$ and $\bar\beta$, $(2)$ from the definition of $\bar\beta_1$, the truncation property of $\psi_1 \le n^{1/8} \ell_n^{-3/8} \tilde{D}_3^{-3/4}$ and steps $(3)$--$(4)$ of term (i) in {\bf Case I}, and $(3)$ from steps $(5)$--$(6)$ of term (i) in {\bf Case I}.

\underline{Term (ii).} We write $\text{(ii)} = \text{(ii.1)} + \text{(ii.2)}$, where $\text{(ii.1)} = \psi^2 \bar\varphi(u)$ and $\text{(ii.2)} = \psi\beta\bar\varphi(u)$. By the truncation property of $\psi(\cdot)$, we have
\begin{eqnarray*}
\text{(ii.1)} &\le& [\ell_n^{-1/3} (\bar\varphi(u))^{-2/3}] \bar\varphi(u) = \ell_n^{-1/3} (\bar\varphi(u))^{1/3}, \\
\text{(ii.2)} &\le& [\ell_n^{-1/6} (\bar\varphi(u))^{-1/3}] \beta\bar\varphi(u) = \beta \ell_n^{-1/6} (\bar\varphi(u))^{2/3}.
\end{eqnarray*}
Recall that $\bar\varphi(u) = C \tilde{D}_4^2 u^{-1}$, $\beta(u) = (2\sqrt{2})^{-1} n^{1/2} u^{-1}$, and $ u_0 = n^{3/8} \ell_n^{-5/8} \tilde{D}_4^{1/2}$ so that
\begin{eqnarray}
\label{eqn:bar_varphi_u0_value}
\bar\varphi(u_0) &=& C n^{-3/2} \ell_n^{5/2} u_0^3, \\ \nonumber
\beta(u_0) &=& C n^{1/2} u_0^{-1}.
\end{eqnarray}
Since $\bar\varphi(\cdot)$ and $\beta(\cdot)$ are decreasing in $u$ and $u \ge u_0$, we get
\begin{eqnarray*}
\text{(ii.1)} &\le& \ell_n^{-1/3} (\bar\varphi(u_0))^{1/3} \le C n^{-1/2} \ell_n^{1/2} u_0 \le C n^{-1/2} \ell_n^{1/2} u, \\
\text{(ii.2)} &\le& \beta(u_0) \ell_n^{-1/6} (\bar\varphi(u_0))^{2/3} \le C n^{-1/2} \ell_n^{3/2} u_0 \le C n^{-1/2} \ell_n^{3/2} u.
\end{eqnarray*}
Therefore, we obtain that $\text{(ii)} \lesssim n^{-1/2} \ell_n^{3/2} u$.

\underline{Term (iii).} By the same argument as in bounding the term (ii) above, we have
\begin{eqnarray*}
\psi \bar\varphi(u) \ell_n^{1/2} &\le_{(1)}& [\ell_n^{-1/6} (\bar\varphi(u))^{-1/3}] \bar\varphi(u) \ell_n^{1/2} = \ell_n^{1/3} (\bar\varphi(u))^{2/3} \le_{(2)} \ell_n^{1/3} (\bar\varphi(u_0))^{2/3} \\
&\le_{(3)}& C n^{-1} \ell_n^2 u_0^2 \le C n^{-1} \ell_n^2 u^2 = C (\underbrace{n^{-1/2} \ell_n^{3/2} u}_{\le 1}) (\underbrace{n^{-1/2} \ell_n^{1/2} u}_{\le n^{-1/2} \ell_n^{3/2} u}) \le C n^{-1/2} \ell_n^{3/2} u.
\end{eqnarray*}
where $(1)$ follows from the truncation property of $\psi(\cdot)$, $(2)$ from the fact that $\bar\varphi(\cdot)$ is decreasing and $u \ge u_0$, $(3)$ from (\ref{eqn:bar_varphi_u0_value}), and the rest equalities and inequalities are obvious.

\underline{Term (iv).} By Lemma \ref{lem:key_bounds_GAR}, we have $\ell_n^{1/2} \psi^{-1} \le C n^{1/2} \ell_n^{3/2} u$ for some constant $C > 0$ depending only on $c_0$ and $C_0$.

\underline{Term (v), (vi), and (vii).} Since $\psi_i(\cdot)$ is non-decreasing, $\psi$ is no greater than its value in {\bf Case I} and these three terms can be handled in the same way as in {\bf Case I}. \\

Therefore, regardless of $u \ge u^*$ or $0 < u < u^*$, (\ref{eqn:GA_master_thm_revised_goal}) always holds for any $\gamma \in (0,1)$ and the proof is complete.
\end{proof}

\section{Auxiliary lemmas in the proof of Section 5}

\begin{lem}[A key bound on $\psi^{-1}$]
\label{lem:key_bounds_GAR}
Let $u^*, u,$ and $\psi_i(\cdot), i =1,2,3,4$ be defined in the proof of Theorem \ref{thm:GA_master_thm}. Let $\psi_i = \psi_i(u)$ and $\psi= \min_{1 \le i \le 4} \psi_i$. Assume that $c_0 \le \tilde{D}_2 \le C_0$. Then we have for all $u > 0$
\begin{eqnarray}
\label{eqn:key_bounds_GAR}
\ell_n^{1/2} \psi^{-1} &\le& C n^{-1/2} \ell_n^{3/2} u,
\end{eqnarray}
where $C > 0$ is a constant depending only on $c_0$ and $C_0$. In particular, if $u \ge u^*$, then we can take $C = 1$.
\end{lem}

\begin{proof}[Proof of Lemma \ref{lem:key_bounds_GAR}]
We shall use the same notations as in the proof of Theorem \ref{thm:GA_master_thm}. We divide the proof of (\ref{eqn:key_bounds_GAR}) into two cases.

\underline{\bf Case I: $u \ge u^*$.} In this case, all the $\psi_i(u)$'s for $i=1,2,3,4$ attain the corresponding maximum of $\psi_i$ at the truncation levels; i.e.
\begin{eqnarray*}
\psi_1 &=& n^{1/8} \ell_n^{-3/8} \tilde{D}_3^{-3/4}, \\
\psi_2 &=& n^{1/2} \ell_n^{-1/2} \|M\|_4^{-1/2}, \\
\psi_3 &=& n^{1/4} \ell_n^{-1/4} D_2^{-1/2}, \\
\psi_4 &=& n^{3/8} \ell_n^{-3/8} D_4^{-1/2}.
\end{eqnarray*}
Then (\ref{eqn:key_bounds_GAR}) follows from the definitions of $u_i$ and the direct calculations
\begin{eqnarray*}
\ell_n^{1/2} \psi_1^{-1} &=& n^{-1/8} \ell_n^{7/8} \tilde{D}_3^{3/4} = n^{-1/2} \ell_n^{3/2} u_1 \le n^{-1/2} \ell_n^{3/2} u, \\
\ell_n^{1/2} \psi_2^{-1} &=& n^{-1/2} \ell_n \|M\|_4^{1/2} = n^{-1/2} \ell_n^{3/2} u_2 \le n^{-1/2} \ell_n^{3/2} u, \\
\ell_n^{1/2} \psi_3^{-1} &=& n^{-1/4} \ell_n^{3/4} D_2^{1/2} = n^{-1/2} \ell_n^{3/2} u_3 \le n^{-1/2} \ell_n^{3/2} u, \\
\ell_n^{1/2} \psi_4^{-1} &=& n^{-3/8} \ell_n^{7/8} D_4^{1/2} = n^{-1/2} \ell_n^{3/2} u_4 \le n^{-1/2} \ell_n^{3/2} u.
\end{eqnarray*}

\underline{\bf Case II: $0< u < u^*$.} In this case, note that $\ell_n^{1/2} \psi^{-1}$ is equal to
\begin{eqnarray*}
\ell_n^{1/2} \max\left\{ \ell_n^{1/6} (\bar\varphi(u))^{1/3}, \; n^{-1/8} \ell_n^{3/8} \tilde{D}_3^{3/4}, \; n^{-1/2} \ell_n^{1/2} \|M\|_4^{1/2}, \; n^{-1/4} \ell_n^{1/4} D_2^{1/2}, \; n^{-3/8} \ell_n^{3/8} D_4^{1/2} \right\}.
\end{eqnarray*}
We only need to bound $\ell_n^{1/2} [\ell_n^{1/6} (\bar\varphi(u))^{1/3}]$ since the remaining terms have the same bounds as in {\bf Case I}. Since $\bar\varphi(\cdot)$ is decreasing in $u$ and $u \ge u_0$, we have
\begin{equation*}
\ell_n^{2/3} (\bar\varphi(u))^{1/3} \le \ell_n^{2/3} (\bar\varphi(u_0))^{1/3}.
\end{equation*}
Recall the definitions $\bar\varphi(u) = C \tilde{D}_4^2 u^{-1}$ and $u_0 = n^{3/8} \ell_n^{-5/8} \tilde{D}_4^{1/2}$. Then, $\tilde{D}_4 = u_0^2 n^{-3/4} \ell_n^{5/4}$ and $\bar\varphi(u_0) = C \tilde{D}_4^2 u_0^{-1} = C u_0^3 n^{-3/2} \ell_n^{5/2}$. Therefore, we have $\ell_n^{2/3} (\bar\varphi(u_0))^{1/3} = C^{1/3} n^{-1/2} \ell_n^{3/2} u_0 \le C^{1/3} n^{-1/2} \ell_n^{3/2} u$ since $u \ge u_0$.
\end{proof}

\begin{lem}
\label{lem:bound-on-rho_ominus}
Let $\rho_\ominus(\alpha) = \Prob(\{\bar{T}_0 \le a_{\hat{L}_0^*}(\alpha) \} \ominus \{ \bar{L}_0 \le a_{\bar{Z}_0}(\alpha) \} )$, where $a_{\bar{Z}_0}(\alpha) = \inf\{t\in\mathbb{R} : \Prob(\bar{Z}_0 \le t) \ge \alpha\}$ is the $\alpha$-th quantile of $\bar{Z}_0$. Assume that $\E g_{mk}^2 \ge C_1$ for all $m,k=1,\cdots,p$. Suppose that
\begin{eqnarray*}
\Prob(|\bar{T}_0 - \bar{L}_0| > \zeta_1) < \zeta_2, \\
\Prob(\Prob_e(|\hat{L}_0^* - \bar{L}_0^*| > \zeta_1) > \zeta_2) < \zeta_2,
\end{eqnarray*}
for some $\zeta_1, \zeta_2 \ge 0$. Then, for every $\alpha \in (0,1)$ and $v > 0$, we have
\begin{equation*}
\rho_\ominus(\alpha) \le 2 \left[ \rho(\bar{L}_0, \bar{Z}_0) + C v^{1/3} (\log{p})^{2/3} + \Prob(\Delta_1 > v) \right] + C' \zeta_1 (\log{p})^{1/2} + 5 \zeta_2,
\end{equation*}
where $\Delta_1$ is defined in (\ref{eqn:gaussian_wild_bootstrap_Delta1}) and $C, C' > 0$ are constants only depending on $C_1$.
\end{lem}

\begin{proof}[Proof of Lemma \ref{lem:bound-on-rho_ominus}]
The proof is a modification of \cite[Theorem 3.2]{cck2013}, verbatim replacing the anti-concentration inequality of \cite[Theorem 2 and 3]{cck2014b} by Nazarov's inequality \cite{nazarov2003}. The benefit of using Nazarov's inequality is that we can have $v^{1/3} (\log{p})^{2/3}$ and $\zeta_1 (\log{p})^{1/2}$, instead of $v^{1/3} (1 \vee \log(p/v))^{2/3}$ and $\zeta_1 (1 \vee \log(p/\zeta))^{1/2}$ in \cite[Theorem 3.2]{cck2013}, respectively, where the former can give the convergence rates in Theorem \ref{thm:gaussian_wild_bootstrap_validity} that decay to zero polynomially fast in $n$ without additional logarithm factors.
\end{proof}

In the following Lemma \ref{lem:gaussian_wild_bootstrap_Delta1_exp_moment}--\ref{lem:gaussian_wild_bootstrap_W_unifpoly_moment}, we shall assume that $p \ge 2$.

\begin{lem}[Bound on $\E(\Delta_1)$: sub-exponential moment]
\label{lem:gaussian_wild_bootstrap_Delta1_exp_moment}
Assume (\ref{eqn:subexponential-kernel-moment-condition}). Let
\begin{equation}
\label{eqn:varpi_1}
\varpi_1(n,p) = \left({\log{p}\over n}\right)^{1/2} B_n + {(\log{p}) (\log(np))^2 \over n} B_n^2.
\end{equation}
and
\begin{equation}
\label{eqn:gaussian_wild_bootstrap_Delta1}
\Delta_1 = {1 \over n} \Big\|\sum_{i=1}^n \tilde{\vg}_i\tilde{\vg}_i^\top - \Gamma_g \Big\|,
\end{equation}
where $\tilde{\vg}_i = \vech(g(\vX_i))$. Then, 
\begin{equation*}
\E (\Delta_1) \le K \varpi_1(n,p)
\end{equation*}
for some absolute constant $K > 0$.
\end{lem}

\begin{proof}[Proof of Lemma \ref{lem:gaussian_wild_bootstrap_Delta1_exp_moment}]
Let $j,k,m,l=1,\cdots,p$ be such that $j \ge k$ and $m \ge l$ and write $\MAX = \max_{j \ge k, m \ge l}$. By \cite[Lemma 8]{cck2014b},
\begin{equation*}
\E (\Delta_1) \le {K_1 \over n} \Bigg\{ (\log{p})^{1\over2} \MAX \left[ \E \sum_{i=1}^n g_{i,jk}^2 g_{i,ml}^2 \right]^{1\over2} + (\log{p}) \left[ \E \max_{i \le n} \MAX g_{i,jk}^2 g_{i,ml}^2 \right]^{1\over2} \Bigg\}.
\end{equation*}
Note that $\E [\max_{i \le n} \MAX g_{i,jk}^2 g_{i,ml}^2]  = \E [\max_{i \le n} \max_{j \ge k} g_{i,jk}^4]$. By \cite[Lemma 2.2.2]{vandervaartwellner1996} and (\ref{eqn:subexponential-kernel-moment-condition}),
$$
\left\|\max_{i \le n} \max_{j \ge k} |g_{i,jk}| \right\|_{\psi_1} \le K_2 (\log(np)) \max_{i \le n} \max_{j \ge k} \left\|g_{i,jk} \right\|_{\psi_1} \le K_2 (\log(np)) B_n.
$$
Therefore, we have
$$
\E [\max_{i \le n} \max_{j \ge k} g_{i,jk}^4] \le K_3 \left\|\max_{i \le n} \max_{j \ge k} |g_{i,jk}| \right\|_{\psi_1}^4 \le K_3 (\log(np))^4 B_n^4.
$$
By the Cauchy-Schwarz inequality and (\ref{eqn:subexponential-kernel-moment-condition}), we have for all $j \ge k$ and $m \ge l$
\begin{eqnarray*}
\left[ \E \sum_{i=1}^n g_{i,jk}^2 g_{i,ml}^2 \right]^{1\over2} &\le& \left[ (\E \sum_{i=1}^n g_{i,jk}^4)^{1/2} (\E \sum_{i=1}^n g_{i,ml}^4)^{1/2} \right]^{1\over2} \\
&\le& \max_{j \ge k} (\E \sum_{i=1}^n g_{i,jk}^4)^{1/2} \le  (2n)^{1/2} B_n.
\end{eqnarray*}
Now, (\ref{eqn:varpi_1}) follows.
\end{proof}

\begin{lem}[Bound on $\E(\Delta_2^{1/2})$: sub-exponential kernel]
\label{lem:gaussian_wild_bootstrap_Delta2_exp_moment}
Let
\begin{equation}
\label{eqn:gaussian_wild_bootstrap_Delta2}
\Delta_2 = {1\over n} \max_{m,k} \sum_{i=1}^n [\hat{g}_{i,mk} - g_{mk}(\vX_i)]^2
\end{equation}
and
\begin{equation}
\label{eqn:varpi_2}
\varpi_2(n,p) = B_n \Big\{ {\log^{3/2}(np) \over n^{1/2}} \vee {\log^2(np) \over n} \Big\}.
\end{equation}
Assume (\ref{eqn:subexponential-kernel-moment-condition}). Then
\begin{equation*}
\E (\Delta_2^{1/2}) \le K \varpi_2(n,p)
\end{equation*}
for some absolute constant $K > 0$.
\end{lem}

\begin{proof}[Proof of Lemma \ref{lem:gaussian_wild_bootstrap_Delta2_exp_moment}]
Write $\MAX=\max_{m,k} \max_{i \le n}$. By the definition of $\hat{g}_i$ and $g(\vX_i)$, we have
\begin{eqnarray}
\label{eqn:Delta_2_decomposition}
\Delta_2^{1/2} &\le& \max_{m,k} \max_{i \le n} \left| \hat{g}_{i,mk} - g_{mk}(\vX_i) \right| \\ \nonumber
&\le& {1\over n} \MAX \Big| \sum_{j=1}^n h_{mk}(\vX_i,\vX'_j) - \E[h_{mk}(\vX_i,\vX'_j) | \vX_1^n] \Big| \\ \nonumber
&& \qquad + {n \choose 2}^{-1} \max_{m,k} \Big| \sum_{1\le j < l \le n} h_{mk}(\vX'_j,\vX'_l) - \E h_{mk} \Big|.
\end{eqnarray}
By \cite[Lemma 8]{cck2014b} conditional on $\vX_1^n$, we have
\begin{eqnarray*}
&& \E \Big[ \MAX \Big| \sum_{j=1}^n h_{mk}(\vX_i,\vX'_j) - \E[h_{mk}(\vX_i,\vX'_j) | \vX_1^n]  \Big| \mid \vX_1^n \Big] \\
&\le& K_1 \Big\{ (\log(np))^{1/2} \MAX \Big[ \sum_{j=1}^n \E\Big(h_{mk}^2(\vX_i,\vX'_j) \mid \vX_1^n \Big) \Big]^{1/2} \\
&& \qquad + (\log(np)) \Big[ \E\Big(\MAX \max_{j \le n} h_{mk}^2(\vX_i,\vX'_j) \mid \vX_1^n \Big) \Big]^{1/2} \Big\}.
\end{eqnarray*}
By \cite[Lemma 2.2.2]{vandervaartwellner1996} and (\ref{eqn:subexponential-kernel-moment-condition}), we have
$$
\Big\| \max_{m,k} \max_{i,j} |h_{mk}(\vX_i, \vX'_j)| \Big\|_{\psi_1} \le K_2 \log(np) \max_{m,k} \max_{i,j} \|h_{mk}(\vX_i,\vX_j')\|_{\psi_1} \le K_2 (\log(np)) B_n.
$$
Then, we have by Jensen's inequality that
$$
\E \Big[ \E \Big(\MAX \max_{j \le n} h_{mk}^2(\vX_i,\vX'_j) \mid \vX_1^n \Big) \Big]^{1/2} \le 2 K_2 (\log(np)) B_n.
$$
By Jensen's inequality twice and \cite[Lemma 2.2.2]{vandervaartwellner1996}, we have
\begin{eqnarray*}
&& \E \Big\{ \MAX \Big[ \sum_{j=1}^n \E\Big(h_{mk}^2(\vX_i,\vX'_j) \mid \vX_1^n \Big) \Big]^{1/2} \Big\} \\
&\le& n^{1/2} \Big\{ \E \Big[ \MAX   \Big( \E[h_{mk}^2(\vX_i,\vX'_1) \mid \vX_1^n] \Big) \Big] \Big\}^{1/2} \\
&\le& n^{1/2} \Big\{ \E \Big[ \max_{m,k} \max_{i \le n} h_{mk}^2(\vX_i,\vX'_1) \Big] \Big\}^{1/2} \le K_3 n^{1/2} (\log(np)) B_n.
\end{eqnarray*}
Then
\begin{eqnarray*}
{1\over n} \E \Big[ \MAX \Big| \sum_{j=1}^n h_{mk}(\vX_i,\vX'_j) - \E [h_{mk}(\vX_i,\vX'_j) \mid \vX_1^n] \Big| \Big] \le K_4 \varpi_2(n,p).
\end{eqnarray*}
Finally, by (\ref{eqn:reducing-Ustat-to-iid-sum}) and (\ref{eqn:reducing-Ustat-to-iid-sum-bernstein-bound}), the expectation of the second term in the last inequality of (\ref{eqn:Delta_2_decomposition}) is bounded by $K_5 B_n \{(\log(p)/n)^{1/2} \vee (\log(p)/n)\}$, which is on the smaller order of $\varpi_2(n,p)$. The lemma now follows.
\end{proof}

\begin{lem}[Bound on $\E\|W\|$: sub-exponential kernel]
\label{lem:gaussian_wild_bootstrap_W_exp_moment}
Let 
\begin{equation}
\label{eqn:bound_W_exp_moment}
\varpi_3(n,p) =  {(\log{p})^{3/2} (\log(np)) \over n} B_n + {\log{p}\over n^{1/2}} +  {(\log{p})^{5/4} \over n^{3/4}} B_n^{1/2}.
\end{equation}
Assume (\ref{eqn:subexponential-kernel-moment-condition}). If $2 \le p \le \exp(b n)$ for some absolute constant $b > 0$, then
\begin{equation*}
\E \|W\| \le K (1 \vee b^{1/2}) \varpi_3(n,p)
\end{equation*}
for some absolute constant $K > 0$.
\end{lem}

\begin{proof}[Proof of Lemma \ref{lem:gaussian_wild_bootstrap_W_exp_moment}]
By Lemma \ref{lem:moment-bounds-subexponential-kernel}, we have $\|M\|_q \le K q! B_n \log(np)$ and $D_2 \le 2^{1/2}$, $D_4 \le 2^{1/4} B_n^{1/2}$. By Theorem \ref{thm:expectation-bound}, we have
$$
\E \|W\| \le K (1 \vee b^{1/2}) n^{1/2} \Big[ \Big({ \log{p} \over n}\Big)^{3/2} B_n \log(np) + \Big({ \log{p} \over n}\Big) + \Big({ \log{p} \over n}\Big)^{5/4} B_n^{1/2} \Big],
$$
from which lemma follows.
\end{proof}

\begin{lem}[Bound on $\E(\Delta_1)$: uniform polynomial moment]
\label{lem:gaussian_wild_bootstrap_Delta1_unifpoly_moment}
Let $q\ge4$. Assume (\ref{eqn:unifpoly-kernel-moment-condition}).
Let
\begin{equation}
\label{eqn:varpi_1_unifpoly}
\varpi_1(n,p) = \left({\log{p}\over n}\right)^{1/2} B_n + {\log{p} \over n^{1-2/q}} B_n^2.
\end{equation}
Then, 
\begin{equation*}
\E (\Delta_1) \le K \varpi_1(n,p),
\end{equation*}
where $\Delta_1$ is defined in (\ref{eqn:gaussian_wild_bootstrap_Delta1}) and $K > 0$ is an absolute constant.
\end{lem}

\begin{proof}[Proof of Lemma \ref{lem:gaussian_wild_bootstrap_Delta1_unifpoly_moment}]
The proof is similar to the argument of Lemma \ref{lem:gaussian_wild_bootstrap_Delta1_exp_moment} with the difference in bounding $\E [\max_{i \le n} \max_{j \ge k} g_{i,jk}^4]$. By Jensen's inequality,
$$
\left\|\max_{i \le n} \max_{j \ge k} {|g_{i,jk}| \over B_n} \right\|_4 \le \left\|\max_{i \le n} \max_{j \ge k} {|g_{i,jk}| \over B_n} \right\|_q \le n^{1/q} \max_{i \le n} \left\| \left\|{g_i \over B_n} \right\| \right\|_q \le n^{1/q}.
$$
Therefore, we have
$$
\E [\max_{i \le n} \max_{j \ge k} g_{i,jk}^4]  \le n^{4/q} B_n^4.
$$
\end{proof}

\begin{lem}[Bound on $\E(\Delta_2^{1/2})$: kernel with uniform polynomial moment]
\label{lem:gaussian_wild_bootstrap_Delta2_unifpoly_moment}
Let $q \ge 4$ and
\begin{equation}
\label{eqn:varpi_2_unifpoly}
\varpi_2(n,p) = {B_n (\log(np))^{1/2} \over n^{1/2-1/q}} \left[ 1 \vee {(\log(np))^{1/2} \over n^{1/2-1/q}}\right].
\end{equation}
Assume (\ref{eqn:unifpoly-kernel-moment-condition}). Then
\begin{equation*}
\E (\Delta_2^{1/2}) \le K \varpi_2(n,p),
\end{equation*}
where $\Delta_2$ is defined in (\ref{eqn:gaussian_wild_bootstrap_Delta2}) and $K > 0$ is an absolute constant.
\end{lem}

\begin{proof}[Proof of Lemma \ref{lem:gaussian_wild_bootstrap_Delta2_unifpoly_moment}]
The proof is similar to the argument in Lemma \ref{lem:gaussian_wild_bootstrap_Delta2_exp_moment}. We only note the differences. First, under (\ref{eqn:unifpoly-kernel-moment-condition}),
$$
\|\max_{m,k} \max_{i,j} |h_{mk}(\vX_i,\vX'_j)|\|_q \le n^{2/q} B_n,
$$
which in combination with the conditional Jensen's inequality imply that
$$
\E \Big[ \E\Big(\max_{m,k} \max_{i,j} h_{mk}^2(\vX_i,\vX'_j) \mid \vX_1^n \Big) \Big]^{1/2} \le n^{2/q} B_n.
$$
Second, by Jensen's inequality twice and \cite[Lemma 2.2.2]{vandervaartwellner1996}, we have
\begin{eqnarray*}
&& \E \Big\{ \MAX \Big[ \sum_{j=1}^n \E\left(h_{mk}^2(\vX_i,\vX'_j) \mid \vX_1^n \right) \Big]^{1/2} \Big\} \\
&\le& n^{1/2} \Big\{ \E \Big[ \max_{m,k} \max_{i \le n} h_{mk}^2(\vX_i,\vX'_1) \Big] \Big\}^{1/2} \\
&\le& n^{1/2} \|\max_{m,k} \max_{i \le n} |h_{mk}(\vX_i,\vX'_1)|\|_q \\
&\le& n^{1/2+1/q} \left\| \|h \| \right\|_q \le n^{1/2+1/q} B_n.
\end{eqnarray*}
Then
\begin{eqnarray*}
{1\over n} \E \Big[ \MAX \Big| \sum_{j=1}^n h_{mk}(\vX_i,\vX'_j) - \E [h_{mk}(\vX_i,\vX'_j) \mid \vX_1^n] \Big| \Big] \le K_1 \varpi_2(n,p).
\end{eqnarray*}
By (\ref{eqn:reducing-Ustat-to-iid-sum}) and \cite[Lemma 8]{cck2013}, we have
\begin{eqnarray*}
&& \E \Big[ {n \choose 2}^{-1} \max_{m,k} \Big| \sum_{1\le j < l \le n} h_{mk}(\vX'_j,\vX'_l) - \E h_{mk} \Big| \Big] \\
&\le& K_2 \Big\{ {(\log{p})^{1/2} \over n} \max_{m,k} \Big[ \sum_{i=1}^{n/2} \E h_{mk}^2(X_i, X_{i+n/2}) \Big]^{1/2} \\
&& \qquad \qquad + \Big({\log{p} \over n}\Big) \Big[ \E \Big( \max_{1 \le i \le n/2} \|h(\vX_i, \vX_{i+n/2})\|^2 \Big) \Big]^{1/2} \Big\} \\
&\le& K_3 \Big\{ \Big({\log{p} \over n}\Big)^{1/2} + {B_n \log{p} \over n^{1-1/q}} \Big\}.
\end{eqnarray*}
Since the last bound is of smaller order than $\varpi_2(n,p)$, the remaining proof follows from the argument in Lemma \ref{lem:gaussian_wild_bootstrap_Delta2_exp_moment}.
\end{proof}

\begin{lem}[Bound on $\E\|W\|$: kernel with uniform polynomial moment]
\label{lem:gaussian_wild_bootstrap_W_unifpoly_moment}
Let $q\ge4$ and 
\begin{equation}
\label{eqn:bound_W_unifpoly_moment}
\varpi_3(n,p) =  {(\log{p})^{3/2} \over n^{1-2/q}} B_n + {\log{p}\over n^{1/2}} +  {(\log{p})^{5/4} \over n^{3/4}} B_n^{1/2}.
\end{equation}
Assume (\ref{eqn:unifpoly-kernel-moment-condition}). If $2 \le p \le \exp(b n)$ for some absolute constant $b > 0$, then
\begin{equation*}
\E \|W\| \le K (1 \vee b^{1/2}) \varpi_3(n,p)
\end{equation*}
for some absolute constant $K > 0$.
\end{lem}

\begin{proof}
By Lemma \ref{lem:moment-bounds-unifpoly-kernel}, we have $\|M\|_4 \le \|M\|_q \le 2 n^{2/q} B_n$, $D_2 \le 1$ and $D_4 \le B_n^{1/2}$. By Theorem \ref{thm:expectation-bound}, we have
$$
\E \|W\| \le  K (1 \vee b^{1/2}) n^{1/2} \Big[ \Big({ \log{p} \over n}\Big)^{3/2} B_n n^{2/q} + \Big({ \log{p} \over n}\Big) + \Big({ \log{p} \over n}\Big)^{5/4} B_n^{1/2} \Big].
$$
\end{proof}

\section{Additional numerical comparisons}
\label{sec:numerical_comparison}

We present more numerical comparisons of the Gaussian approximation on the H\'ajek projection and its wild bootstrap version for the covariance matrix. We consider two mean-zero distributions from the elliptical family \cite{muirhead1982}:
\begin{enumerate}
\item[(M1)] (sub-exponential moment) The $\varepsilon$-contaminated $p$-variate elliptical normal distribution with density function
\begin{eqnarray}
\nonumber
f(\vx; \varepsilon, \nu, V) &=& {1-\varepsilon \over (2\pi)^{p/2} \det(V)^{1/2}} \exp\left(-{\vx^\top V^{-1} \vx \over2} \right) \\ \label{eqn:eps_contaminated_normal_distn}
&& \qquad + {\varepsilon \over (2 \pi \nu^2)^{p/2} \det(V)^{1/2}} \exp\left(-{\vx^\top V^{-1} \vx \over2 \nu^2} \right);
\end{eqnarray}
\item[(M2)] (polynomial moment) The $p$-variate elliptical $t$-distribution with degree of freedom $\nu$ and density function
\begin{equation}
\label{eqn:elliptic_t_distn}
f(\vx; \nu, V) = {\Gamma(\nu+p)/2 \over \Gamma(\nu/2) (\nu \pi)^{p/2} \det(V)^{1/2}} \left( 1 + {\vx^\top V^{-1} \vx \over \nu} \right)^{-(\nu+p)/2}.
\end{equation}
\end{enumerate}

For the positive-definite matrix $V$, we consider three dependence models:
\begin{enumerate}
\item[(D1)] strong dependence model with $V = 0.9\times\vone_p\vone_p^\top + 0.1\times \Id_p$, where $\vone_p$ is the $p \times 1$ vector of all ones;
\item[(D2)] moderate dependence AR(1) model with $V = \{v_{mk}\}_{m,k=1}^p$ and $v_{mk} = 0.7^{|m-k|}$;
\item[(D3)] weak dependence AR(1) model with $V = \{v_{mk}\}_{m,k=1}^p$ and $v_{mk} = 0.3^{|m-k|}$.
\end{enumerate}
 
We use $\varepsilon=0.2$ and $\nu=1.5$ in (M1) and $\nu=10$ in (M2). For the chosen parameters, the two distributions have the same variance scaling for each $V$, while the kurtosis of the sub-exponential case is $0.16$ and the polynomial case is $1/3$. We compare the finite sample performance on $n=500$ and $p=40$ so that there are 820 covariance parameters. In each setup, we compare the approximation quality of $\bar{T}_0$ using $\bar{Z}_0$ and $\hat{L}^*_0$. All results are reported over 5000 simulation runs.

\begin{figure}[t!] 
   \centering
      \subfigure{\label{subfig:approx_wild_bootstrap_eps_contaminated_n=200_p=40} \includegraphics[scale=0.22]{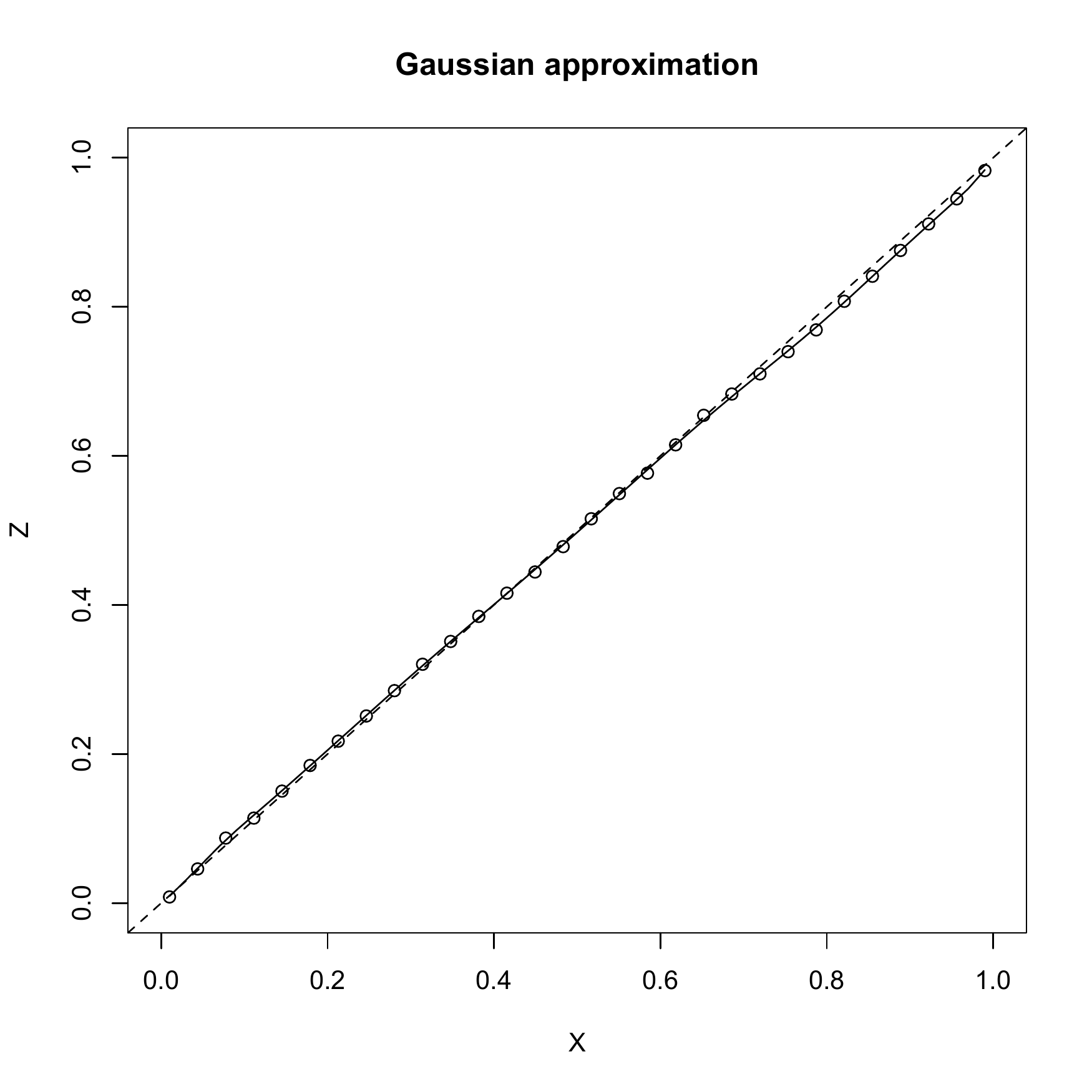}}
      \subfigure{\label{subfig:approx_wild_bootstrap_eps_contaminated_n=200_p=40} \includegraphics[scale=0.22]{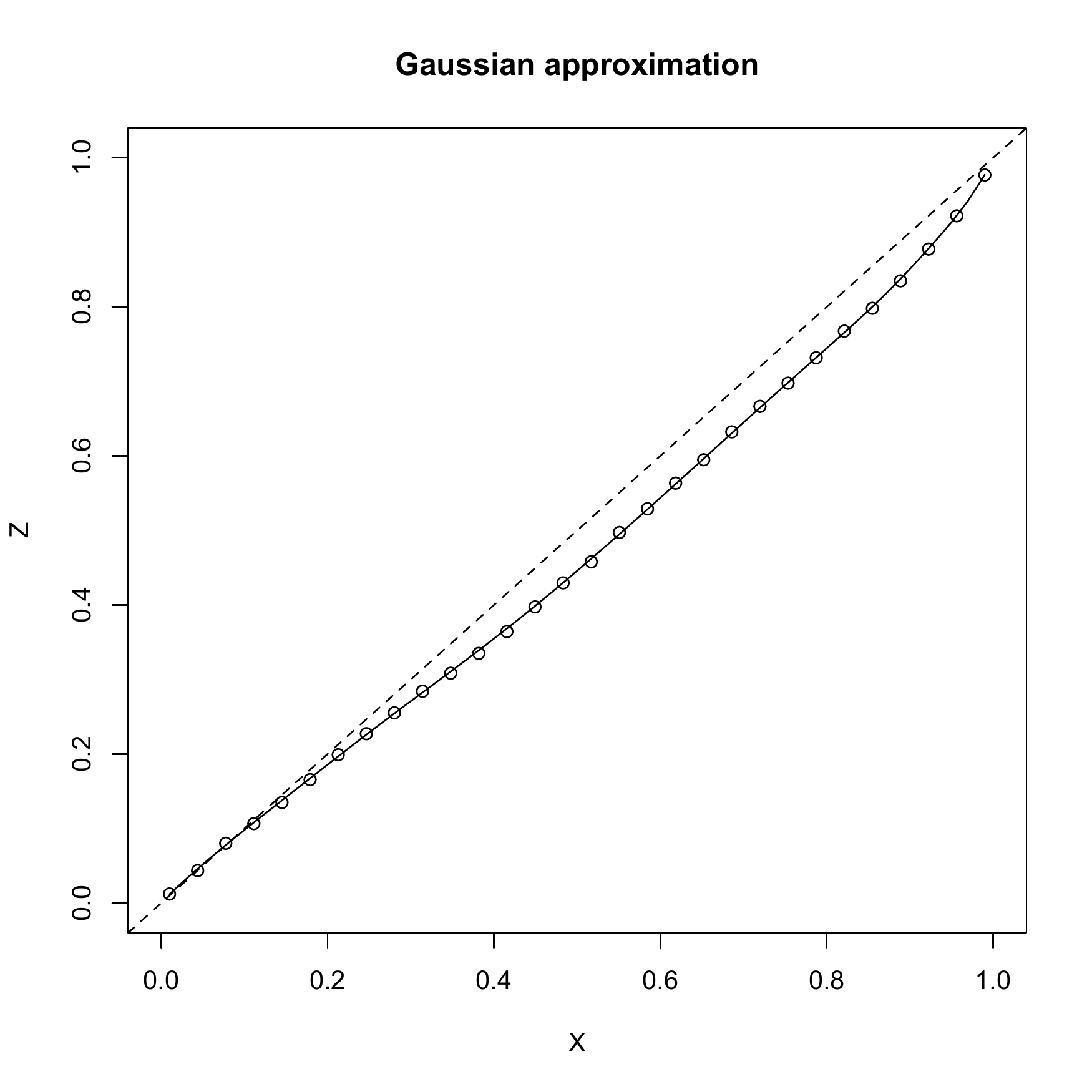}}
      \subfigure{\label{subfig:approx_wild_bootstrap_eps_contaminated_n=200_p=40} \includegraphics[scale=0.22]{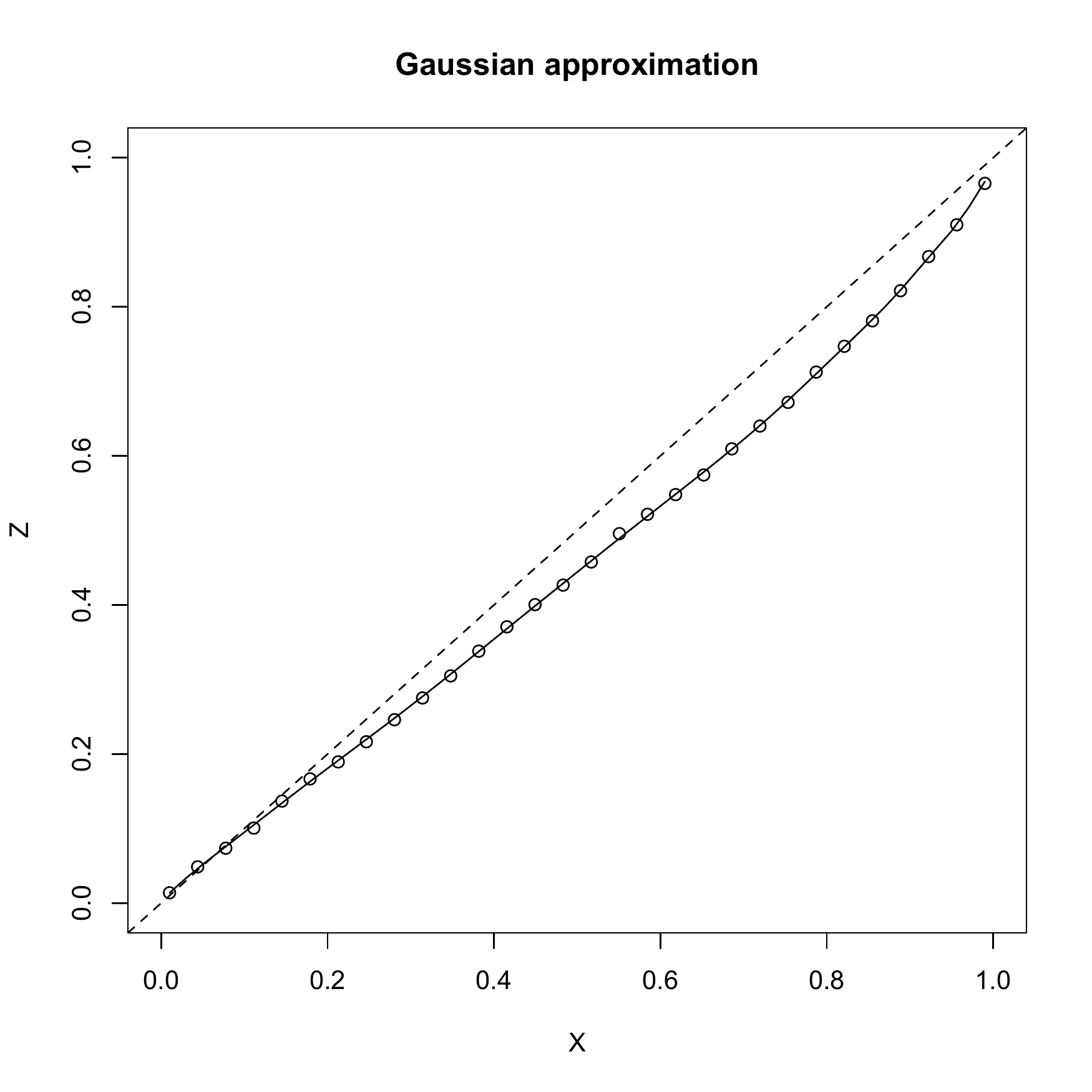}}  \\
      
	\subfigure {\label{subfig:approx_wild_bootstrap_n=200_p=40}\includegraphics[scale=0.22]{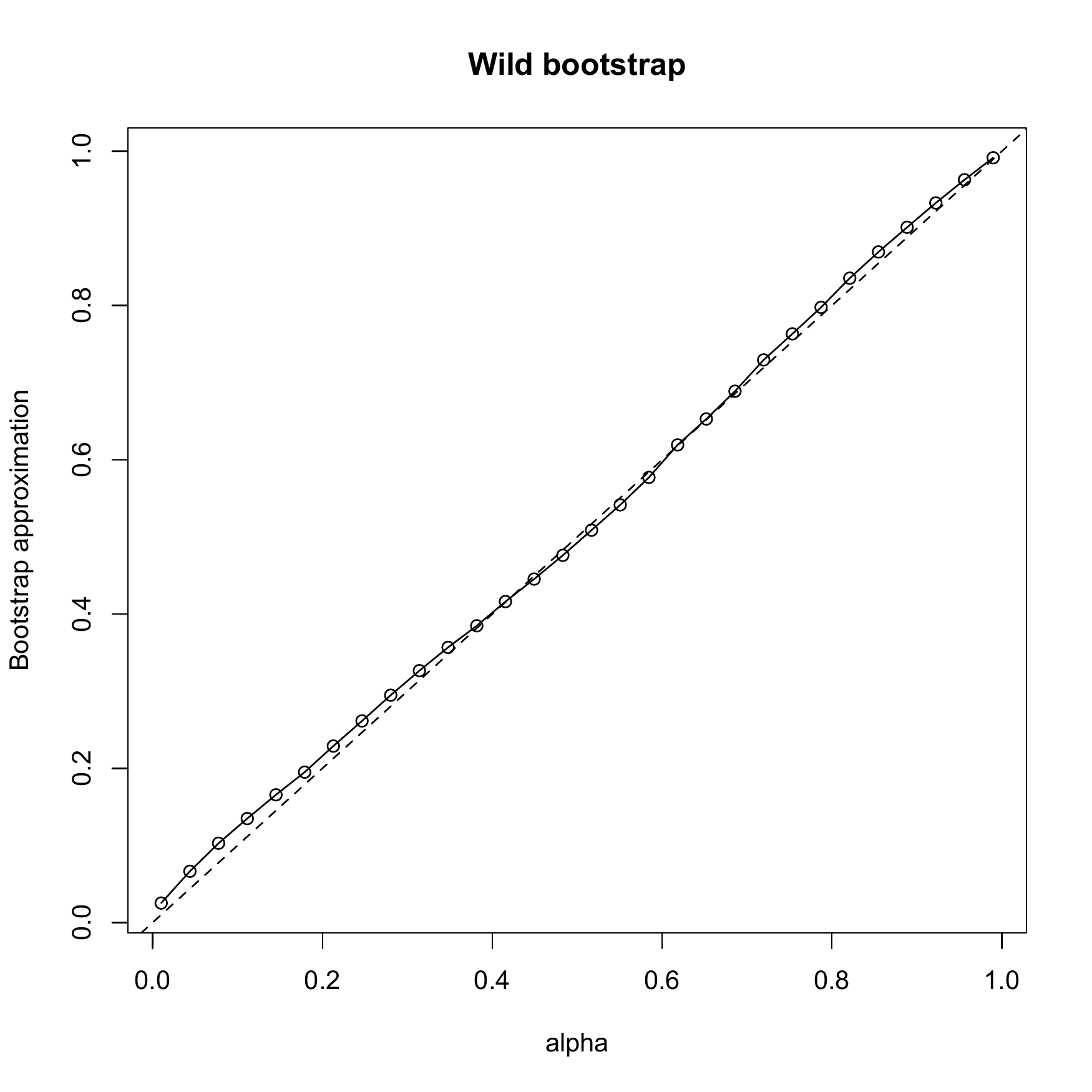}} 
	\subfigure {\label{subfig:approx_wild_bootstrap_n=200_p=40}\includegraphics[scale=0.22]{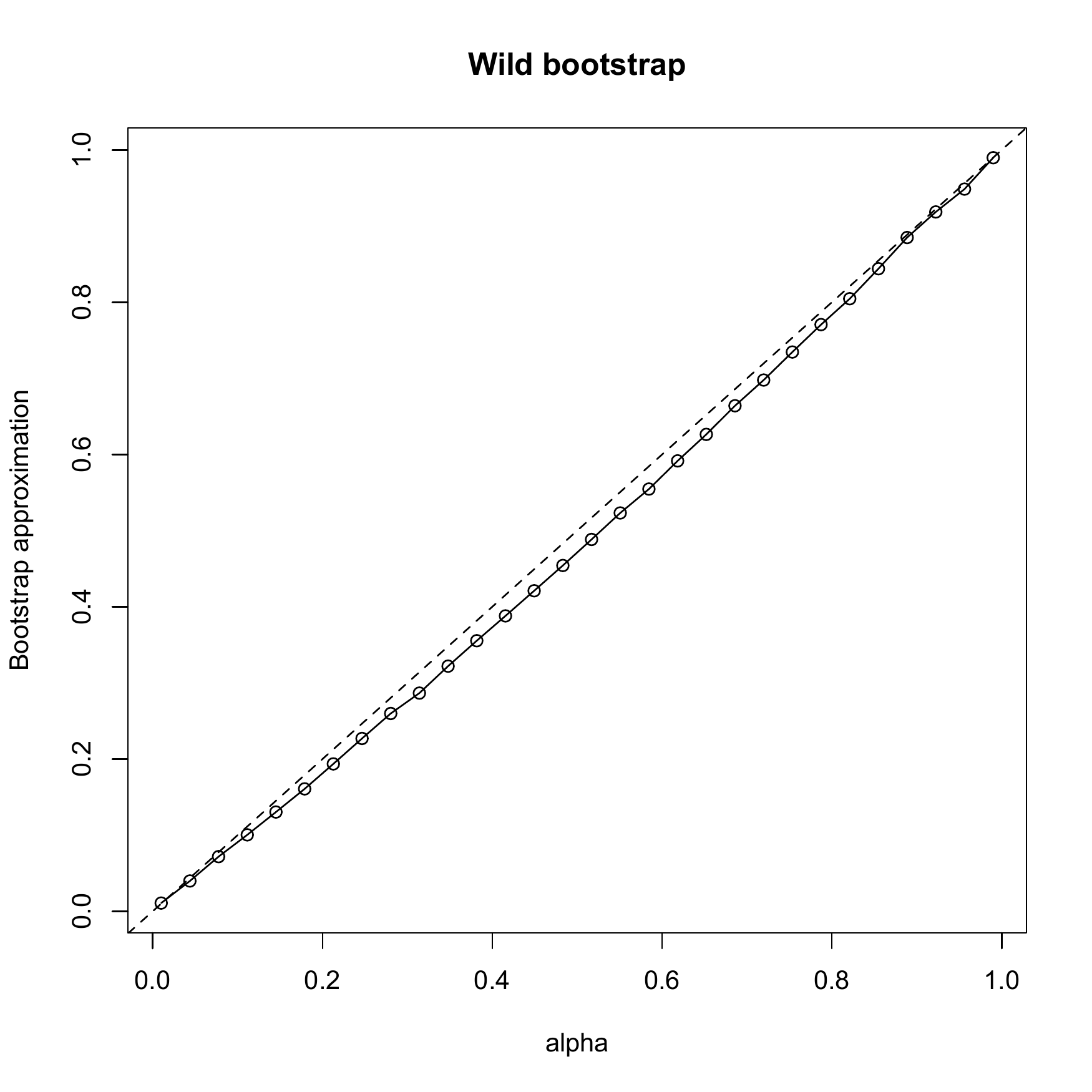}}
	\subfigure {\label{subfig:approx_wild_bootstrap_n=200_p=40}\includegraphics[scale=0.22]{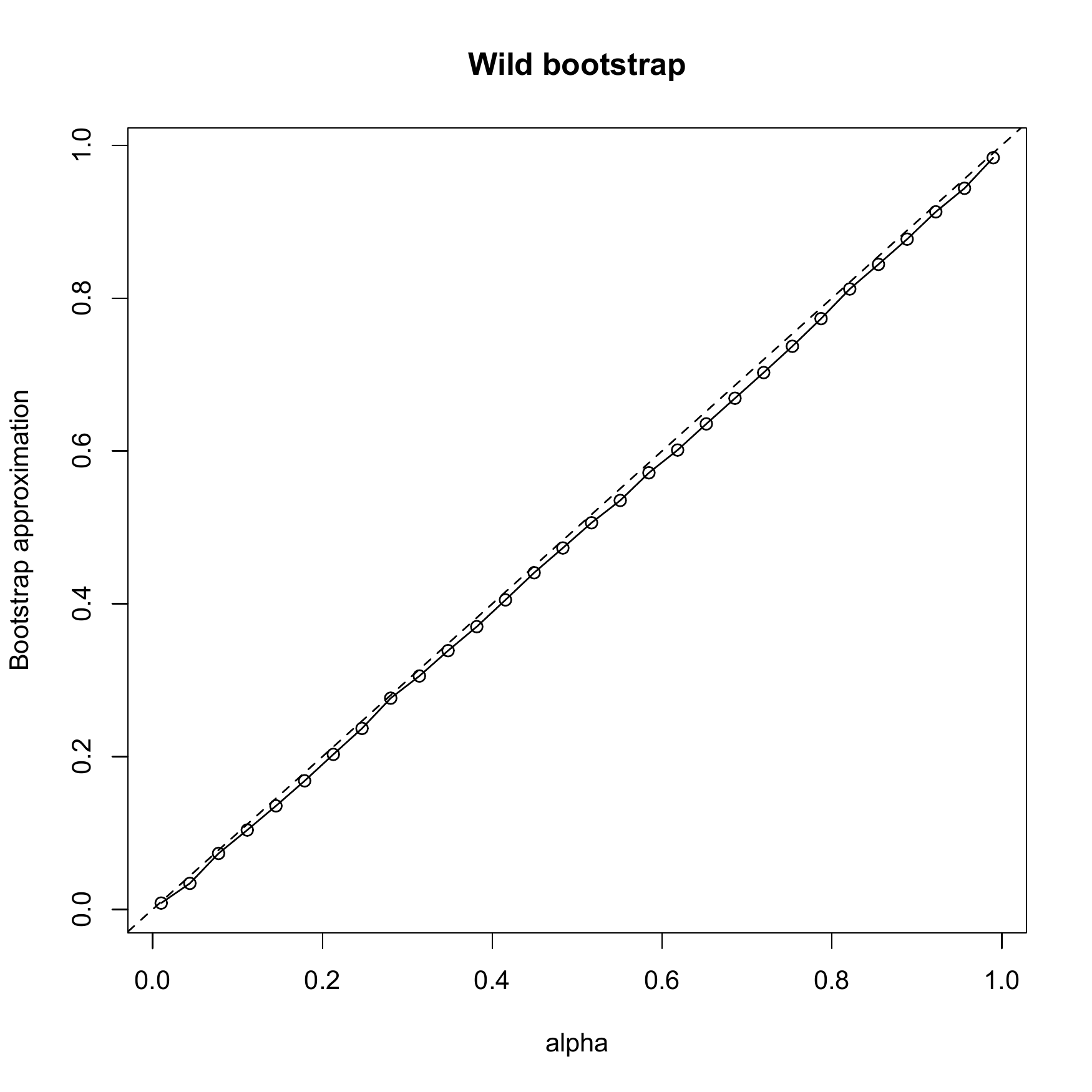}} \\

   \caption{Gaussian approximations (top row) by $\bar{Z}_0$  and wild bootstrap approximations (bottom row) by $\hat{L}^*_0$ for the $\varepsilon$-contaminated normal distribution model: left (M1)+(D1), middle (M1)+(D2), and right (M1)+(D3). Sample size $n=500$ and dimension $p=40$.}
   \label{fig:p=40_n=500_eps_cont_normal}
\end{figure}

\begin{figure}[t!] 
   \centering
	\subfigure{\label{subfig:approx_wild_bootstrap_eps_contaminated_n=200_p=40} \includegraphics[scale=0.22]{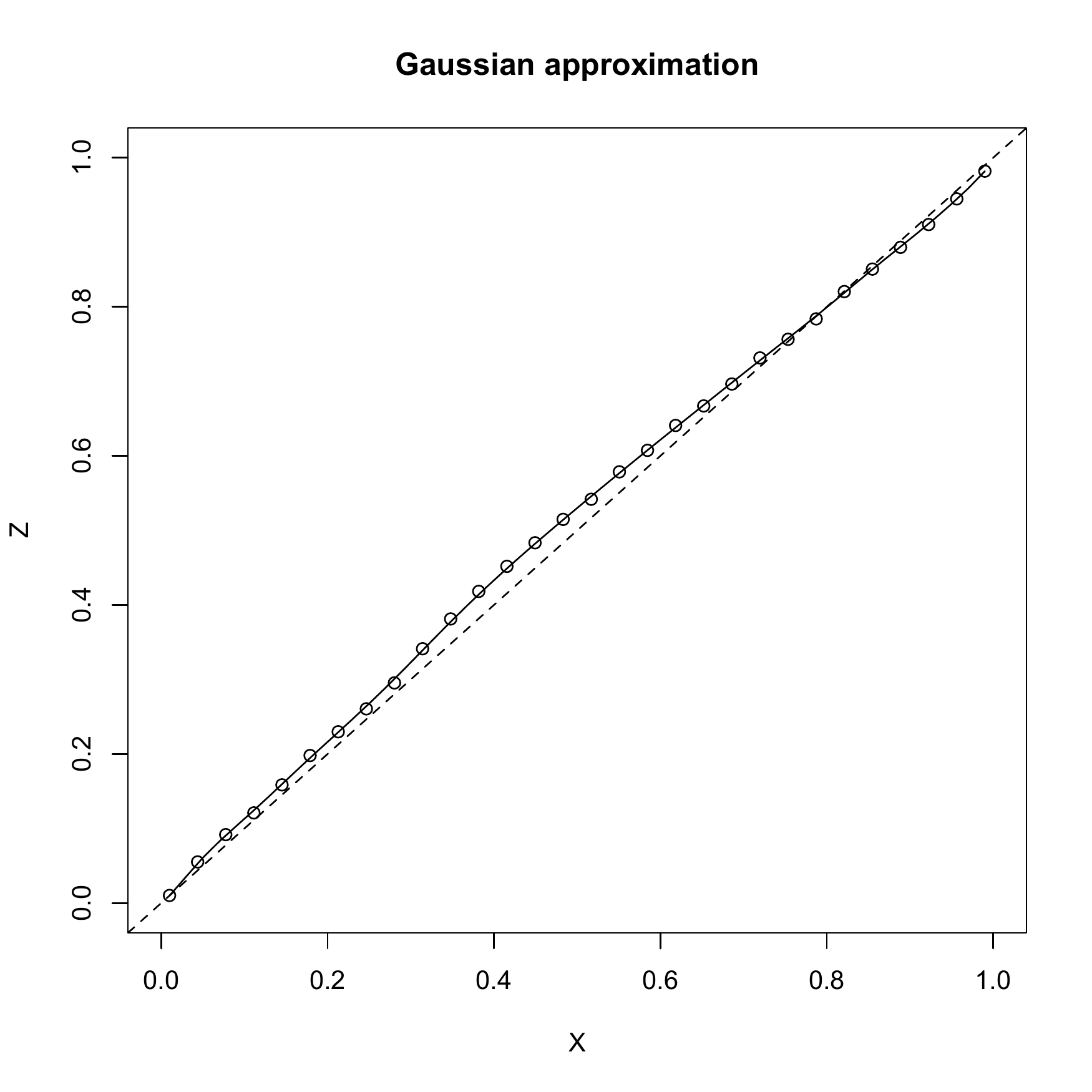}}
	\subfigure{\label{subfig:approx_wild_bootstrap_eps_contaminated_n=200_p=40} \includegraphics[scale=0.22]{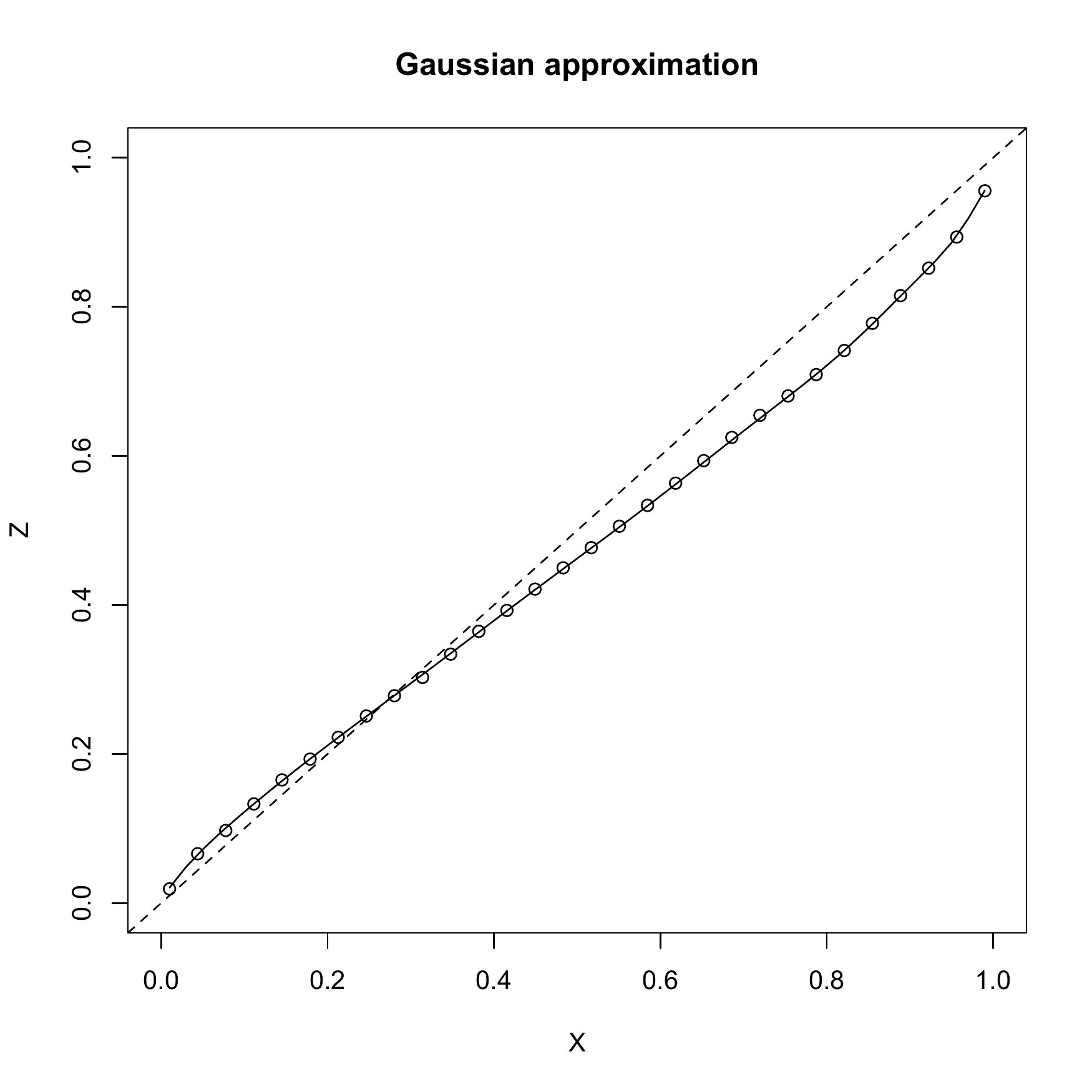}}
	\subfigure{\label{subfig:approx_wild_bootstrap_eps_contaminated_n=200_p=40} \includegraphics[scale=0.22]{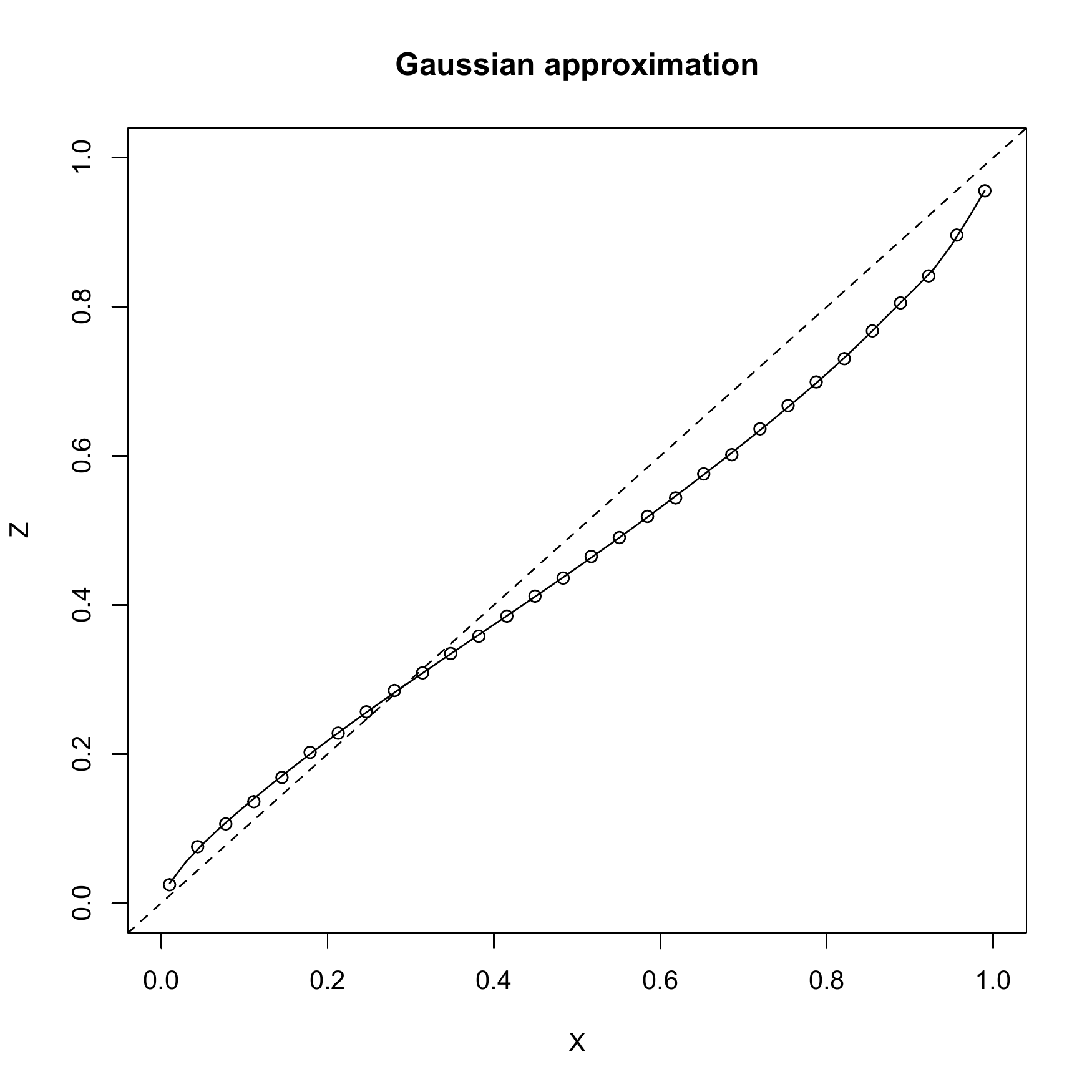}}  \\
	
	\subfigure {\label{subfig:approx_wild_bootstrap_n=200_p=40}\includegraphics[scale=0.22]{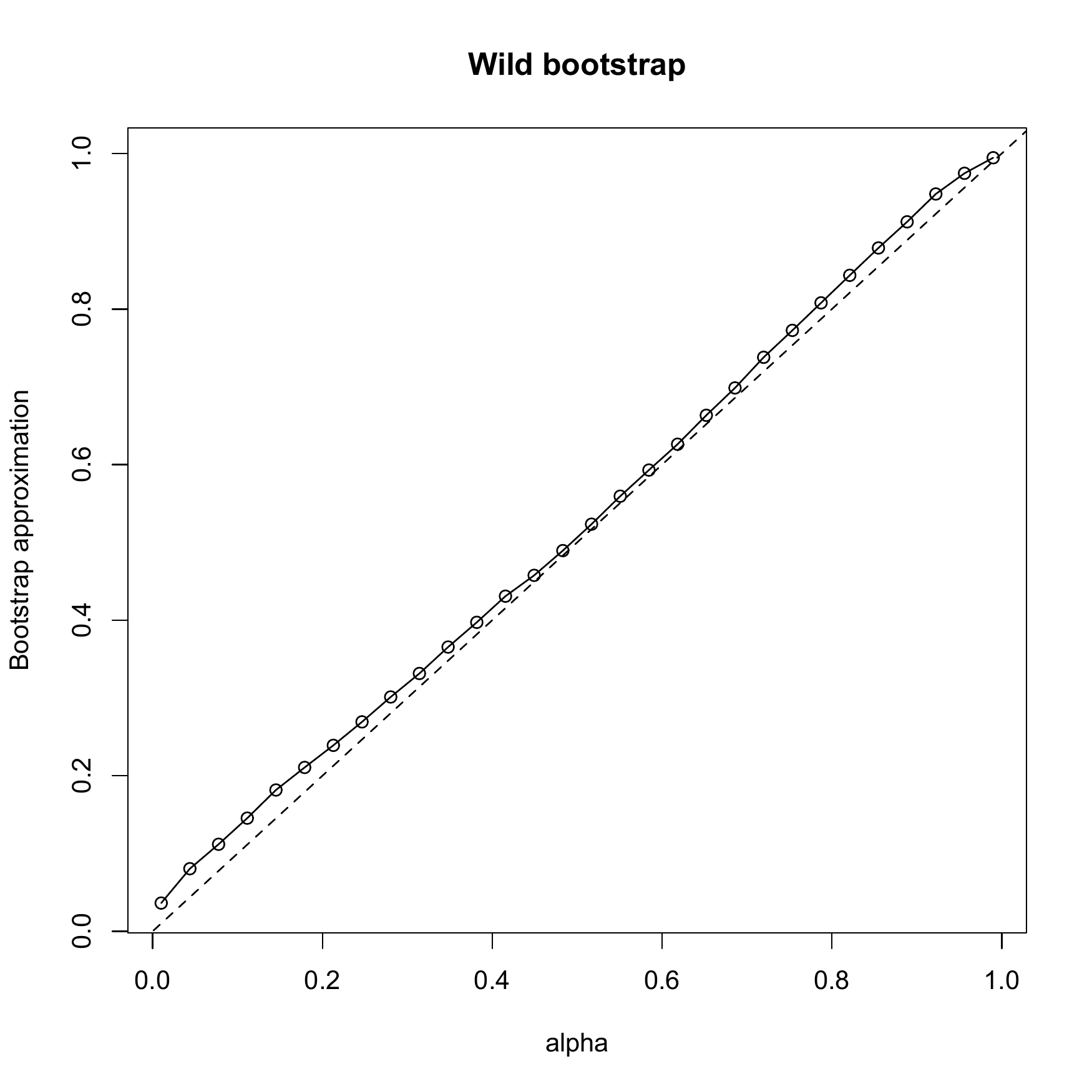}}
	\subfigure {\label{subfig:approx_wild_bootstrap_n=200_p=40}\includegraphics[scale=0.22]{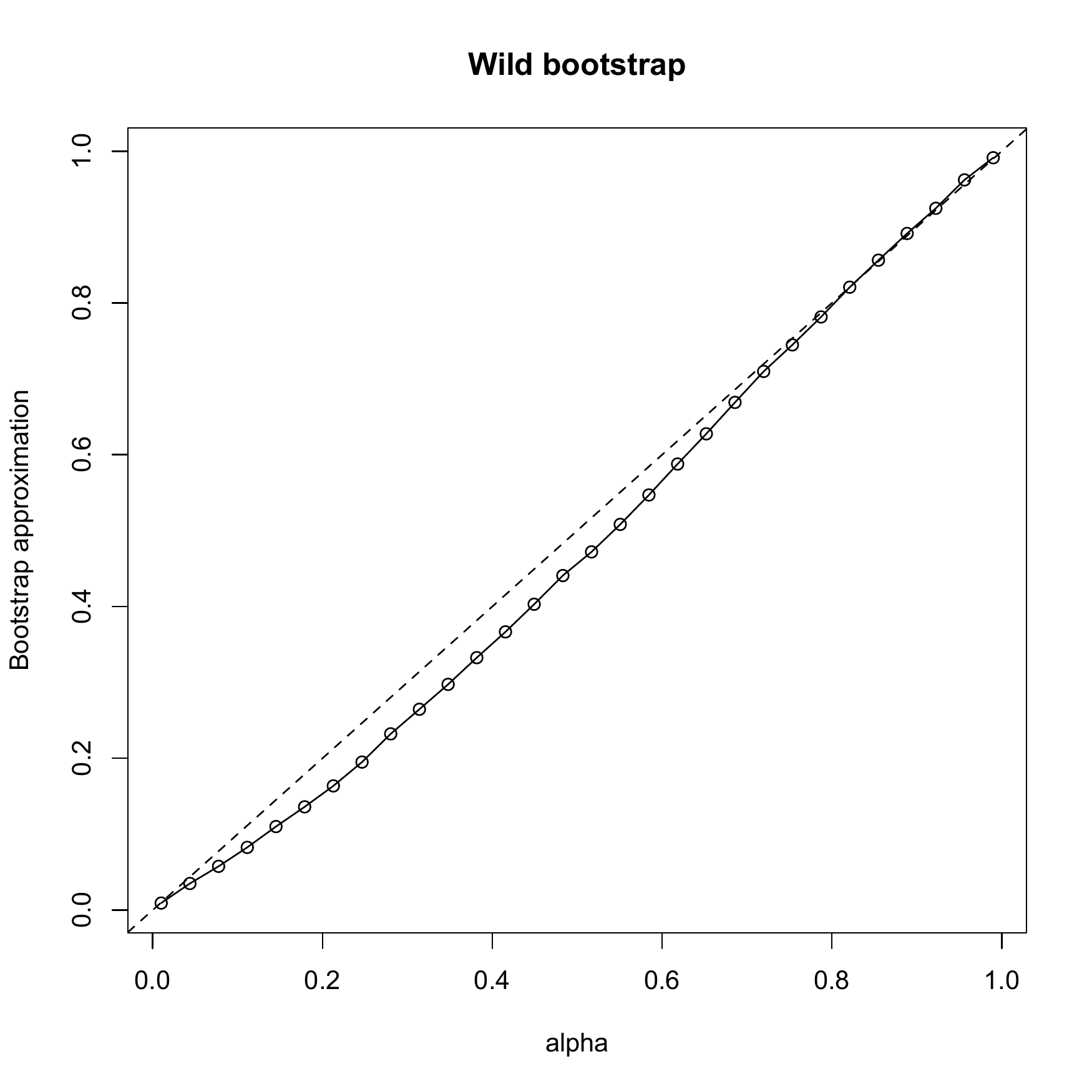}}
	\subfigure {\label{subfig:approx_wild_bootstrap_n=200_p=40}\includegraphics[scale=0.22]{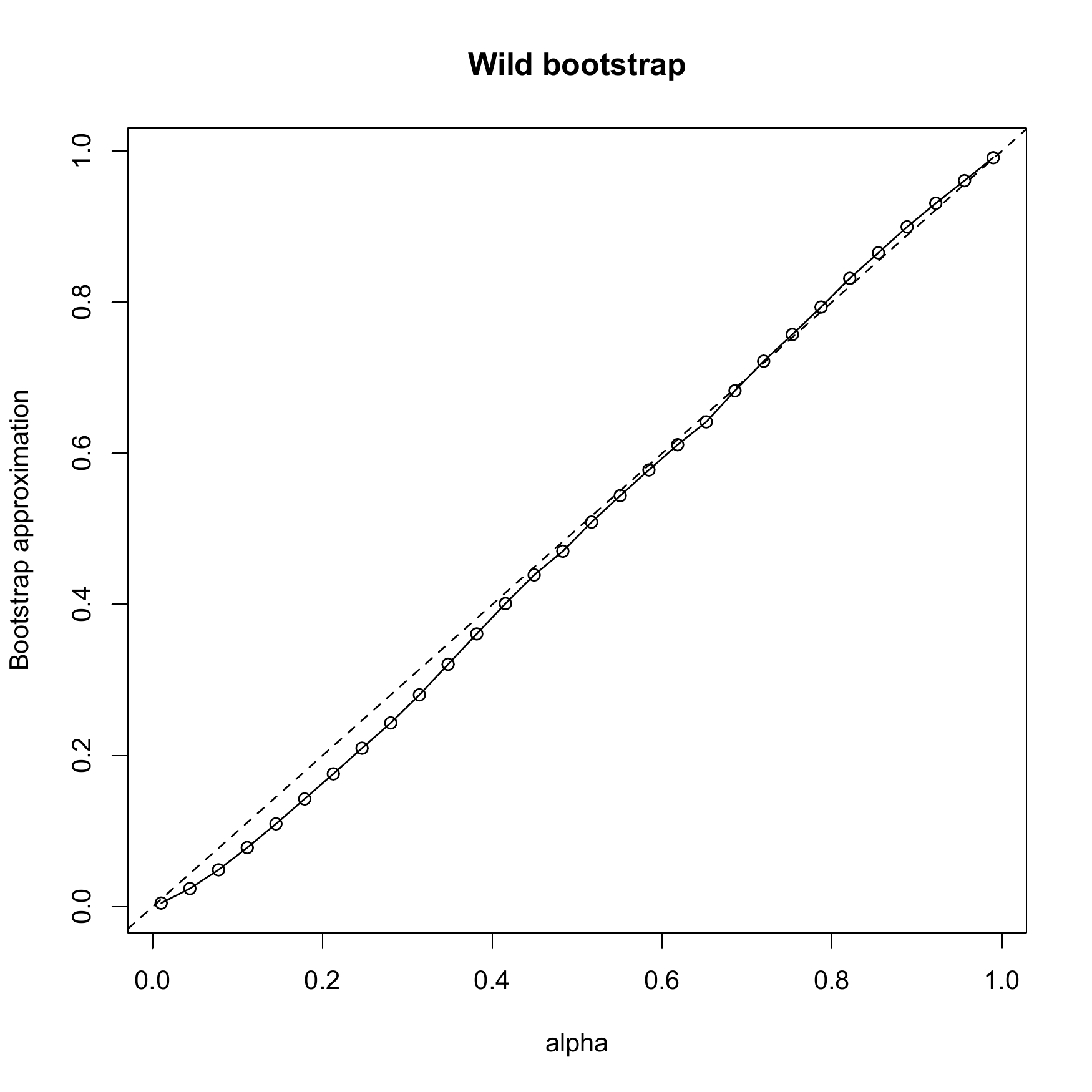}} \\
	
   \caption{Gaussian approximations (top row) by $\bar{Z}_0$  and wild bootstrap approximations (bottom row) by $\hat{L}^*_0$ for the elliptic $t$-distribution model: left (M2)+(D1), middle (M2)+(D2), and right (M2)+(D3). Sample size $n=500$ and dimension $p=40$.}
   \label{fig:p=40_n=500_elliptical_t}
\end{figure}

First, Figure \ref{fig:p=40_n=500_eps_cont_normal} shows a better approximation than Figure \ref{fig:p=40_n=500_elliptical_t} for $\bar{Z}_0$ and $\hat{L}_0^*$. This is predicted by our theory in Section \ref{sec:gaussian-approx} and \ref{subsec:gassuain_wild_bootstrap}. Second, the approximation becomes more accurate as the dependence gets stronger in both Figure \ref{fig:p=40_n=500_eps_cont_normal} and Figure \ref{fig:p=40_n=500_elliptical_t}. Third, the wild bootstrap has high-quality approximation for the upper tail probabilities, which is particularly relevant for statistical applications; see Section \ref{sec:stat_apps} below. This occurs even when the Gaussian approximation has deteriorated performance such as in the weakly dependent AR(1) models (D2) and (D3). One possible explanation for this phenomenon can be the numeric instability for simulating the $p'\times1$ normal random vectors $Z_i$ in the approximation, where $p'=p(p+1)/2$. On the contrary, the wild bootstrap only requires the simulation of $n$ univariate normal random variables.

\section{Two additional application examples}
\label{sec:more_examples}

In this section, we provide two more examples for applying the Gaussian wild bootstrap procedure. We only state results for subgaussian observations. For the uniform polynomial moment case, we can easily obtain similar results as in Section \ref{sec:stat_apps}. For a matrix $\Theta = \{\theta_{mk}\}_{m,k=1}^p$ and a vector $\mbf\theta$, we write $|\Theta|_{L^1} = \max_{1 \le k \le p}\sum_{m=1}^p |\theta_{mk}|$ is the matrix $L^1$-norm of $\Theta$ and $|\mbf\theta|_w = (\sum_{j=1}^p |\theta_j|^w)^{1/w}, w\ge1$, is the $\ell^w$-norm of $\mbf\theta$, where $|\mbf\theta|_\infty = \max_{1 \le j \le p} |\theta_j| $ is the max-norm.

\subsection{Estimation of the sparse precision matrix}
\label{subsec:sparse_prec_mat}

Precision matrix, i.e. the inverse of the covariance matrix $\Omega = \Sigma^{-1}$, is an important object in high-dimensional statistics because it closely ties to the Gaussian graphical models and partial correlation graphs \cite{meinshausenbuhlmann2006,yuanlin2007,rothmanbickellevinazhu2008a,pengwangzhouzhu2009a,yuan2010a,cailiuluo2011a}. For multivariate Gaussian observations $\vX_i$, zero entries in the precision matrix correspond to missing edges in the graphical models; i.e. $\omega_{mk} = 0$ means that $X_m$ and $X_k$ are conditionally independent given the values of all other variables \cite{dempster1972,lauritzen1996a}. To avoid the overfitting for graphical models with a large number of nodes, the {\it sparsity} is a widely considered structural assumption. Here, we consider the estimation of $\Omega$ by using the CLIME method \cite{cailiuluo2011a}
\begin{equation}
\label{eqn:clime}
\hat\Omega(\lambda) = \text{argmin}_{\Theta \in \mathbb{R}^{p \times p}} |\Theta|_1 \quad \text{subject to} \quad \|\hat{S} \Theta - \Id_{p \times p}\| \le \lambda,
\end{equation}
where $\lambda \ge 0$ is a tuning parameter to control the sparsity in $\hat\Omega(\lambda)$ and $|\Theta|_1 = \sum_{m,k=1}^p |\theta_{mk}|$. As in the thresholded covariance matrix estimation case in Section \ref{subsec:tuning_selection_thresholded_cov_mat}, the performance of CLIME depends on the selection of tuning parameter $\lambda$. A popular approach is to use the cross-validation (CV), whose theoretical properties again are unclear in the high-dimensional setup. Here, we shall apply the Gaussian wild bootstrap to determine $\lambda$. Let $r \in [0,1)$ and 
$$
\tilde{\cal G}(r, M, \zeta_p) = \Big \{ \Theta \in \mathbb{S}_+^{p \times p} : |\Theta|_{L^1} \le M, \; \max_{m \le p} \sum_{k=1}^p |\theta_{mk}|^r \le \zeta_p \Big \},
$$
where $\mathbb{S}_+^{p \times p}$ is the collection of positive-definite $p \times p$ symmetric matrices.

\begin{thm}[Adaptive tuning parameter selection for CLIME: subgaussian observations]
\label{thm:clime_precision_mat_rate_adaptive}
Let $\nu \ge 1$ and $\vX_i$ be iid mean zero random vectors such that $X_{ik} \sim  \text{subgaussian}(\nu^2)$ for all $k=1,\cdots,p$. Suppose that there exist constants $C_i > 0, i=1,\cdots,4,$ such that $\{\Gamma_g\}_{(j,k),(j,k)} \ge C_1$, $\|X_{1k}\|_4 \le C_2$, $\|X_{1k}\|_6 \le C_3 \nu^{1/3}$ and $\|X_{1k}\|_8 \le C_4 \nu^{1/2}$ for all $j,k=1,\cdots,p$. Assume that $\Omega \in \tilde{\cal G}(r, M, \zeta_p)$ and and $\nu^4 \log^7(np) \le C_1 n^{1-K}$ for some $K \in (0,1)$. Choose $\lambda_* = M a_{\hat{L}_0^*}(1-\alpha)$, where the bootstrap samples are generated with the covariance matrix kernel in (\ref{eqn:sample-covmat-kernel}). Then, we have with probability at least $1-\alpha-C n^{-K/8}$ for some constant $C > 0$ depending only on $C_1,\cdots,C_5$ such that
\begin{eqnarray}
\label{eqn:spectral_rate_clime_precision_mat_rate_adaptive_wild_bootstrap}
\|\hat\Omega(\lambda_*) - \Omega\|_2 &\le& C_r \zeta_p M^{2-2r} a_{\hat{L}_0^*}(1-\alpha)^{1-r}, \\
\label{eqn:Frobenius_rate_clime_precision_mat_rate_adaptive_wild_bootstrap}
p^{-1} |\hat\Omega(\lambda_*) - \Omega|_F^2 &\le& 4 C_r \zeta_p M^{4-2r} a_{\hat{L}_0^*}(1-\alpha)^{2-r},
\end{eqnarray}
where $C_r = 2^{3-2r} (1+2^{1-r}+3^{1-r})$. In addition, we have $\E[a_{\hat{L}_0^*}(1-\alpha)] \le C' (\log(p)/n)^{1/2}$ for some constant $C' > 0$ depending only on $\alpha$ and $C_1,\cdots,C_5$. In particular, $\E[\lambda_*] \le C' M (\log(p)/n)^{1/2}$.
\end{thm}

Now, we compare Theorem \ref{thm:clime_precision_mat_rate_adaptive} with \cite[Theorem 1(a) and 4(a)]{cailiuluo2011a}. Let $\eta \in (0,1/4)$ and $K,\tau \in (0,\infty)$ be bounded constants. Assuming that $\log(p)/n \le \eta$, $\E[\exp(t X_{ij}^2)] \le K$ for all $|t| \le \eta$ and $i=1,\cdots,n; j=1,\cdots,p$, and $\Omega \in \tilde{\cal G}(r, M, \zeta_p)$, \cite{cailiuluo2011a} showed that with probability at least $1-4p^{-\tau}$
\begin{eqnarray*}
\|\hat\Omega(\lambda_\Delta) - \Omega\|_2 &\le& C'_\Delta \zeta_p M^{2-2r} (\log(p)/n)^{(1-r)/2}, \\
p^{-1} |\hat\Omega(\lambda_\Delta) - \Omega|_F^2 &\le& 4 C'_\Delta \zeta_p M^{4-2r} (\log(p)/n)^{1-r/2},
\end{eqnarray*}
where $\lambda_\Delta = C_\Delta M (\log(p)/n)^{1/2}$, $C_\Delta = 2 \eta^{-2} (2+\tau+\eta^{-1}e^2K^2)^2$, and $C'_\Delta = C_r C_\Delta^{1-r}$. If $X_{ij} \sim \text{subgaussian}(\nu^2)$, then $\eta \le \nu^{-2}$ for large enough $K$. Therefore, $C_\Delta \gtrsim \nu^8$ and $C'_\Delta \gtrsim \nu^{8(1-r)}$, both diverging to infinity as $\nu^2 \to \infty$ (i.e. $\eta \to 0$). Therefore, $\lambda_* = o_\Prob(\lambda_\Delta)$ and the convergence rates in (\ref{eqn:spectral_rate_clime_precision_mat_rate_adaptive_wild_bootstrap}) and (\ref{eqn:Frobenius_rate_clime_precision_mat_rate_adaptive_wild_bootstrap}) are much sharper than those obtained in \cite[Theorem 1(a) and 4(a)]{cailiuluo2011a}.

On the other hand, the turning parameter $\lambda_*$ requires the knowledge of $M$ and thus the estimator $\hat\Omega(\lambda_*)$ is not fully data-dependent. In contrast with the thresholded covariance matrix estimation problem in Section \ref{subsec:tuning_selection_thresholded_cov_mat}, the fundamental difficulty here for estimating the precision matrix is that there is no sample analog of $\Omega$ when $p > n$ and $M$ can be viewed as a stability parameter in the sparse inversion of the matrix $\hat{S}$. That is, the larger $M$, the more difficult to estimate $\Omega = \Sigma^{-1}$; in particular for CLIME, the rates (\ref{eqn:spectral_rate_clime_precision_mat_rate_adaptive_wild_bootstrap}) and (\ref{eqn:Frobenius_rate_clime_precision_mat_rate_adaptive_wild_bootstrap}) become slower. In addition, $|\Omega|_{L^1}$ plays a similar role in the graphical Lasso model for estimating the sparse precision matrix \cite{ravikumarwainwrightraskuttiyu2008a}. The same comments apply to the problem of estimating the sparse linear functionals in Section \ref{subsec:sparse_linear_functionals}.

\subsection{Estimation of the sparse linear functionals}
\label{subsec:sparse_linear_functionals}

Consider estimation of the linear functional $\mbf\theta = \Sigma^{-1} \vb$, where $\vb$ is a fixed known $p \times 1$ vector and $\Sigma = \Var(\vX_i)$. Functionals of such form are related to the solution of the linear equality constrained quadratic program
\begin{equation}
\label{eqn:lcqp}
\text{minimize}_{\vw \in \mathbb{R}^{p\times p}} \vw^\top \Sigma \vw \quad \text{subject to} \quad \vb^\top \vw = 1,
\end{equation}
which arises naturally in Markowitz portfolio selection, linear discriminant analysis, array signal processing, best linear unbiased estimator (BLUE), and optimal linear prediction for univariate time series \cite{markowitz1952,maizouyuan2012a,guerci1999a,mcmurrypolitis2015}. For example, in Markowitz portfolio selection, the portfolio risk $\Var(\vX_i^\top \vw)$ is minimized subject to the constraint that the expected mean return $\E(\vX_i^\top \vw)$ is fixed at certain level. The solution of (\ref{eqn:lcqp}) $\vw^*= (\vb^\top \Sigma^{-1} \vb)^{-1} \Sigma^{-1} \vb$ is proportional to $\mbf\theta$ and the optimal value of (\ref{eqn:lcqp}) is $(\vb^\top \Sigma^{-1} \vb)^{-1}$. A naive approach to estimate $\mbf\theta$ has two steps: first construct an invertible estimator $\hat\Sigma$ of $\Sigma$ and second estimate $\mbf\theta$ by $\hat\Sigma^{-1} \vb$. This two-step estimator may not be consistent for $\mbf\theta$ in high-dimensions even though $\hat\Sigma$ is a spectral norm consistent (and typically regularized) estimator of $\Sigma$ because in the worst case $|\hat{\mbf\theta}-\mbf\theta| = |\hat\Sigma^{-1} \vb - \Sigma^{-1} \vb| \le \|\hat\Sigma^{-1} - \Sigma^{-1}\|_2 \cdot |\vb|$ does not converge if $|\vb| \to \infty$ at faster rate than $\|\hat\Sigma^{-1} - \Sigma^{-1}\|_2 \to 0$. However, if $\mbf\theta$ has some structural assumptions such as sparsity, then we can directly estimate $\mbf\theta$ without the intermediate step for estimating $\hat\Sigma$ or $\hat\Sigma^{-1}$. Sparsity in $\mbf\theta$ is often a plausible assumption in real applications. For instance, the sparse portfolio has been considered in \cite{brodieetal2009} to obtain the stable portfolio optimization and to facilitate the transaction cost for a large number of assets. When $\mbf\theta$ is sparse, the following Dantzig-selector type problem has been proposed in \cite{chenxuwu2015+} to estimate $\mbf\theta$
\begin{equation}
\label{eqn:dantzig_linear_functional}
\hat{\mbf\theta}(\lambda) = \text{argmin}_{\vw \in \mathbb{R}^{p\times p}} |\vw|_1 \quad \text{subject to} \quad |\hat{S} \vw - \vb|_\infty \le \lambda,
\end{equation}
where $\lambda \ge 0$ is a tuning parameter to control the sparsity in $\hat{\mbf\theta}(\lambda)$. The optimization problem (\ref{eqn:dantzig_linear_functional}) can be solved by linear programming and thus there are computationally efficient algorithms for obtaining $\hat{\mbf\theta}(\lambda)$. The intuition of (\ref{eqn:dantzig_linear_functional}) is that since $\Sigma \mbf\theta = \vb$, we should expect that $\hat{S} \hat{\mbf\theta} \approx \vb$ for a reasonably good estimator $\hat{\mbf\theta}$. Under the sparsity assumption on $\mbf\theta$ and suitable moment conditions on $\vX_i$, \cite{chenxuwu2015+} obtained the rate of convergence for $\hat{\mbf\theta}(\lambda)$. However, a remaining issue for using (\ref{eqn:dantzig_linear_functional}) on real data is to properly select the tuning parameter $\lambda$. Different from the thresholded covariance matrix estimation where the sparsity is assumed in $\Sigma$, here we do not require this structure in the linear functional estimation. Instead, we impose the sparsity assumption directly on $\mbf\theta$. Let $r \in [0,1)$ and
$$
{\cal G}'(r, \zeta_p) = \Big \{\vw \in \mathbb{R}^p : \sum_{j=1}^p |w_j|^r \le \zeta_p \Big \}.
$$
Here, $\zeta_p$ controls the sparsity level of the elements in ${\cal G}'(r, \zeta_p)$. Without assuming any structure on $\Sigma$, we can allow stronger dependence in $\Sigma$ and therefore the Gaussian wild bootstrap approximation may perform better in this case.

\begin{thm}[Adaptive tuning parameter selection in linear functional estimation: subgaussian observations]
\label{thm:linear_functional_estimation_adaptive_wild_bootstrap}
Let $\nu \ge 1$ and $\vX_i$ be iid mean zero random vectors such that $X_{ik} \sim  \text{subgaussian}(\nu^2)$ for all $k=1,\cdots,p$. Suppose that there exist constants $C_i > 0, i=1,\cdots,4,$ such that $\{\Gamma_g\}_{(j,k),(j,k)} \ge C_1$, $\|X_{1k}\|_4 \le C_2$, $\|X_{1k}\|_6 \le C_3 \nu^{1/3}$ and $\|X_{1k}\|_8 \le C_4 \nu^{1/2}$ for all $j,k=1,\cdots,p$. Assume that $\mbf\theta \in {\cal G}'(r, \zeta_p)$ and $\nu^4 \log^7(np) \le C_5 n^{1-K}$ for some $K \in (0,1)$. Let $|\mbf\theta|_1 \le M$ and choose $\lambda_* = M a_{\hat{L}_0^*}(1-\alpha)$, where the bootstrap samples are generated with the covariance matrix kernel in (\ref{eqn:sample-covmat-kernel}). Then we have for all $w \in [1,\infty]$
\begin{equation}
\label{eqn:rate_linear_functional_estimation_adaptive_wild_bootstrap}
|\hat{\mbf\theta}(\lambda_*) - \mbf\theta|_w \le (2 \cdot 6^{1 \over w} \cdot 5^{1-r\over w}) \zeta_p^{1\over w} (M |\Sigma^{-1}|_{L^1})^{1-{r \over w}} a_{\hat{L}_0^*}(1-\alpha)^{1-{r \over w}}
\end{equation}
with probability at least $1-\alpha-C n^{-K/8}$ for some constant $C > 0$ depending only on $C_1,\cdots,C_5$. In addition, we have
\begin{equation}
\label{eqn:linfun_lambda_*-bound_subgaussian}
\E[\lambda_*] \le C' M (\log(p)/n)^{1/2},
\end{equation}
where $C' > 0$ is a constant depending only on $\alpha$ and $C_1,\cdots,C_5$.
\end{thm}

The tuning parameter $\lambda_\Delta = C_\Delta M \sqrt{\log(p)/n}$ is selected in \cite[Theorem II.1]{chenxuwu2015+}, which is non-adaptive and the constant $C_\Delta > 0$ depends on the underlying data distribution $F$ through $\nu^2$. In particular, $C_\Delta \to \infty$ as $\nu^2 \to \infty$. Theorem \ref{thm:linear_functional_estimation_adaptive_wild_bootstrap} shows that the bootstrap tuning parameter selection strategy is less conservative in view of (\ref{eqn:linfun_lambda_*-bound_subgaussian}) and the rate (\ref{eqn:rate_linear_functional_estimation_adaptive_wild_bootstrap}) can be much tighter than $\hat{\mbf\theta}(\lambda_\Delta)$ when $\nu^2 \to \infty$. However, as in Section \ref{subsec:sparse_prec_mat}, the turning parameter $\lambda_*$ here requires the knowledge of $M$ and thus the estimator $\hat{\mbf\theta}(\lambda_*)$ is not fully data-dependent. But this is due to the fundamental difficulty for the lack of the sample analog of $\mbf\theta$ in this problem.

\subsection{Proof of Theorem \ref{thm:clime_precision_mat_rate_adaptive} and \ref{thm:linear_functional_estimation_adaptive_wild_bootstrap}}

\begin{lem}
\label{lem:precmat_general_bound}
Suppose that $\Omega \in \tilde{\cal G}(r, M, \zeta_p)$ and let $\lambda \ge |\Omega|_{L^1} \|\hat{S}-\Sigma\|$. Then, we have
\begin{eqnarray*}
|\hat\Omega(\lambda) - \Omega|_\infty &\le& 4 |\Omega|_{L^1} \lambda, \\
\|\hat\Omega(\lambda) - \Omega\|_2 &\le& C_1 \zeta_p \lambda^{1-r}, \\
p^{-1} |\hat\Omega(\lambda) - \Omega|_F^2 &\le& C_2 \zeta_p \lambda^{2-r},
\end{eqnarray*}
where $C_1 \le 2 (1+2^{1-r}+3^{1-r}) (4 |\Omega|_{L^1})^{1-r}$ and $C_2 \le 4 |\Omega|_{L^1} C_1$.
\end{lem}

\begin{proof}[Proof of Lemma \ref{lem:precmat_general_bound}]
See \cite[Theorem 6]{cailiuluo2011a}.
\end{proof}

\begin{lem}
\label{lem:linfun_general_bound}
Let $\lambda \ge |\mbf\theta|_1 \|\hat{S} - \Sigma\|$. Then, $\mbf\theta$ satisfies $|\hat{S} \mbf\theta - \vb|_\infty \le \lambda$. For the Dantzig-selector estimator $\hat{\mbf\theta}(\lambda)$ in (\ref{eqn:dantzig_linear_functional}), we have
$$
|\hat{\mbf\theta}(\lambda) - \mbf\theta|_w \le [6 D(5\lambda|\Sigma^{-1}|_{L^1})]^{1 \over w} (2\lambda|\Sigma^{-1}|_{L^1})^{1-{1 \over w}},
$$
where $D(u) = \sum_{j=1}^p (|\theta_j| \wedge u), u \ge 0$, is the smallness measure of $\mbf\theta$.
\end{lem}

\begin{proof}[Proof of Lemma \ref{lem:linfun_general_bound}]
See \cite[Lemma V.6]{chenxuwu2015+}.
\end{proof}

\begin{proof}[Proof of Theorem \ref{thm:clime_precision_mat_rate_adaptive}]
Let $\lambda_\diamond = |\Omega|_{L^1} \|\hat{S}-\Sigma\|$. By the subgaussian assumption and Lemma \ref{lem:moment-bounds-gaussian-obs}, we have (\ref{eqn:check_GA1}) for some large enough constant $C > 0$ depending only on $C_2,C_3,C_4$ so that (\ref{eqn:check_GA1}) holds. Since $\{\Gamma_g\}_{(j,k),(j,k)} \ge C_1$ for all $j,k=1,\cdots,p$ and $\nu^4 \log^7(np) \le C_5 n^{1-K}$, by Theorem \ref{thm:gaussian_wild_bootstrap_validity}, we have $ \|\hat{S}-\Sigma\| \le a_{\hat{L}_0^*}(1-\alpha)$ with probability at least $1-\alpha-Cn^{-K/8}$, where $C > 0$ is a constant depending only on $C_i,i=1,\cdots,5$. Since $|\Omega|_{L^1} \le M$ for $\Omega \in \tilde{\cal G}(r, M, \zeta_p)$, $\Prob(\lambda_\diamond \le \lambda_*) \ge 1- \alpha -Cn^{-K/8}$. Then, (\ref{eqn:spectral_rate_clime_precision_mat_rate_adaptive_wild_bootstrap}) and (\ref{eqn:Frobenius_rate_clime_precision_mat_rate_adaptive_wild_bootstrap}) follow from Lemma \ref{lem:precmat_general_bound} applied to the event $\{\lambda_\diamond \le \lambda_*\}$. The bounds for $\E[a_{\hat{L}_0^*}(1-\alpha)]$ and $\E[\lambda_*]$ are the same as those in Theorem \ref{thm:thresholded_cov_mat_rate_adaptive}.
\end{proof}

\begin{proof}[Proof of Theorem \ref{thm:linear_functional_estimation_adaptive_wild_bootstrap}]
For $\mbf\theta \in {\cal G}'(r, C_0, \zeta_p)$, we have $D(u) \le 2 u^{1-r} \zeta_p$. Let $\lambda_\diamond = |\mbf\theta|_1 \|\hat{S}-\Sigma\|$. Following the proof of Theorem \ref{thm:clime_precision_mat_rate_adaptive}, we have $\Prob(\lambda_\diamond \le \lambda_*) \ge 1- \alpha -Cn^{-K/8}$.  By Lemma \ref{lem:linfun_general_bound}, we have with probability at least $1- \alpha -Cn^{-K/8}$
\begin{eqnarray*}
|\hat{\mbf\theta}(\lambda_*)-\mbf\theta|_w &\le& [6 D(5 \lambda_* |\Sigma^{-1}|_{L^1}) ]^{1 \over w} (2 \lambda_* |\Sigma^{-1}|_{L^1})^{1-{1 \over w}} \\
&\le& (2 \cdot 6^{1 \over w} \cdot 5^{1-r\over w}) \zeta_p^{1\over w} |\Sigma^{-1}|_{L^1}^{1-{r \over w}} \lambda_*^{1-{r \over w}},
\end{eqnarray*}
which is (\ref{eqn:rate_linear_functional_estimation_adaptive_wild_bootstrap}). The bound for $\E[\lambda_*]$ is immediate.
\end{proof}

\section{Higher-order moment inequalities of the decoupled and canonical V-statistics}
\label{sec:concentration_inequalities}

Let ${\vX'}_1^n$ be an independent copy of $\vX_1^n$ following the distribution $F$. Here, we present some higher-order moment and concentration inequalities for 
$$V' = \sum_{1 \le i , j \le n} f(\vX_i, \vX'_j),$$
which is the unnormalized version of the decoupled V-statistics. Formally, U-statistics is asymptotically equivalent to V-statistics by removing the diagonal sum $\sum_{i=1}^n f(\vX_i,\vX'_i)$. So here we also refer $V'$ as the decoupled U-statistics with the kernel $f$.
We consider the canonical kernel $f$. Tail probability inequalities for $\|V'\|$ is closely related to the moment bounds of higher-orders. Let $B$ be a separable Banach space, $(B^*, \|\cdot\|_*)$ the dual space of $(B,\|\cdot\|)$, and $B^*_1$ the unit ball in $B^*$. Let $I_n = \{1,\cdots,n\}$ and $\iota = (i, j) \in I_n^2$ be the collection of index pairs implicitly running over $i,j \in I_n$. For notation simplicity, we write $\fiota = f(\vX_i, \vX'_j)$ and $V' = \sum_\iota \fiota$. For $I \subset I_2$, we let $\iota_I$ be the $|I|$-dimensional vector such that it is the restriction of $\iota$ to the coordinates in $I$ (e.g. if $\iota = (2,5)$, then $\iota_{\{1\}} = 2$). For $J \in I_2$, we denote $\E_J$ as the expectation taken w.r.t. the random variables $X_i^{(j)}$ for all $j \in J$ and $i \in I_n$. By convention, $\E_\emptyset X = X$. Following \cite{adamczak2006}, for $J \subset I \subset I_2$ and $I \neq \emptyset$, we define
\begin{eqnarray}
\nonumber
\vertiii{\fiota}_{I,J} &=& \E_{I\setminus J} \sup\Big\{ \E_J \sum_{\iota_I} \phi(\fiota) \prod_{j \in J} g_{\iota_j}^{(j)}(\vX_{\iota_j}^{(j)}) : \phi \in B^*_1, \\
&& \quad  g^{(j)}_i : \mathbb{R}^p \to \mathbb{R}, j \in J, i \in I_n, \text{ and } \E \sum_i |g^{(j)}_i(\vX^{(j)}_{i})|^2 \le 1 \Big\}. \label{eqn:kernel-norms}
\end{eqnarray}
If $I = \emptyset$, then by convention $\vertiii{\fiota}_{\emptyset,\emptyset} = \|\fiota\|$. As remarked by \cite{adamczak2006}, $\vertiii{f_{(\iota)}}_{I_2, J}$ is a deterministic quantity and it is in fact a norm. For $I \neq I_2$, $\vertiii{f_{(\iota)}}_{I, J}$ is random variable only depending on $\{\vX^{(k)}_{\iota_k}\}_{k \in I^c}$ where $\vX^{(1)}_i = \vX_i$ and $\vX^{(2)}_j = \vX'_j$.

\begin{lem}[Higher-order moment inequality for unbounded canonical kernel]
\label{lem:exp-ineq-bounded-kernels}
Let $h$ be a canonical kernel of order two w.r.t. $F$. Then, there exists an absolute constant $K > 0$ such that we have for all $q \ge 2$
\begin{eqnarray}
\nonumber
&& \E \|\sum_\iota \fiota\|^q \le K^q \Big\{ (\E\|\sum_\iota \fiota\|)^q  \\ \nonumber
&& + q^{q/2} \Big[ \vertiii{\fiota}_{I_2, \{1\}}^q + \vertiii{\fiota}_{I_2, \{2\}}^q + [ \E_2 \sum_j (\E_1 \|\sum_i \fiota\|)^2 ]^{q/2} + [ \E_1 \sum_i (\E_2 \| \sum_j \fiota\|)^2 ]^{q/2} \\ \nonumber
&&  + \Big( \sum_\iota \E \|\fiota\|^2 \Big)^{q/2} + \E_2 \max_j \Big( \sum_i \E_1 \|\fiota\|^2\Big)^{q/2} + \E_1 \max_i \Big( \sum_j \E_2 \|\fiota\|^2\Big)^{q/2} \Big] \\\nonumber
&&  + q^q \Big[ \vertiii{\fiota}_{I_2, I_2}^q + \E_2 \max_j \vertiii{\fiota}_{\{1\}, \emptyset}^q + \E_1 \max_i \vertiii{\fiota}_{\{2\}, \emptyset}^q \Big]  \\\nonumber
&&  + q^{3q/2} \Big[ \E_2 \max_j \vertiii{\fiota}_{\{1\}, \{1\}}^q + \E_1 \max_i \vertiii{\fiota}_{\{2\}, \{2\}}^q \Big] + q^{2q} \E \max_\iota \|\fiota\|^q \Big\}. \\ \label{eqn:moment-bound-canonical-kernels-order2}
\end{eqnarray}
\end{lem}

For bounded kernels, (\ref{eqn:moment-bound-canonical-kernels-order2}) leads to an exponential concentration inequality; c.f. Corollary \ref{cor:exp-ineq-bounded-kernels} below. Lemma \ref{lem:exp-ineq-bounded-kernels} can be viewed as the matrix-variate version of the moments and exponential inequalities for the real-valued U-statistics in \cite{ginelatalazinn2000}. On the other hand, it can also be viewed as an extension of the upper tail part of Talagrand's inequality for maxima of empirical processes of iid random variables to U-statistics \cite{talagrand1996,massart2000}. Indeed, Lemma \ref{lem:exp-ineq-bounded-kernels} recovers \cite[Theorem 3.2]{ginelatalazinn2000} for real-valued kernels. For $f$ taking values in $\mathbb{R}$, it was established in \cite[Theorem 3.2]{ginelatalazinn2000} that there exists an absolute constant $K > 0$ such that for all $q \ge 2$
\begin{eqnarray}
\nonumber
&&  \E|\sum_\iota \fiota|^q \le K^q \Big\{ q^{q/2}  (\sum_\iota \E \fiota^2)^{q/2} + q^q \|\fiota\|_{L^2\to L^2}^q   \\\nonumber
&&  + q^{3q/2} \Big[  \E_1 \max_i (\E_2 \sum_j \fiota^2)^{q/2} + \E_2 \max_j (\E_1 \sum_i \fiota^2)^{q/2}  \Big] + q^{2q} \E \max_\iota |\fiota|^q \Big\}, \\ \label{eqn:glz2000-thm3.2}
\end{eqnarray}
where
\begin{eqnarray*}
\|\fiota\|_{L^2\to L^2} = \sup\Big\{ \E \sum_\iota \fiota g^{(1)}_i(\vX^{(1)}_i) g^{(2)}_j(\vX^{(2)}_j): \E \sum_i |g^{(1)}_i(\vX^{(1)}_i)|^2 \le 1, \E \sum_j |g^{(2)}_j(\vX^{(2)}_j)|^2 \le 1 \Big\}.
\end{eqnarray*}
To compare (\ref{eqn:glz2000-thm3.2}) and (\ref{eqn:moment-bound-canonical-kernels-order2}), we first compute the terms with coefficient $q^{3q/2}$ in (\ref{eqn:moment-bound-canonical-kernels-order2}). By (\ref{eqn:alternative-expression-h11}), we have
\begin{eqnarray*}
\E_2 \max_j \vertiii{\fiota}_{\{1\}, \{1\}}^q  \le \E_2 \max_j \Big( \E_1 \sum_i \fiota^2 \Big)^{q/2},
\end{eqnarray*}
which is handled by the $q^{3q/2}$ term in (\ref{eqn:glz2000-thm3.2}). Similar bound holds for $\E_1 \max_i \vertiii{\fiota}_{\{2\}, \{2\}}^q$. Next, we consider the $q^q$ terms. Observe that $\vertiii{\fiota}_{I_2, I_2} = \|\fiota\|_{L^2\to L^2}$. In addition, by (\ref{eqn:alternative-expression-h1emptyset}) and Jensen's inequality, we have
\begin{eqnarray*}
\E_2 \max_j \vertiii{\fiota}_{\{1\}, \emptyset}^q = \E_2 \max_j \Big( \E_1 |\sum_i \fiota| \Big)^q \le \E_2 \max_j \Big( \E_1 \sum_i \fiota^2 \Big)^{q/2},
\end{eqnarray*}
which is again smaller than $q^{3q/2}$ terms in (\ref{eqn:glz2000-thm3.2}). For the $q^{q/2}$ terms, by (\ref{eqn:alternative-expression-hI_21}), the degeneracy of $h$, Jensen's inequality, and the orthogonality of $\fiota$ conditional on $\vX_1^n$, we have
\begin{equation}
\label{eqn:not_type2_space}
\max\Big\{ \vertiii{\fiota}_{I_2, \{1\}}^2, \; \E_1 \sum_i (\E_2 \| \sum_j \fiota \|)^2 \Big\} \le \E_1 \sum_i \E_2 ( \sum_j \fiota )^2 = \sum_\iota \E \fiota^2.
\end{equation}

\begin{rmk}
On the contrary, unlike the real-valued kernels where the higher-order moment bounds and exponential inequalities involving only the $L^2$ and $L^\infty$ norms, the matrix-valued U-statistics involves more subtle balance among the ``mixed norms" of $h$ in (\ref{eqn:moment-bound-canonical-kernels-order2}). Those quantities are expressed in terms of the maxima of empirical processes and it thus can be much smaller than a straightforward extension of (\ref{eqn:glz2000-thm3.2}). This remark also applies to Corollary \ref{cor:exp-ineq-bounded-kernels}.
\qed
\end{rmk}

\subsection{Exponential inequality for bounded kernels}

As an immediate consequence of Lemma \ref{lem:exp-ineq-bounded-kernels}, we have the following concentration inequality, which is equivalent to (\ref{eqn:moment-bound-canonical-kernels-order2}) up to constants.

\begin{cor}[Exponential inequality for bounded canonical kernel]
\label{cor:exp-ineq-bounded-kernels}
Let $f$ be a bounded canonical kernel of order two w.r.t. $F$. Then, there exist absolute constants $K, M, M' > 0$ such that for all $t>0$
\begin{equation}
\label{eqn:exp-ineq-bounded-kernels}
\Prob(\|\sum_\iota \fiota\| \ge K \E\|\sum_\iota \fiota\| + t) \le  M \exp\left\{ -M' \min\left[ \left({t \over D}\right)^2, \left({t \over C}\right), \left({t \over B}\right)^{2\over3}, \left({t \over A}\right)^{1\over2} \right] \right\},
\end{equation}
where
\begin{eqnarray*}
A &=&  \|\fiota\|_\infty, \qquad B = \left\| \vertiii{\fiota}_{\{1\},\{1\}}   \right\|_\infty + \left\| \vertiii{\fiota}_{\{2\},\{2\}}  \right\|_\infty, \\
C &=& \|\fiota\|_{I_2,I_2} + \left\| \vertiii{\fiota}_{\{1\},\emptyset}  \right\|_\infty + \left\| \vertiii{\fiota}_{\{2\},\emptyset}  \right\|_\infty, \\
D &=& n^{1\over2} \left[ n \E \|\fiota\|^2 + \left\| \E_1 \|\fiota\|^2 \right\|_\infty + \left\| \E_2 \|\fiota\|^2 \right\|_\infty \right]^{1\over2} + \vertiii{\fiota}_{I_2, \{1\}} \\
&&  + \vertiii{\fiota}_{I_2, \{2\}} + [ \E_1 \sum_i (\E_2 \| \sum_j \fiota\|)^2 ]^{1\over2} + [ \E_2 \sum_j (\E_1 \| \sum_i \fiota\|)^2 ]^{1\over2}.
\end{eqnarray*}
\end{cor}

Proof of Corollary \ref{cor:exp-ineq-bounded-kernels} is based on a standard argument combining the Chebyshev and higher-order moment inequalities (Theorem \ref{lem:exp-ineq-bounded-kernels}); see e.g. \cite{ginelatalazinn2000}. Corollary \ref{cor:exp-ineq-bounded-kernels} is a Bernstein-type inequality for the high-dimensional matrix-valued U-statistics. Exact computations of $A, B, C, D$ are quite complicated. Below, we shall give a less sharp (with uniform bounds) but user-friendly version of Corollary \ref{cor:exp-ineq-bounded-kernels}.

\begin{cor}
\label{cor:exp-ineq-bounded-kernels-simplified}
If $f$ be a bounded canonical kernel of order two w.r.t. $F$ and $\|f\|_\infty \le C_0$, then there exist constants $K, M > 0$ and a constant $M' > 0$ depending only $C_0$ such that for all $t>0$
\begin{equation}
\label{eqn:exp-ineq-bounded-kernels-simplified}
\Prob(\|\sum_\iota \fiota\| \ge K \E\|\sum_\iota \fiota\| + t) \le  M \exp\left[ -M' \min\left( {t^2 \over n^3}, {t \over n}, {t^{2/3} \over n^{1/3}}, t^{1/2} \right) \right].
\end{equation}
\end{cor}


\begin{rmk}
\label{rmk:decoupled=undecoupled}
\begin{enumerate}
\item Except for a worse factor on $t^2/n^3$, the exponential inequality (\ref{eqn:exp-ineq-bounded-kernels-simplified}) is the same as \cite[Corollary 3.4]{ginelatalazinn2000}, which considered the real-valued kernels and the corresponding term is $t^2/n^2$. The loss of order $n$ is due to the lack of orthogonality for the sup-norm. Since the Banach space $(\mathbb{R}^{p\times p}, \|\cdot\|)$ is of type 1 and nothing more \cite[Chapter 9.2]{ledouxtalagrand1991} , this term in (\ref{eqn:exp-ineq-bounded-kernels-simplified}) cannot be improved.

\item In view of the undecoupled results of \cite{delaPenaGine1999,delapenamontgomerysmith1995}, inequalities (\ref{eqn:moment-bound-canonical-kernels-order2}), (\ref{eqn:exp-ineq-bounded-kernels}) and (\ref{eqn:exp-ineq-bounded-kernels-simplified}) also hold for $\sum_{1\le i,j \le n} f(\vX_i, \vX_j)$.

\end{enumerate}
\qed
\end{rmk}

\subsection{Proof of results in Section \ref{sec:concentration_inequalities}}

\begin{proof}[Proof of Lemma \ref{lem:exp-ineq-bounded-kernels}]
Note that $\vertiii{\fiota}_{I_2, \emptyset} = \E\|\sum_\iota \fiota\|$. By \cite[Theorem 1]{adamczak2006} with order two, there exists an absolute constant $K_1 > 0$ such that for all $q \ge 2$, we have
\begin{eqnarray}
\nonumber
\E \|\sum_\iota \fiota\|^q &\le& K_1^q \Big\{ (\E\|\sum_\iota \fiota\|)^q + q^{q/2} \Big[ \vertiii{\fiota}_{I_2, \{1\}}^q + \vertiii{\fiota}_{I_2, \{2\}}^q \Big] \\ \nonumber
&& + q^q \Big[ \vertiii{\fiota}_{I_2, I_2}^q + \sum_j \E_2 \vertiii{\fiota}_{\{1\}, \emptyset}^q + \sum_i \E_1 \vertiii{\fiota}_{\{2\}, \emptyset}^q \Big]  \\ \nonumber
&& + q^{3q/2} \Big[ \sum_j \E_2 \vertiii{\fiota}_{\{1\}, \{1\}}^q + \sum_i \E_1 \vertiii{\fiota}_{\{2\}, \{2\}}^q \Big] \\ \label{eqn:exp-ineq-grand-decomp}
&& + q^{2q} \E \sum_\iota \|\fiota\|^q \Big\}.
\end{eqnarray}
First, we replace the summation in the last term on the RHS of (\ref{eqn:exp-ineq-grand-decomp}) by the maximum over $\iota$, with proper modifications of the argument from \cite{ginelatalazinn2000} in the Banach space setting. Applying Lemma \ref{lem:replace-sum-by-max} with $s = q/2$ and $\alpha = 4$ conditional on $\vX_1^n$, we have
\begin{eqnarray}
\label{eqn:exp-ineq-last-term}
& & \qquad q^{2q} \sum_\iota \E \|\fiota\|^q \\ \nonumber
& \le & 2^{2q+1} \Big[1+\Big({q\over2}\Big)^4\Big] \sum_i \E_1 \Big\{ \Big({q\over2}\Big)^{2q} \E_2 \max_j \|\fiota\|^q + (\sum_j \E_2 \|\fiota\|^2)^{q\over2} \Big\}.
\end{eqnarray}
By a second application of Lemma \ref{lem:replace-sum-by-max} for $\E_1$ with $s = q/2$ and $\alpha = 0$, we can bound the second term on the RHS of (\ref{eqn:exp-ineq-last-term}) by
\begin{equation*}
\sum_i \E_1 (\sum_j \E_2 \|\fiota\|^2)^{q\over2} \le 2 \Big[ \E_1 \max_i (\sum_j \E_2 \|\fiota\|^2)^{q\over2} + (\sum_\iota \E \|\fiota\|^2)^{q\over2} \Big].
\end{equation*}
For the first term on the RHS of (\ref{eqn:exp-ineq-last-term}), by Fubini's theorem to interchange the order of $\E_1$ and $\E_2$, we proceed as
\begin{eqnarray*}
&& \qquad \Big({q\over2}\Big)^{2q} \E_2 \sum_i \E_1 \max_j \|\fiota\|^q \\
&\le_{(*)}& 2 \Big[1+\Big({q\over2}\Big)^4\Big] \E_2 \Big\{ \Big({q\over2}\Big)^{2q} \E_1 \max_{\iota} \|\fiota\|^q + (\sum_i \E_1 \max_j \|\fiota\|^2)^{q\over2}  \Big\} \\
&\le& 2 \Big[1+\Big({q\over2}\Big)^4\Big] \Big\{ \Big({q\over2}\Big)^{2q} \E \max_{\iota} \|\fiota\|^q + \E_2 \Big[ \sum_j \Big(\sum_i \E_1 \|\fiota\|^2\Big) \Big]^{q\over2}  \Big\},
\end{eqnarray*}
where in step $(*)$ a third application of Lemma \ref{lem:replace-sum-by-max} is used with $s = q/2$ and $\alpha = 4$. By Lemma \ref{lem:rosenthal-latala-nonnegative-rv} with $s = q/2$ and by Lemma \ref{lem:replace-sum-by-max} with $s=q/2$ and $\alpha=0$ both on $\E_2$, we have
\begin{eqnarray*}
&& \qquad \E_2 \Big[ \sum_j \Big(\sum_i \E_1 \|\fiota\|^2\Big) \Big]^{q\over2} \\
&\le& K_2^q \Big[  \Big({q\over2}\Big)^{q\over2} \sum_j \E_2 \Big(\sum_i \E_1 \|\fiota\|^2\Big)^{q\over2} +  \Big( \sum_\iota \E \|\fiota\|^2 \Big)^{q\over2}   \Big] \\
&\le& K_3^q \Big({q\over2}\Big)^{q\over2} \Big[ \E_2 \max_j \Big( \sum_i \E_1 \|\fiota\|^2\Big)^{q/2}  +  \Big( \sum_\iota \E \|\fiota\|^2 \Big)^{q\over2}   \Big].
\end{eqnarray*}
Now, substituting those estimates into (\ref{eqn:exp-ineq-last-term}), it follows that
\begin{eqnarray}
\label{eqn:exp-ineq-last-term-bound}
&& \qquad q^{2q} \sum_{i<j} \E \|\hiota\|^q \le K_4^q \Big\{ q^{2q} \E \max_\iota \|\fiota\|^q + q^{q\over2} (\sum_\iota \E \|\fiota\|^2)^{q\over2} \\ \nonumber
&& \quad + q^{q/2} \Big[ \E_2 \max_j \Big( \sum_i \E_1 \|\fiota\|^2\Big)^{q/2} + \E_1 \max_i \Big( \sum_j \E_2 \|\fiota\|^2\Big)^{q/2} \Big] \Big\}.
\end{eqnarray}
Next, by Lemma \ref{eqn:replace_h11_sum_with_max}, we can replace the $q^{3q/2} \sum_j \E_2 \vertiii{\fiota}_{\{1\}, \{1\}}^q$ in (\ref{eqn:exp-ineq-grand-decomp}) by $q^{3q/2} \E_2 \max_j \vertiii{\fiota}_{\{1\}, \{1\}}^q$, with an additional term $(\sum_\iota \E \|\fiota\|^2)^{q/2}$ that have already been subsumed in the previous steps. Similarly, by Lemma \ref{eqn:replace_h1emptyset_sum_with_max}, we may replace $q^q \sum_j \E_2 \vertiii{\fiota}_{\{1\},\emptyset}^q$ in (\ref{eqn:exp-ineq-grand-decomp}) by $q^q \E_2 \max_j \vertiii{\fiota}_{\{1\},\emptyset}^q + [ \E_2 \sum_j (\E_1 \|\sum_i \fiota\|)^2 ]^{q/2}$ and then move the second term to the $q^{q/2}$ terms. Now, (\ref{eqn:moment-bound-canonical-kernels-order2}) follows from (\ref{eqn:exp-ineq-grand-decomp}) and (\ref{eqn:exp-ineq-last-term-bound}).
\end{proof}

\begin{proof}[Proof of Corollary \ref{cor:exp-ineq-bounded-kernels-simplified}]
By (\ref{eqn:alternative-expression-h11}) and (\ref{eqn:alternative-expression-h1emptyset}), elementary calculations show that $A = O(1)$, $B = O(n^{1/2})$, $C = O(n)$ for bounded kernels. For $D$, it is easy to see that
\begin{eqnarray}
\label{eqn:alternative-expression-hI_21}
&& \vertiii{\fiota}_{I_2, \{1\}} = \E_2 \sup_{\phi \in B^*_1} \Big[ \E_1 \sum_i (\sum_j \phi(\fiota) )^2 \Big]^{1/2} \\ \nonumber
&\le& \E_2 \Big[ \E_1 \sum_i \| \sum_j \fiota \|^2 \Big]^{1/2} \le \Big[ \E_1 \sum_i \E_2 \| \sum_j \fiota \|^2 \Big]^{1/2}.
\end{eqnarray}
Therefore, $D = O(n^{3/2})$ and (\ref{eqn:exp-ineq-bounded-kernels-simplified}) follows from (\ref{eqn:exp-ineq-bounded-kernels}).
\end{proof}

\begin{lem}
\label{lem:replace-sum-by-max}
Let $s \ge 1$ and $\alpha \ge 0$. Let $\xi_i$ be nonnegative independent random variables. Then, we have
\begin{equation}
\label{eqn:replace-sum-by-max}
s^{\alpha s} \sum_i \E \xi_i^s \le 2 (1+s^\alpha) \max\big\{ s^{\alpha s} \E \max_i \xi_i^s,  (\sum_i \E \xi_i)^s \big\}.
\end{equation}
\end{lem}
\begin{proof}[Proof of Lemma \ref{lem:replace-sum-by-max}]
See equation (2.6) in \cite{ginelatalazinn2000}.
\end{proof}

\begin{lem}
\label{lem:rosenthal-latala-nonnegative-rv}
Let $s \ge1 $ and $\xi_i$ be nonnegative independent random variables. Then, we have
\begin{equation}
\label{eqn:rosenthal-latala-nonnegative-rv}
 \E (\sum_i \xi_i)^s \le (2e)^s \max\big\{ e s^{s-1} \sum_i \E \xi_i^s,  e^s (\sum_i \E \xi_i)^s \big\}.
\end{equation}
\end{lem}

\begin{proof}[Proof of Lemma \ref{lem:rosenthal-latala-nonnegative-rv}]
See \cite{latala1997}.
\end{proof}

The following Lemma \ref{eqn:replace_h11_sum_with_max} and \ref{eqn:replace_h1emptyset_sum_with_max} hold for both non-degenerate and canonical V-statistics with the kernel $f$.

\begin{lem}
\label{eqn:replace_h11_sum_with_max}
Let $q \ge 2$ and $\vertiii{\fiota}_{\{1\}, \{1\}}$ be defined in (\ref{eqn:kernel-norms}). Then, there exists an absolute constant $K > 0$ such that
\begin{equation*}
q^{3q/2} \sum_j \E_2 \vertiii{\fiota}_{\{1\}, \{1\}}^q \le K^q \Big[ q^{3q/2} \E_2 \max_j \vertiii{\fiota}_{\{1\}, \{1\}}^q + (\sum_\iota \E \|\fiota\|^2)^{q/2} \Big].
\end{equation*}
\end{lem}

\begin{proof}[Proof of Lemma \ref{eqn:replace_h11_sum_with_max}]
By duality, we have
\begin{eqnarray}
\nonumber
\vertiii{\fiota}_{\{1\}, \{1\}}  &=& \sup\Big\{ \E_1 \sum_i \phi(\fiota) g^{(1)}_i(\vX^{(1)}_i) : \phi \in B^*_1, \E \sum_i |g^{(1)}_i(\vX^{(1)}_i)|^2 \le 1 \Big\} \\ \label{eqn:alternative-expression-h11}
&=& \sup\Big\{ \Big( \E_1 \sum_i \phi^2(\fiota) \Big)^{1/2} : \phi \in B^*_1 \Big\} .
\end{eqnarray}
Therefore, $\{\vertiii{\fiota}_{\{1\}, \{1\}}^2\}_{j=1}^n$ is a sequence of nonnegative independent random variables. By Lemma \ref{lem:replace-sum-by-max} with $\alpha=3$ and $s=q/2$, we have
\begin{eqnarray*}
&& q^{3q/2} \sum_j \E_2 \vertiii{\fiota}_{\{1\}, \{1\}}^q \\
&\le& 2^{1+3q/2} \Big[1+\Big({q\over2}\Big)^3\Big] \Big[ \Big({q\over2}\Big)^{3q/2} \E_2 \max_j \vertiii{\fiota}_{\{1\}, \{1\}}^q + (\sum_j \E_2 \vertiii{\fiota}_{\{1\}, \{1\}}^2)^{q/2} \Big].
\end{eqnarray*}
By (\ref{eqn:alternative-expression-h11}), $\vertiii{\fiota}_{\{1\}, \{1\}}^2 \le \sum_i \E_1 \|\fiota\|^2$ and the lemma follows.
\end{proof}

\begin{lem}
\label{eqn:replace_h1emptyset_sum_with_max}
Let $q \ge 2$ and $\vertiii{\fiota}_{\{1\}, \emptyset}$ be defined in (\ref{eqn:kernel-norms}). Then, there exists an absolute constant $K > 0$ such that
\begin{equation*}
q^q \sum_j \E_2 \vertiii{\fiota}_{\{1\}, \emptyset}^q \le K^q \Big[ q^q \E_2 \max_j \vertiii{\fiota}_{\{1\}, \emptyset}^q + \Big( \E_2 \sum_j (\E_1 \|\sum_i \fiota\|)^2 \Big)^{q/2} \Big].
\end{equation*}
\end{lem}

\begin{proof}[Proof of Lemma \ref{eqn:replace_h1emptyset_sum_with_max}]
As in the proof of Lemma \ref{eqn:replace_h11_sum_with_max}, by duality we have
\begin{equation}
\label{eqn:alternative-expression-h1emptyset}
\vertiii{\fiota}_{\{1\}, \emptyset} =  \E_1 \sup\Big\{ \sum_i \phi(\fiota) : \phi \in B^*_1 \Big\}  = \E_1 \|\sum_i \fiota\|.
\end{equation}
By Lemma \ref{lem:replace-sum-by-max} with $\alpha=2$ and $s=q/2$, we have
\begin{equation*}
q^q \sum_j \E_2 \vertiii{\fiota}_{\{1\}, \emptyset}^q \le K^q \Big[ q^q \E_2 \max_j \vertiii{\fiota}_{\{1\}, \emptyset}^q + (\sum_j \E_2 \vertiii{\fiota}_{\{1\}, \emptyset}^2)^{q/2} \Big].
\end{equation*}
\end{proof}

\end{document}